\documentclass[12pt]{amsart}
\usepackage{amsmath}
\usepackage{amssymb}
\usepackage{xcolor}
\usepackage{tikz}
\usetikzlibrary{cd}
\usetikzlibrary{knots}
\usetikzlibrary{calc}

\tikzset{string/.style={ultra thick}}
\tikzset{smallstring/.style={thick,scale=0.75,every node/.style={transform shape}}}
\tikzset{
    triple/.style args={[#1] in [#2] in [#3]}{
        #1,preaction={preaction={draw,#3},draw,#2}
    }
}
\tikzset{
    quadruple/.style args={[#1] in [#2] in [#3] in [#4]}{
        #1,preaction={preaction={preaction={draw,#4},draw,#3}, draw,#2}
    }
} 
\tikzset{
	super thick/.style={line width=3pt},
	more thick/.style={line width=1pt},
}

\usepackage{graphicx}
\usepackage{fullpage}
\usepackage[T1]{fontenc}
\usepackage[final]{microtype}
\usepackage{ifthen}
\usepackage{libertine}
\usepackage[libertine]{newtxmath}
\usepackage[all]{xy}

\usepackage{xcolor}
\definecolor{dark-red}{rgb}{0.7,0.25,0.25}
\definecolor{dark-blue}{rgb}{0.15,0.15,0.55}
\definecolor{medium-blue}{rgb}{0,0,0.65}
\definecolor{DarkGreen}{RGB}{0,150,0}

\usepackage[pdftex,plainpages=false,hypertexnames=false,pdfpagelabels,breaklinks]{hyperref}
\hypersetup{
   colorlinks, linkcolor={purple},
   citecolor={medium-blue}, urlcolor={medium-blue}
}
\usepackage[backend=biber,style=alphabetic,doi=false,isbn=false,url=false,minnames=6,maxnames=6]{biblatex}
\renewbibmacro{in:}{%
  \ifentrytype{article}{}{\printtext{\bibstring{in}\intitlepunct}}}
\renewbibmacro*{volume+number+eid}{%
  \printtext{vol.}
  \printfield{volume}%
  \setunit*{\addnbspace}
  \printfield{number}%
  \setunit{\addcomma\space}%
  \printfield{eid}}
\DeclareFieldFormat[article]{number}{\mkbibparens{#1}}
\usepackage{silence}
\newcommand{\doi}[1]{\href{https://dx.doi.org/#1}{{\small DOI:#1}}}

\newcommand{\googlebooks}[1]{(preview at \href{https://books.google.com/books?id=#1}{google books})}

\newcommand{\numdam}[1]{}

\theoremstyle{plain}
\newtheorem{prop}{Proposition}[section]

\newtheorem{thm}[prop]{Theorem}
\newtheorem{lem}[prop]{Lemma}
\newtheorem{cor}[prop]{Corollary}
\newtheorem*{thm*}{Theorem}
\numberwithin{equation}{section}

\theoremstyle{remark}
\newtheorem{example}[prop]{Example}

\newtheorem{remark}[prop]{Remark}
\newtheorem{warning}[prop]{Warning}

\theoremstyle{definition}
\newtheorem{defn}[prop]{Definition}         

\newtheorem{nota}[prop]{Notation}

\newcommand{\sslash}{\mathbin{/\mkern-6mu/}}

\DeclareMathOperator{\ev}{ev}
\DeclareMathOperator{\coev}{coev}

\DeclareMathOperator{\mate}{mate}
\DeclareMathOperator{\op}{op}

\newcommand{\dslash}[1]{_{{\sslash}#1}}
\newcommand{\id}{\mathbf{1}}
\newcommand{\sVec}{{\mathsf {sVec}}}
\renewcommand{\Vec}{{\mathsf {Vec}}}
\newcommand{\Fun}{{\mathsf {Fun}}}
\newcommand{\Set}{{\mathsf {Set}}}

\def\semicolon{;}
\def\applytolist#1{
    \expandafter\def\csname multi#1\endcsname##1{
        \def\multiack{##1}\ifx\multiack\semicolon
            \def\next{\relax}
        \else
            \csname #1\endcsname{##1}
            \def\next{\csname multi#1\endcsname}
        \fi
        \next}
    \csname multi#1\endcsname}

\def\calc#1{\expandafter\def\csname c#1\endcsname{{\mathcal #1}}}
\applytolist{calc}QWERTYUIOPLKJHGFDSAZXCVBNM;
\def\bbc#1{\expandafter\def\csname bb#1\endcsname{{\mathbb #1}}}
\applytolist{bbc}QWERTYUIOPLKJHGFDSAZXCVBNM;
\def\bfc#1{\expandafter\def\csname bf#1\endcsname{{\mathbf #1}}}
\applytolist{bfc}QWERTYUIOPLKJHGFDSAZXCVBNM;

\makeatletter
\newlength{\L@UnitsRaiseDisplaystyle}
\newlength{\L@UnitsRaiseTextstyle}
\newlength{\L@UnitsRaiseScriptstyle}
\DeclareRobustCommand*{\@UnitsNiceFrac}[3][]{%
  \ifthenelse{\boolean{mmode}}{%
    \settoheight{\L@UnitsRaiseDisplaystyle}{%
      \ensuremath{\displaystyle#1{M}}%
    }%
    \settoheight{\L@UnitsRaiseTextstyle}{%
      \ensuremath{\textstyle#1{M}}%
    }%
    \settoheight{\L@UnitsRaiseScriptstyle}{%
      \ensuremath{\scriptstyle#1{M}}%
    }%
    \settoheight{\@tempdima}{%
      \ensuremath{\scriptscriptstyle#1{M}}%
    }%
    \addtolength{\L@UnitsRaiseDisplaystyle}{%
      -\L@UnitsRaiseScriptstyle%
    }%
    \addtolength{\L@UnitsRaiseTextstyle}{%
      -\L@UnitsRaiseScriptstyle%
    }%
    \addtolength{\L@UnitsRaiseScriptstyle}{-\@tempdima}%
    \mathchoice
      {%
        \raisebox{\L@UnitsRaiseDisplaystyle}{%
          \ensuremath{\scriptstyle#1{#2}}%
        }%
      }%
      {%
        \raisebox{\L@UnitsRaiseTextstyle}{%
          \ensuremath{\scriptstyle#1{#2}}%
        }%
      }%
      {%
        \raisebox{\L@UnitsRaiseScriptstyle}{%
          \ensuremath{\scriptscriptstyle#1{#2}}%
        }%
      }%
      {%
        \raisebox{\L@UnitsRaiseScriptstyle}{%
          \ensuremath{\scriptscriptstyle#1{#2}}%
        }%
      }%
    \mkern-2mu{\sslash}\mkern-1mu%
    \bgroup
      \mathchoice
        {\scriptstyle}%
        {\scriptstyle}%
        {\scriptscriptstyle}%
        {\scriptscriptstyle}%
      #1{#3}%
    \egroup
  }%
  {%
    \settoheight{\L@UnitsRaiseTextstyle}{#1{M}}%
    \settoheight{\@tempdima}{%
      \ensuremath{%
        \mbox{\fontsize\sf@size\z@\selectfont#1{M}}%
      }%
    }%
    \addtolength{\L@UnitsRaiseTextstyle}{-\@tempdima}%
    \raisebox{\L@UnitsRaiseTextstyle}{%
      \ensuremath{%
        \mbox{\fontsize\sf@size\z@\selectfont#1{#2}}%
      }%
    }%
    \ensuremath{\mkern-2mu}{\sslash}\ensuremath{\mkern-1mu}%
    \ensuremath{%
      \mbox{\fontsize\sf@size\z@\selectfont#1{#3}}%
    }%
  }%
}

\DeclareRobustCommand*{\nicefrac}{\@UnitsNiceFrac}%
\makeatother

\begin{document}

\title{Completion for braided enriched monoidal categories}
\author{Scott Morrison, David Penneys, and Julia Plavnik}
\date{\today}
\begin{abstract}
Monoidal categories enriched in a braided monoidal category $\mathcal{V}$ are classified by braided oplax monoidal functors from $\mathcal{V}$ to the Drinfeld centers of ordinary monoidal categories.
In this article, we prove that this classifying functor is strongly monoidal if and only if the original $\mathcal{V}$-monoidal category is tensored over $\mathcal{V}$.
We then define a completion operation which produces a tensored $\cV$-monoidal category $\overline{\mathcal{C}}$ from an arbitrary $\mathcal{V}$-monoidal category $\mathcal{C}$, and we determine many equivalent conditions which imply $\mathcal{C}$ and $\overline{\mathcal{C}}$ are $\mathcal{V}$-monoidally equivalent.

Since being tensored is a property of the underlying $\mathcal{V}$-category of a $\mathcal{V}$-monoidal category, we begin by studying the equivalence between (tensored) $\mathcal{V}$-categories and oplax (strong) $\mathcal{V}$-module categories respectively.
We then define the completion operation for $\mathcal{V}$-categories, and adapt these results to the $\mathcal{V}$-monoidal setting.

\end{abstract}
\maketitle

\section{Introduction}

In \cite{1701.00567}, the first two authors studied the notion of a monoidal category enriched in a braided monoidal category $\cV$.
While the notion of a monoidal category enriched in a symmetric closed monoidal category has been extensively studied in the enriched category theory literature (e.g. \cite{MR749468,MR2177301,MR2219705,MR3775482}), 
the fact that the base for enrichment may be a braided monoidal category $\cV$ that is not symmetric has not been extensively explored. We note that Remark 5.2 of \cite{MR1250465} foreshadows this development, and \cite{MR2927355} studies an even more general setting,
of a category enriched in a duoidal category, which specialises to the present one when the two tensor products agree.
In \cite{1701.00567}, we proved a classification result for $\cV$-monoidal categories, which we improve slightly in \S\ref{sec:ClassificationOfClosedVMonoidalCategories} below to obtain

\begin{thm}
\label{thm:ClosedOplaxVMonoidal}
Let $\cV$ be a braided monoidal category.
There is a bijective correspondence
\[
\left\{\,
\parbox{7cm}{\rm Closed $\cV$-monoidal categories $\cC$ such that $\cC^\cV(1_\cC\to -)$ admits a left adjoint}\,\right\}
\,\,\cong\,\,
\left\{\,\parbox{8.3cm}{\rm Pairs $(\cT,\cF^{\scriptscriptstyle Z})$ with $\cT$ a closed monoidal category and $\cF^{\scriptscriptstyle Z}: \cV\to Z(\cT)$ braided oplax monoidal, such that $\cF:=\cF^{\scriptscriptstyle Z}\circ R$ admits a right adjoint}\,\right\}.
\]
\end{thm}
(Note that the main result of \cite{1701.00567} required `rigid' in place of `closed'.)

Here, $\cC^\cV$ denotes the \emph{underlying category} of $\cC$ with the same objects and hom sets $\cC^\cV(a\to b) := \cV(1_\cV \to \cC(a\to b))$.
We call $\cC$ \emph{closed} if every $\cV$-functor $a\otimes - : \cC \to \cC$ admits a right $\cV$-adjoint $[a,-]$.
We refer the reader to \S\ref{sec:VAdjunctions} below or to \cite[\S1.11]{MR2177301} for a discussion about $\cV$-adjunctions between $\cV$-functors.
On the other side, $R: Z(\cT) \to \cT$ denotes the forgetful functor, and by an abuse of nomenclature, we assume all oplax monoidal functors are \emph{strongly unital}, i.e., $\cF(1_\cV) = 1_\cT$, and the oplaxitor morphisms $\nu_{u, v}$ are isomorphisms whenever $u$ or $v$ is $1_\cV$.

As in \cite{1701.00567}, we adopt the convention that we write composition of maps from \emph{left to right}, contrary to the convention of most of mathematics.
This has implications for other conventions, like our conventions for the internal hom in a closed monoidal category (Notation \ref{nota:InternalHomAndCompostionInVHat}) 
and the evaluation and coevaluation in a rigid monoidal category (Example \ref{ex:VRigid}).
We also suppress all associators and unitors of monoidal categories to ease the notation.

Theorem \ref{thm:ClosedOplaxVMonoidal} is somewhat surprising due to the presence of the adjective \emph{oplax}, and the absence of any type of rigidity or pivotal structure.
In comparison, the article \cite{1607.06041} shows there is an equivalence of categories between
\emph{anchored planar algebras} in a braided pivotal category $\cV$ and triples $(\cC, \cF^{\scriptscriptstyle Z}, c)$ where 
$\cC$ is a pivotal category, 
$\cF^{\scriptscriptstyle Z}: \cV \to Z(\cC)$ is a braided pivotal strong monoidal functor, 
and
$c$ generates $\cC$ as a \emph{module tensor category} \cite{1701.00567}.

It is thus natural to ask the question: if this braided oplax monoidal functor $\cV \to Z(\cT)$ is in fact strong monoidal, what can we say about the corresponding  $\cV$-monoidal category?
In \cite[\S1.2]{1701.00567}, we claimed this property is exactly that $\cC$ is \emph{tensored} over $\cV$.
To define the property of being tensored, we need $\cV$ to be closed to define the \emph{self-enrichment} $\widehat{\cV}$ (see Examples \ref{example:SelfEnrichment} and \ref{ex:SelfEnrichedVMonoidal}).
A $\cV$-category $\cC$ is called tensored if every $\cV$-representable functor $\cC(a\to -): \cC \to \widehat{\cV}$ (see \cite[\S 1.6]{MR2177301} or \S\ref{ex:VRepresentable}) admits a left $\cV$-adjoint; notice this is a property of the underlying $\cV$-category of a $\cV$-monoidal category.
In \S\ref{sec:ClassificationOfTensoredClosedVMonoidalCategories} below, we prove the following.

\begin{thm}
\label{thm:TensoredVMonoidalEquivalence}
Let $\cV$ be a closed monoidal category.
Under Theorem \ref{thm:ClosedOplaxVMonoidal}, there is a bijective correspondence
\[
\left\{\,
\parbox{6.9cm}{\rm Tensored closed $\cV$-monoidal categories}
\,\right\}
\,\,\cong\,\,
\left\{\,\parbox{8.3cm}{\rm Pairs $(\cT,\cF^{\scriptscriptstyle Z})$ with $\cT$ a closed monoidal category and $\cF^{\scriptscriptstyle Z}: \cV\to Z(\cT)$ braided strong monoidal, such that $\cF:=\cF^{\scriptscriptstyle Z}\circ R$ admits a right adjoint}\,\right\}.
\]
We also get a bijective correspondence replacing closed with rigid on both sides above.
\end{thm}

As a final application, we discuss the notion of \emph{completion} for an arbitrary 
$\cV$-monoidal category $\cC$ when $\cV$ is closed.
The completion $\overline{\cC}$ is a tensored $\cV$-monoidal category which comes with a canonical inclusion $\cV$-monoidal functor $\cI:\cC \to \overline{\cC}$ which satisfies a universal property.
There are interesting open questions related to this universal property, and we refer the reader to Remark \ref{remark:OpenQuestionAboutCompletion} for more details.
We give many equivalent conditions under which $\cC$ is $\cV$-monoidally equivalent to its completion $\overline{\cC}$ in Theorem \ref{thm:EquivalentConditionsForVMonoidalCComplete}.

\begin{remark}
When $\cV$ is symmetric, completion of a $\cV$-category can also be thought of as a partial cocompletion, closing the representable presheaves in $\cV-Fun(\cC^{op} \to \hat{\cV})$ under certain weighted colimits. We have not seen a development of this idea when $\cV$ is merely braided (or indeed a generalisation of weighted colimits to that setting). Even before passing to the case of $\cV$-monoidal categories, this article takes a more pedestrian approach.

Indeed Ross Street has suggested to us that the monoidal structure we specify here on the completion of a $\cV$-monoidal category should be thought of as the restriction of Day convolution on the entire category of presheaves.
At this point, we do not know how relaxed one can by about $\cV$ and have this work: when $\cV$ is symmetric and cocomplete, it should be straightforward, but we have not yet studied Day convolution on presheaves when $\cV$ is merely braided and so have not verified that Day convolution provides an alternative route to the completion of a $\cV$-monoidal category.
\end{remark}

\subsection{Tensored \texorpdfstring{$\cV$}{V}-categories}

As mentioned above, being tensored is a property, and not extra structure, of the underlying $\cV$-category of a $\cV$-monoidal category, which we view as enriched in the closed monoidal category $\cV$ where we forget the braiding.
Thus to understand tensored $\cV$-monoidal categories and the completion operation, we begin by working \emph{one categorical level down} with ordinary $\cV$-categories.

Versions of the following folklore theorems have been proven many times, and \cite{MR649797} is the earliest appearance of which we are aware.\footnote{
Other proofs appear in \cite{MR1897810}, \cite{MR1466618} and \cite[Lem.~4.7]{1110.3567}, all of which do not cite \cite{MR649797}!
 }

\begin{thm}
\label{thm:OplaxVMod}
Let $\cV$ be a monoidal category.
There is a bijective correspondence
\[
\left\{\,
\parbox{5.7cm}{\rm$\cV$-categories $\cC$ such that each $\cC^\cV(a\to -) : \cC^\cV \to \cV$ admits a left adjoint}
\,\right\}
\,\,\cong\,\,
\left\{\,
\parbox{7cm}{\rm Strongly unital oplax $\cV$-modules $\cM$ such that each $m\vartriangleleft -$ admits a right adjoint}
\,\right\}.
\]
\end{thm}

Here, an oplax $\cV$-module means the associators $\alpha_{m,u,v} \in \cM(m\vartriangleleft uv \to m\vartriangleleft u\vartriangleleft v)$ need not be an isomorphism, and strongly unital means that each unitor $\rho_m \in \cM(m \to m\vartriangleleft 1_\cV)$ is an isomorphism, as is $\alpha_{m,u,v}$ whenever $u$ or $v$ is $1_\cV$.

\begin{thm}
\label{thm:StrongVMod}
Let $\cV$ be a closed monoidal category.
Under Theorem \ref{thm:OplaxVMod}, there is a bijective correspondence
\[
\left\{\,
\parbox{4cm}{\rm Tensored $\cV$-categories}
\,\right\}
\,\,\cong\,\,
\left\{\,
\parbox{6cm}{\rm Strong $\cV$-modules such that each $m\vartriangleleft -$ admits a right $\cV$-adjoint}
\,\right\}.
\]
\end{thm}

It is also interesting to note that the article \cite{MR649797} also points out that many results proven when $\cV$ is symmetric closed still hold when $\cV$ is merely braided closed!

Unfortunately, many of the above proofs available leave substantial parts to the reader.
In order to fully understand the operation of completion for $\cV$-categories, we provide yet another independent proof of Theorems \ref{thm:OplaxVMod} and \ref{thm:StrongVMod} in \S\ref{sec:V-cat and V-mod} and \S\ref{sec:Tensored V-cat and Strong V-mod} respectively.
Along the way, we prove a helpful result on lifting underlying adjunctions to $\cV$-adjunctions in Theorem \ref{thm:LiftToVAdjunction}, which introduces the notion of a \emph{tensored $\cV$-functor} (see \S\ref{sec:IntroCompletionForVCategories} right below and \S\ref{sec:TensoredVFunctors} for more details).

Now in the case of a closed $\cV$-monoidal category $\cC$ for which the functor $\cC^\cV(1_\cC \to -): \cC^\cV \to \cV$ admits a left adjoint, the oplax $\cV$-module structure of $\cC^\cV$ can be easily written in terms of the classifying oplax monoidal functor $\cF = \cF^{\scriptscriptstyle Z}: \cV \to \cC^\cV$ from Theorem \ref{thm:ClosedOplaxVMonoidal} by $c\vartriangleleft v := c\cF(v)$ and $\alpha_{c,u,v} := \id_c \nu_{u,v}$.

\subsection{Completion for \texorpdfstring{$\cV$}{V}-(monoidal) categories}
\label{sec:IntroCompletionForVCategories}

We discuss the operation of completion for $\cV$-categories in \S\ref{sec:CompletionForVCats}.
Starting with a closed monoidal category $\cV$ and a $\cV$-category $\cC$, 
we define the 
objects of the $\cV$-category $\overline{\cC}$ as the formal expression $a\blacktriangleleft u$ for $a\in \cC$ and $u \in \cV$, and we define the hom objects by $\overline{\cC}(a\blacktriangleleft u \to b\blacktriangleleft v) := \widehat{\cV}(u \to \cC(a\to b)v)$.
Similar to the self-enrichment $\widehat{\cV}$, the composition morphism $-\circ_{\overline{\cC}}-$ is given by taking the mate of 
$$
(\varepsilon^{\widehat{\cV}}_{u \to \cC(a\to b)v} \id_{\widehat{\cV}(v\to \cC(b\to c)w)})
\circ
(\id_{\cC(a\to b)} \varepsilon^{\widehat{\cV}}_{v \to \cC(b\to c)w})
\circ
((-\circ_\cC -)\id_w)
$$
under the adjunction
\begin{align*}
\cV(\overline{\cC}(a\blacktriangleleft u \to b\blacktriangleleft v)
&\overline{\cC}(b\blacktriangleleft v\to c\blacktriangleleft w) \to
\overline{\cC}(a\blacktriangleleft u\to c\blacktriangleleft w))
\\&=
\cV(\widehat{\cV}(u\to \cC(a\to b)v)\widehat{\cV}(v\to \cC(b\to c)w) \to \widehat{\cV}(u\to \cC(a\to c)w))
\\&\cong
\cV(u\widehat{\cV}(u\to \cC(a\to b)v)\widehat{\cV}(v\to \cC(b\to c)w) \to \cC(a\to c)w)
\end{align*}
Here, $\varepsilon^{\widehat{\cV}}_{u \to v} : \cV(u \widehat{V}(u \to v) \to v)$ is the unit of the adjunction $\cV(uw \to v) \cong \cV(w \to \widehat{\cV}(u\to v))$
(notice our convention for internal hom is not the usual one).

We get a canonical $\cV$-functor $\cI : \cC \to \overline{\cC}$ by defining $\cI(c) = c\blacktriangleleft 1_\cV$ and $\cI_{c\to d}$ is the mate of $\id_{\cC(c\to d)}$ under the adjunction
$$
\cV(\cC(c\to d) \to \overline{\cC}(\cI(c) \to \cI(d)))
=
\cV(\cC(c\to d) \to \widehat{\cV}(1_\cV \to \cC(c\to d)))
\cong
\cV(\cC(c\to d) \to \cC(c\to d)).
$$

Now when we assume our $\cV$-category $\cC$ is tensored, we give many equivalent conditions under which $\cC$ is equivalent to its completion $\overline{\cC}$ in Theorem \ref{thm:WhenTensoredIsEquivalentToCompletion}.
We state a few of these conditions here.

\begin{thm}
\label{thm:IntroEquivalentCharacterizationsOfEquivalence}
Suppose $\cC$ is a tensored $\cV$-category.
The following are equivalent:
\begin{enumerate}
\item
Every $\cV$-representable functor $\cR^a=\cC(a\to -):\cC\to \widehat{\cV}$ is tensored.
\item
The $\cV$-functor $\cI: \cC \to \overline{\cC}$ is tensored.
\item
The $\cV$-functor $\cI: \cC \to \overline{\cC}$ witnesses a $\cV$-equivalence.
\end{enumerate}
\end{thm}

Suppose $\cC$ and $\cD$ are $\cV$-categories, such that all representable functors for $\cC^\cV$ and for $\cD^\cV$ admit left adjoints. (In Definition \ref{def:oplax-tensor} we called such categories \emph{oplax tensored}.) In Definition \ref{defn:TensoredVFunctor} we define the notion of a $\cV$-functor $\cF:\cC \to \cD$ being \emph{tensored}, and we sketch this here.
Given a $\cV$-functor $\cF: \cC \to \cD$ 
the underlying functor $\cF^\cV : \cC^\cV \to \cD^\cV$ can be canonically endowed with the structure of a strongly unital \emph{lax} $\cV$-module functor between the \emph{oplax} $\cV$-module categories obtained from Theorem \ref{thm:OplaxVMod}.
Indeed, we define the \emph{laxitor} $\mu_{c,v}^\cF \in \cD^\cV(\cF(c)\vartriangleleft v \to \cF(c\vartriangleleft v))$ to be the mate of 
$\eta^\cC \circ \cF_{c \to c\vartriangleleft v}$ under the adjunction
$$
\cD^\cV(\cF(c)\vartriangleleft v \to \cF(c\vartriangleleft v)) 
= 
\cV(1_\cV \to \cD(\cF(c)\vartriangleleft v \to \cF(c\vartriangleleft v)))
\cong
\cV(v \to \cD(\cF(c) \to \cF(c\vartriangleleft v))),
$$
where $\eta^\cC$ is the unit of the adjunction
$\cV(v \to \cC(c \to c\vartriangleleft v)) \cong \cC^\cV(c\vartriangleleft v\to c\vartriangleleft v)$.
Now we say the $\cV$-functor $\cF$ is tensored if $\mu^\cF_{c,v}$ is an isomorphism for all $c\in \cC$ and $v\in \cV$.

As an interesting aside, we prove in Lemma \ref{lem:RigidCharacterization} that a closed monoidal category $\cV$ is \emph{rigid} if and only if every $\cV$-representable functor $\widehat{\cV}(v \to -)$ is tensored.
As a corollary, we see that when $\cV$ is rigid, then a $\cV$-category $\cC$ is tensored if and only if it is $\cV$-equivalent to its $\cV$-completion via the inclusion $\cI: \cC \to \overline{\cC}$.

Returning to the $\cV$-monoidal setting, when $\cC$ is an arbitrary $\cV$-monoidal category, in \S\ref{sec:CompletionForVMonoidalCategories},
we endow $\overline{\cC}$ with a $\cV$-monoidal structure by defining 
$(a\blacktriangleleft u)(b\blacktriangleleft v) := ab\blacktriangleleft uv$ 
and
the tensor product morphism $-\otimes_{\overline{\cC}}-$ as the mate of
$$
(\id_u \beta_{w, \widehat{\cV}(u\to\cC(a\to b)v)} \id_{\widehat{\cV}(w\to \cC(c\to d)x)})
\circ
(\varepsilon^{\widehat{\cV}}_{u \to \cC(a\to b)v} \varepsilon^{\widehat{\cV}}_{w \to \cC(c\to d)x})
\circ
(\id_{\cC(a\to b)} \beta_{v,\cC(c\to d)}^{-1} \id_x)
\circ
((-\otimes_\cC-)\id_{vx})
$$
under the adjunction
\begin{align*}
\cV(\widehat{\cV}(u\to \cC(a\to b)v)& \widehat{\cV}(w\to \cC(c\to d)x) \to \widehat{\cV}(uw\to \cC(ac\to bd)vx)
\\&\cong
\cV(uw \widehat{\cV}(u\to \cC(a\to b)v)\widehat{\cV}(w\to \cC(c\to d)x) \to \cC(ac\to bd)vx).
\end{align*}
In this setting, the $\cV$-functor $\cI : \cC \to \overline{\cC}$ can be trivially equipped with the structure of a $\cV$-monoidal functor.
Again as being tensored is a property of the underlying $\cV$-category, we get the following formal consequence of Theorem \ref{thm:IntroEquivalentCharacterizationsOfEquivalence}.

\begin{cor}
Suppose $\cC$ is a tensored closed $\cV$-monoidal category.
The following are equivalent:
\begin{enumerate}
\item
Every $\cV$-representable functor $\cR^a=\cC(a\to -):\cC\to \widehat{\cV}$ is tensored.
\item
The $\cV$-functor $\cI: \cC \to \overline{\cC}$ is tensored.
\item
The $\cV$-monoidal functor $\cI: \cC \to \overline{\cC}$ witnesses a $\cV$-monoidal equivalence.
\end{enumerate}
\end{cor}

As an example, which has many obvious generalizations, one may consider the category of vector spaces $\Vec$ as enriched in super vector spaces $\sVec$ in the trivial way which does not see the odd line.
The completion $\overline{\Vec}$ is again $\sVec$.

Finally, in \S\ref{sec:CompletionClosedWhenVRigid}, we discuss the open question of whether $\cC$ being closed $\cV$-monoidal implies $\overline{\cC}$ is closed.
This problem is similar in spirit to the open question raised in Remark \ref{remark:OpenQuestionAboutCompletion} pertaining to the universal property of the completion.
Clearly any of the hypotheses of the above corollary imply $\overline{\cC}$ is closed, as it is then $\cV$-monoidally equivalent to the closed tensored $\cV$-monoidal category $\cC$.
We show that $\overline{\cC}$ is closed when $\cC$ is closed whenever $\cV$ is rigid in Proposition \ref{prop:ClosedImpliesCompletionClosed} below.

\subsection*{Acknowledgements}

Substantial portions of this project were discussed at
\begin{itemize}
\item
the 2016 Trimester on von Neumann algebras at the Hausdorff Institute for Mathematics (HIM)
\item
the 2016 Noncommutative Geometry and Operator Algebras Spring Institute at the Isaac Newton Institute
supported by EPSRC grant no EP/K032208/1
\item
the 2016 and 2017 AIM SQuaREs on Classifying fusion categories, and
\item
the 2018 AMS MRC on Quantum Symmetries: Subfactors and Fusion Categories
supported by NSF DMS grant 1641020.
\end{itemize}
The authors would like to thank all the various organizers and granting agencies involved.

The authors would like to thank 
Marcel Bischoff,
Andr\'{e} Henriques, 
Ross Street, and
Kevin Walker
for helpful conversations.
Scott Morrison was partially supported by Australian Research Council grants
`Subfactors and symmetries' DP140100732 and `Low dimensional categories' DP160103479.
David Penneys was partially supported by NSF DMS grants 1500387/1655912 and 1654159.
Julia Plavnik was partially supported by NSF DMS grant 1802503.

\section{Enriched categories}

For this article, $\cV$ will denote a monoidal category.
As in \cite{1701.00567}, we will always write composition of morphisms from left to right, contrary to the way most mathematics is written.
This results in other conventional differences which we address as they arise.
To ease the notation, we will omit writing tensor product symbols whenever possible and suppress all unitors and associators in $\cV$.

\begin{nota}
\label{nota:InternalHomAndCompostionInVHat}
When $\cV$ is closed, our convention for the internal hom $\widehat{\cV}(u\to v)$ for $u,v: \cV$ is given by the following adjunction:
\begin{equation}
\label{eq:InternalHom}
\cV(uw \to v)
\cong
\cV(w \to \widehat{\cV}(u\to  v)).
\end{equation}
The \emph{evaluation morphism} or \emph{counit} $\varepsilon^{\widehat{\cV}}_{u\to v}\in \cV(u\widehat{\cV}(u\to v) \to v)$ is the mate of $\id_{\widehat{\cV}(u\to v)}$ under the adjunction
$$
\cV(u\widehat{\cV}(u\to v)\to v)
\cong
\cV(\widehat{\cV}(u\to v)\to \widehat{\cV}(u\to v)).
$$
\end{nota}

\subsection{\texorpdfstring{$\cV$}{V}-modules}

We briefly discuss module categories and module functors.

\begin{defn}
An \emph{oplax} right $\cV$-module category is a 1-category $\cM$ together with the following data:
\begin{itemize}
\item
a bifunctor $\vartriangleleft : \cM \times \cV \to \cM$
\item
\emph{oplaxitor} morphisms
$\alpha_{m,u,v}\in
\cM(m \vartriangleleft uv
\to
m\vartriangleleft u\vartriangleleft v)$
for $m\in \cM$, $u,v\in \cV$,
and
\item
distinguished morphisms $\rho_m\in \cM(m\vartriangleleft 1_\cV \to m)$
\end{itemize}
satisfying the following axioms:
\begin{itemize}
\item
(naturality)
the $\alpha_{m,u,v}$ and $\rho_m$ are natural in all variables,
\item
(associativity) for all $m\in \cM$ and $u,v,w\in \cV$, the following diagram commutes (where we suppress the associator in $\cV$)
$$
\begin{tikzcd}
m\vartriangleleft uvw
\ar[rr, "\alpha_{m,u,vw}"]
\ar[d, "\alpha_{m,uv,w}"]
&&
m\vartriangleleft u\vartriangleleft vw
\ar[d, "\alpha_{m\vartriangleleft u,v,w}"]
\\
m\vartriangleleft uv \vartriangleleft w
\ar[rr, "\alpha_{m,u,v}\vartriangleleft \id_w"]
&&
m\vartriangleleft u\vartriangleleft v \vartriangleleft w
\end{tikzcd},
$$
and
\item
(unitality)
for all $m\in \cM$ and $v\in \cV$, the following diagram commutes:
$$
\begin{tikzcd}
m\vartriangleleft 1_\cV v
\ar[dr, "{=}"]
\ar[rr, "\alpha_{m,1_\cV, v}"]
&&
m\vartriangleleft 1_\cV \vartriangleleft v
\ar[dl, "\rho_m\vartriangleleft \id_v"]
\\
{}
&
m\vartriangleleft v
\end{tikzcd}
$$
\end{itemize}
An oplax right $\cV$-module category is called \emph{strongly unital}\footnote{In \cite{1701.00567} we called this property `strictly unital', but now prefer `strongly unital'.} if all morphisms $\rho$ above are isomorphisms.
An oplax right $\cV$-module category is called a \emph{strong} right $\cV$-module category if all morphisms $\alpha, \rho$ above are isomorphisms.
\end{defn}

Generally, the term `strong' is omitted, and we merely refer to $\cV$-module categories and oplax $\cV$-module categories.
There is a similar notion of a(n oplax) left $\cV$-module category.

\begin{defn}
\label{defn:LaxVModuleFunctor}
Suppose $\cM, \cN$ are right oplax $\cV$-modules.
A \emph{lax $\cV$-module functor $\cM \to \cN$} is a pair $(\cF, \mu)$ where $\cF: \cM \to \cN$ is a functor and $\mu$ is a family of maps
$
\{
\mu_{m,v}\in \cN(\cF(m)\vartriangleleft v \to \cF(m\vartriangleleft v))
\}
$
such that
\begin{itemize}
\item 
(naturality)
$\mu_{m,v}$ is natural in $m\in \cM$ and $v\in \cV$,
\item
(associativity)
for all $m\in\cM$ and $u,v\in \cV$,
$\mu_{m,uv}\circ \cF(\alpha^\cC_{m,u,v}) = \alpha^\cD_{\cF(m), u,v}\circ (\mu_{m,u}\vartriangleleft \id_v) \circ \mu_{m\vartriangleleft u, v}$ as morphisms in $\cN(\cF(m)\vartriangleleft uv \to \cF(m\vartriangleleft u\vartriangleleft v))$, and
\item
(unitality)
for all $m\in \cM$, $\mu_{m,1_\cV}\circ  \cF(\rho^\cM_m) = \rho^\cN_{m} \in \cN(\cF(m)\vartriangleleft 1_\cV \to \cF(m))$.
\end{itemize} 
A lax $\cV$-module functor is called \emph{strongly unital} if every $\mu_{m,1_\cV}$ is an isomorphism.
A lax $\cV$-module functor is called a \emph{(strong) $\cV$-module functor} if every $\mu_{m,v}$ is an isomorphism.

An \emph{equivalence} between oplax $\cV$-modules consists of a pair of strong $\cV$-module functors between oplax $\cV$-modules which is an equivalence of the underlying categories.
\end{defn}

\begin{remark}
\label{rem:AutomaticStrictUnital}
If $\cM$ and $\cN$ are strongly unital so that $\rho^\cM_m$ and $\rho^\cN_{n}$ are isomorphisms for $m\in\cM$ and $n\in \cN$, then any lax $\cV$-module functor $\cF$ is automatically strongly unital.
Observe two out of three morphisms in the unital relation being invertible implies the third is invertible as well.
\end{remark}

\subsection{\texorpdfstring{$\cV$}{V}-categories, \texorpdfstring{$\cV$}{V}-functors, and \texorpdfstring{$\cV$}{V}-natural transformations}

We now recall the basics of enriched categories from \cite[\S1]{MR2177301}.

\begin{defn}
A $\cV$-category $\cC$ consists of a collection of objects,
an assignment of a \emph{hom object} $\cC(a\to b)$ to every pair of objects $a,b\in \cC$,
to each object $a\in \cC$, an \emph{identity element} $j_a\in \cV(1_\cV \to \cC(a\to a))$,
and to each triple of objects $a,b,c\in \cC$, a \emph{composition morphism}
$
-\circ_\cC - \in \cV(\cC(a\to b)\cC(b\to c) \to \cC(a\to c))
$.
The composition morphisms must be associative, and the identity element must satisfy $(j_a \id_{\cC(a\to b)})\circ (-\circ_\cC-) = \id_{\cC(a\to b)} = (\id_{\cC(a\to b)}j_b)\circ (-\circ_\cC-)$ for all $a,b\in \cC$.
\end{defn}

\begin{example}[Self enrichment]
\label{example:SelfEnrichment}
When $\cV$ is closed, we may define the self enriched category $\widehat{\cV}$, which is $\cV$ thought of as a $\cV$-category.
The objects are the same as before, and hom objects are given by internal hom $\widehat{\cV}(u \to v)$.
The composition map $-\circ_{\widehat{\cV}}-:\widehat{\cV}(u \to v)\widehat{\cV}(v \to w) \to \widehat{\cV}(u \to w)$ is the mate of $(\varepsilon^{\widehat{\cV}}_{u\to v} \id_{\widehat{\cV}(v \to w)})\circ \varepsilon^{\widehat{\cV}}_{v\to w}$ under the adjunction
$$
\cV(u\widehat{\cV}(u \to v)\widehat{\cV}(v \to w) \to w) \cong \cV(\widehat{\cV}(u \to v)\widehat{\cV}(v \to w)\to \widehat{\cV}(u \to w)).
$$
\end{example}

\begin{defn}
Suppose $\cC, \cD$ are $\cV$-categories.
A $\cV$-functor $\cF: \cC \to \cD$ is a function on objects together with an assignment of a morphism $\cF_{a\to b}\in \cV(\cC(a\to b) \to \cD(\cF(a) \to \cF(b)))$ to each pair of objects $a,b\in\cC$  such that $(\cF_{a\to b}\cF_{b\to c})\circ (-\circ_\cD -) = (-\circ_\cC-) \circ \cF_{a\to c}$.
\end{defn}

\begin{example}[$\cV$-representable functors]
\label{ex:VRepresentable}
Suppose $\cC$ is a $\cV$-category with $\cV$ closed and $a\in \cC$.
We define the $\cV$-\emph{representable functor}\footnote{
Kelly defines $\cV$-representable functors in \cite[\S 1.6]{MR2177301} under the assumption $\cV$ is symmetric, but this assumption is not necessary.
Indeed, $\cV$ does not even need to be braided, merely monoidal.
This is also observed by \cite{MR649797}.
}
$\cC(a\to -): \cC\to \widehat{\cV}$ on objects by $b\mapsto \cC(a\to b)$ and $\cC(a\to b)_{b\to c}$ is the mate of $-\circ_\cC-$ under the adjunction
$$
\cV(\cC(b\to c) \to \widehat{\cV}(\cC(a\to b)\to \cC(a\to c)))
\cong
\cV(\cC(a\to b)\cC(b\to c) \to \cC(a\to c)).
$$
\end{example}

Suppose $\cF, \cG: \cD \to \cD$ are $\cV$-functors.

\begin{defn}
A $1_\cV$-graded $\cV$-natural transformation $\sigma : \cF \Rightarrow \cG$ is a collection of morphisms $\sigma_a\in \cV(1_\cV \to \cD(\cF(a) \to \cG(a)))$ such that for all $a,b\in \cC$, the following diagram commutes:
\begin{equation}
\label{eq:VNaturalTransformation}
\begin{tikzcd}
\cC(a\to b)
\ar[rr, "\sigma_a\cG_{a\to b}"]
\ar[d, "\cF_{a\to b}\sigma_b"]
&&
\cD(\cF(a) \to \cG(a))\cD(\cG(a)\to \cG(b))
\ar[d, "-\circ_\cD-"]
\\
\cD(\cF(a)\to \cF(b))\cD(\cF(b) \to \cG(b))
\ar[rr, "-\circ_\cD-"]
&&
\cD(\cF(a)\to \cG(b)).
\end{tikzcd}
\end{equation}
\end{defn}

\subsection{The underlying category/functor}

Let $\cC$ be a $\cV$-category.

\begin{defn}
The \emph{underlying category} $\cC^\cV$ of $\cC$ has the same objects as $\cC$, and the morphism sets are given by $\cC^\cV(a\to b) = \cV(1_\cV \to \cC(a\to b))$.
The identity in $\cC^\cV(a\to a) = \cV(1_\cV \to \cC(a\to a))$ is $j_a$, and composition is given for $f\in \cC^\cV(a\to b)$ and $g\in \cC^\cV(b\to c)$ by $f\circ g = (fg)\circ (-\circ_\cC-)$.
It is straightforward to verify using the axioms of a $\cV$-category that $\cC^\cV$ is an ordinary category.
\end{defn}

\begin{example}
\label{ex:UnderlyingCategoryOfVhat}
When $\cV$ is closed, the underlying category $\widehat{\cV}^\cV$ can be identified with $\cV$ under the adjunction
$$
\widehat{\cV}^\cV(u \to v) = \cV(1_\cV \to [u,v]) \cong \cV(u \to v).
$$
\end{example}

Suppose now $\cF: \cC \to \cD$ is a $\cV$-functor between $\cV$-categories.

\begin{defn}
The \emph{underlying functor} $\cF^\cV : \cC^\cV \to \cD^\cV$ is defined on objects by $\cF^\cV(c) = \cF(c)$ and on morphims by mapping $f\in \cC^\cV(a\to b)$ to $f\circ \cF_{a\to b}\in \cD^\cV(\cF(a)\to \cF(b))$.
It is straightforward to verify using the axioms of a $\cV$-functor that $\cF^\cV$ is an ordinary functor.
\end{defn}

\begin{example}[Representable functors]
\label{ex:Representable}
Suppose $\cC$ is a $\cV$-category and $a\in \cC$.
We define the \emph{representable functor} $\cR_a = \cC(a\to -): \cC^\cV \to \cV$ on objects by $b\mapsto \cC(a\to b)$ and on morphisms $f\in \cC^\cV(b \to c) = \cV(1_\cV \to \cC(b\to c))$ by $\cR_a(f)=(\id_{\cC(a\to b)} f)\circ (-\circ_\cC-)$.
It is straightforward to verify tht $\cR_a$ is a functor using the axioms of a $\cV$-category.
\end{example}

\begin{lem}
\label{lem:UnderlyingFunctorOfVRepresentable}
Suppose $\cV$ is closed, $\cC$ is a $\cV$-category, and $a\in \cC$.
Under the identification of $\widehat{\cV}^\cV = \cV$ in Example \ref{ex:UnderlyingCategoryOfVhat},
the underlying functor of the $\cV$-representable functor $\cR^a: \cC \to \widehat{\cV}$ is equal to the representable functor $\cR_a = (\cR^a)^\cV: \cC^\cV \to \cV$.
\end{lem}
\begin{proof}
On objects, we have $(\cR^a)^\cV(b) = \cC(a\to b) = \cR_a(b)$.
If $f\in \cC^\cV(b\to c) = \cV(1_\cV \to \cC(a\to b))$,
$(\cR^a)^\cV(f)$ is the mate of $f\circ \cR^a_{b\to c}$ under the adjunction
$$
\widehat{\cV}^\cV( \cC(a\to b) \to \cC(a \to c))
=
\cV(1_\cV \to \widehat{\cV}(\cC(a\to b)\to \cC(a\to c)))
\cong
\cV(\cC(a\to b)\to  \cC(a\to c)),
$$
which is exactly $(\id_{\cC(a\to b)} f)\circ (-\circ_\cC-) = \cR_a(f)$ by Example \ref{ex:VRepresentable}.
\end{proof}

Suppose that $\cF,\cG: \cC \to \cD$ are $\cV$-functors and $\sigma: \cF \Rightarrow \cG$ is a $\cV$-natural transformation.

\begin{defn}
The \emph{underlying natural transformation} $\sigma^\cV: \cF^\cV \Rightarrow \cG^\cV$ is defined by
$\sigma^\cV_a = \sigma_a \in \cV(1_\cV \to \cD(\cF(a) \to \cG(a))) = \cD^\cV(\cF^\cV(a)\to \cG^\cV(a))$.
It is straightforward to verify using \eqref{eq:VNaturalTransformation} that $\sigma^\cV$ is an ordinary natural transformation.
\end{defn}

\subsection{Representable functors, mates, and the Yoneda Lemma}

Suppose $\cC$ and $\cD$ are ordinary categories, $\cL: \cC \to \cD$ and $\cR: \cD \to \cC$ are functors, and $\cL$ is left adjoint to $\cR$, denoted $\cL\dashv \cR$.
This means for all $a\in \cC$ and $d\in \cD$, we have a natural isomorphism
\begin{equation}
\cD(\cL(a) \to d) \cong \cC(a \to \cR(d)).
\end{equation}
For $f\in \cD(\cL(a) \to d)$ and $g\in \cC(a \to \cR(d))$, we say $f$ is the mate of $g$ if $g$ is corresponds to $f$ under the above natural isomorphism. 
We also say $g$ is the mate of $f$.

The following helpful identities hold via naturality.
\begin{align}
\label{eq:MateWithL}
&\text{
If $f_1\in  \cC(a\to \cR(c))$ and $f_2\in \cC(\cR(c)\to \cR(d))$, we have $\mate(f_1\circ f_2) = \cL(f_1)\circ\mate(f_2)$.
}
\\
\label{eq:MateWithR}
&\text{
If $g_1\in  \cD(\cL(a)\to \cL(b))$ and $g_2\in \cD( \cL(b)\to d)$, we have $\mate(g_1\circ g_2) = \mate(g_1)\circ\cR(g_2)$.
}
\end{align}
The \emph{counit} of the adjunction $\cL \dashv \cR$ is the natural isomorphism $\eta: \id_\cC \Rightarrow \cL\circ \cR$ defined by $\eta_a \in \cC(a \to \cR(\cL(a)))\cong \cD(\cL(a) \to \cL(a))$ is the mate of $\id_{\cL(a)}$.
The \emph{unit} of the adjunction is the natural isomorphism $\varepsilon: \cR\circ \cL \Rightarrow \id_\cD$ where $\varepsilon_d \in \cD(\cL(\cR(d)) \to d)\cong \cC(\cR(d) \to \cR(d))$ is the mate of $\id_{\cR(d)}$.

The Yoneda lemma gives fully faithful functors $\cC \hookrightarrow \Fun(\cC \to \Set)$ by $c \mapsto \cC(c \to -)$ and  $\cC \hookrightarrow \Fun(\cC^{\op} \to \Set)$ by $c \mapsto \cC( -\to c)$.
Recall that a functor $\cF: \cC \to \Set$ is called \emph{representable} if there is a pair $(c, \gamma)$ consisting of a representing object $c\in \cC$ and a natural isomorphism $\gamma:\cF\Rightarrow \cC(c \to -)$.
Given any two representations $(a,\alpha)$ and $(b, \beta)$ of $\cF$, there is a canonical isomorphism $f\in \cC(a\to b)$ such that $\cC(f\to -) = \alpha^{-1} \circ \beta$.
Moreover, we obtain $f$ and its inverse by taking the `mate'\footnote{The use of the term `mate' in this context is not standard nomenclature, but it gives a good feel for the style of the proofs that follow using mates.} of the appropriate identity morphisms under the natural isomorphisms
$$
\cC(a\to c) \cong \cF(c) \cong \cC(b \to c).
$$
Taking $c= b$, we have $f=\alpha_b(\beta^{-1}_b(\id_b))$.
Taking $c = a$, we have $f^{-1} = \beta_a(\alpha_a^{-1}(\id_a))$.
We summarize this fact as follows:
\begin{equation}
\label{eq:YonedaInverses}
\xymatrix@R=2pt{
\cC(a\to b)
\ar@{<->}[r]^(.57){\cong}
&
\cF(b)
\ar@{<->}[r]^(.43){\cong}
&
\cC(b\to b)
&
\cC(a \to a)
\ar@{<->}[r]^(.57){\cong}
&
\cF(a)
\ar@{<->}[r]^(.43){\cong}
&
\cC(b\to a)
\\
f
\ar@{<->}[r]
&
\beta_b^{-1}(\id_b)
\ar@{<->}[r]
&
\id_b
&
\id_{a}
\ar@{<->}[r]
&
\alpha_a^{-1}(\id_a)
\ar@{<->}[r]
&
f^{-1}
}
\end{equation}

\subsection{\texorpdfstring{$\cV$}{V}-adjunctions}
\label{sec:VAdjunctions}

Suppose $\cC$ and $\cD$ are $\cV$-categories and $\cL: \cC \to \cD$ and $\cR: \cD \to \cC$ are $\cV$-functors.

\begin{defn}
We say that $\cL$ is a left $\cV$-adjoint of $\cR$ (equivalently $\cR$ is a right $\cV$-adjoint of $\cL$), denoted $\cL \dashv_\cV \cR$, if
there is a family of isomorphisms $\theta_{a, d} \in\cV(\cD(\cL(a)\to d)\to \cC(a\to \cR(d)))$ for $a\in \cC$ and $d\in \cD$,
such that for all $a,b\in \cC$ and all $c,d\in \cD$, the following two diagrams commute:
\begin{equation}
\label{eq:V-Adjoint1}
\begin{tikzcd}
\cC(a\to b) \cC(b \to \cR(d))
\ar[d, "\cL_{a\to b}\theta_{b, d}^{-1}"]
\ar[rr, "-\circ_\cC-"]
&&
\cC(a\to \cR(d))
\ar[d, "\theta_{a, d}^{-1}"]
\\
\cD(\cL(a)\to \cL(b))\cD(\cL(b) \to d)
\ar[rr, "-\circ_\cD-"]
&&
\cD(\cL(a)\to d)
\end{tikzcd}
\end{equation}
\begin{equation}
\label{eq:V-Adjoint2}
\begin{tikzcd}
\cD(\cL(a) \to c)\cD(c\to d)
\ar[d, "\theta_{a, c}\cR_{c\to d}"]
\ar[rr, "-\circ_\cD-"]
&&
\cD(\cL(a) \to d)
\ar[d, "\theta_{a, d}"]
\\
\cC(a \to \cR(c))\cC(\cR(c)\to \cR(d))
\ar[rr, "-\circ_\cC-"]
&&
\cC(a\to \cR(d))
\end{tikzcd}
\end{equation}
\end{defn}

\begin{remark}
In the above definition, we would like to be able to simply say that there is a $1_\cV$-graded natural isomorphism between functors
\begin{equation}
\label{eq:VNaturalIso}
\cD(\cL(a) \to d) \cong \cC(a \to \cR(d)).
\end{equation}
from $\cC^{\op} \times \cD \to \widehat{\cV}$
but this doesn't quite make sense at our level of generality; without assuming $\cV$ is braided we cannot form $\cC^{\op}$ nor the product $\cV$-category.
\end{remark}

\begin{remark}
Existence of a left $\cV$-adjoint of $\cR$ is equivalent to each $\cV$-functor $\cC(a \to \cR(-)) : \cD \to \widehat{\cV}$ being $\cV$-representable as in Example \ref{ex:VRepresentable}, and similarly for the existence of a right $\cV$-adjoint \cite[\S1.11]{MR2177301}.
Here, the $\cV$-functor $\cC(a \to \cR(-))$ is defined by setting
$\cC(a \to \cR(-))_{c\to d}$ to be the mate of $(\id \cR_{c\to d})\circ(-\circ_\cC-)$ under the adjunction
$$
\cV( \cD(c\to d) \to \widehat{\cV}(\cC(a\to \cR(c)) \to \cC(a\to \cR(d)))
\cong
\cV( \cC(a\to \cR(c))\cD(c\to d) \to  \cC(a\to \cR(d))).
$$
Indeed, by applying Adjunction \eqref{eq:InternalHom}, a $\cV$-natural isomorphism $\theta$ as in \eqref{eq:VNaturalIso} is equivalent to a collection of $1_\cV$-graded $\cV$-natural isomorphisms $\sigma^a: \cD(\cL(a)\to -) \Rightarrow \cC(a \to \cR(-))$.
One simply sets $\sigma^a_d$ to be the mate of $\theta_{a\to d}$ under the adjunction
$$
\cV(1_\cV \to \widehat{\cV}(\cD(\cL(a)\to d) \to \cC(a \to \cR(d)))
\cong
\cV(\cD(\cL(a)\to d) \to \cC(a \to \cR(d))).
$$
\end{remark}

\begin{remark}
\label{rem:UnderlyingAdjunction}
Note that if $\cL\dashv_\cV\cR$, then we have an adjunction of underlying functors $\cL^\cV\dashv \cR^\cV$ \cite[\S1.11]{MR2177301}.
Indeed the unit and counit of the underlying adjunction, denoted 
 $\eta^\cV: \id_{\cC^\cV} \Rightarrow \cL^\cV\circ \cR^\cV$ and $\varepsilon^\cV: \cR^\cV\circ \cL^\cV \Rightarrow \id_{\cD^\cV}$,
are given by 
$\eta^\cV_a := j_{\cL(a)}\circ \theta_{a,\cL(a)} \in \cC^\cV(a \to \cR(\cL(a)))$ 
and 
$\varepsilon^\cV_d := j_{\cR(d)}\circ \theta^{-1}_{\cR(d), d} \in \cD^\cV(\cL(\cR(d)) \to d)$.
For later use, we record the helpful relations
\begin{equation}
\label{eq:ThetaKappa}
\begin{split}
\theta_{a,d}
&=
(\eta^\cV_{a}\cR_{\cL(a)\to d} )\circ (-\circ_\cC-)
\in
\cV(\cD(\cL(a)\to d) \to \cC(a \to \cR(d)))
\\
\kappa_{a,d}
:=
\theta^{-1}_{a,d}
&=
(\cL_{a\to \cR(d)} \varepsilon^\cV_{d})\circ (-\circ_\cD-)
\in
\cV(\cC(a\to \cR(d)) \to \cD(\cL(a) \to d))
\end{split}
\end{equation}
which is easily verified using the naturality conditions \eqref{eq:V-Adjoint1} and \eqref{eq:V-Adjoint2}.
\end{remark}

\begin{warning}
\label{warning:BadNotationForUnderlyingAdjunction}
We warn the reader that the notation for the unit $\varepsilon^\cV$ and counit $\eta^\cV$ of the underlying adjunction of an arbitrary $\cV$-adjunction $\cL\dashv\cR$ is very similar to the notation for the unit $\varepsilon^{\widehat{\cV}}$ and the counit $\eta^{\widehat{\cV}}$ of the adjunction 
$\cV(uv \to w) \cong \cV(v \to \widehat{\cV}(u \to w))$.
The unit of this adjunction $\varepsilon^{\widehat{\cV}}$ is also called the evaluation morphism
for the self-enrichment $\widehat{\cV}$ from Notation \ref{nota:InternalHomAndCompostionInVHat} and Example \ref{example:SelfEnrichment}.
\end{warning}

We now prove a helpful lemma on lifting underlying adjunctions to $\cV$-adjunctions which is distilled from the first paragraph in \cite[p.~24]{MR2177301}.
Suppose $\cC,\cD$ are $\cV$-categories with $\cV$ closed, and $\cL: \cC\to \cD$ and $\cR: \cD \to \cC$ are $\cV$-functors.
Suppose $\cL^\cV\dashv \cR^\cV$, and denote the unit and counit of this adjunction by
$\eta^\cV_a \in \cV(1_\cV \to \cC(a \to \cR(\cL(a))))$
and
$\varepsilon^\cV_d \in \cV(1_\cV \to \cD(\cL(\cR(d))\to d))$.
For $a\in \cC$ and $d\in \cD$, define $\theta_{a,d}$ and $\kappa_{a,d}$ via the formulas \eqref{eq:ThetaKappa} above.
(We write $\kappa$ instead of $\theta^{-1}$ as we don't yet know $\theta$ is invertible.)
\begin{lem}
\label{lem:PromoteToVAdjunction}
If $\kappa_{a,d} = \theta_{a,d}^{-1}$ for all $a\in \cC$ and $d\in \cD$, then $\cL\dashv_\cV \cR$.
\end{lem}
\begin{proof}
It remains to prove the naturality conditions \eqref{eq:V-Adjoint1} and \eqref{eq:V-Adjoint2} for $\theta$.
We prove \eqref{eq:V-Adjoint2} and the other is as easy.
For all $a\in \cC$ and $c,d\in \cD$, 
\begin{align*}
(\theta_{a,c}\cR_{c\to d})\circ (-\circ_\cC-)
&=
(
[(\eta^\cV_{a}\cR_{\cL(a)\to c} )\circ (-\circ_\cC-)]
\cR_{c\to d}
)
\circ
(-\circ_\cC-)
\\&=
(
\eta^\cV_{a}
[(\cR_{\cL(a)\to c} \cR_{c\to d})\circ (-\circ_\cC-)]
)
\circ
(-\circ_\cC-)
\\&=
(
\eta^\cV_{a}
[(-\circ_\cD-) \circ\cR_{\cL(a)\to d}]
)
\circ
(-\circ_\cC-)
\\&=
(-\circ_\cD-)
\circ
(
[\eta^\cV_{a}\cR_{\cL(a)\to d}]
\circ
(-\circ_\cC-)
)
\\&=
(-\circ_\cD-) \circ \theta_{a,d}.
\qedhere
\end{align*}
\end{proof}

In Section \ref{sec:TensoredVFunctors}, we will prove another result, Theorem \ref{thm:LiftToVAdjunction}, about promoting ordinary adjunctions to $\cV$-adjunctions.

\subsection{Tensored \texorpdfstring{$\cV$}{V}-categories}

Suppose that $\cV$ is closed so that we may form $\widehat{\cV}$.

\begin{example}
\label{example:InnerHomVAdjunction}
We may consider $u\otimes-$ as a $\cV$-functor $\widehat{\cV}\to \widehat{\cV}$ by setting $(u\otimes-)_{v\to w}$ to be the mate of $\id_u \varepsilon^{\widehat{\cV}}_{v\to w}$ under the adjunction
$$
\cV(uv\widehat{\cV}(v\to w) \to uw)
\cong
\cV(\widehat{\cV}(v\to w) \to \widehat{\cV}(uv\to uw))
=
\cV(\widehat{\cV}(v\to w)\to \widehat{\cV}(uv\to uw))
.
$$
Similarly, we may consider $\widehat{\cV}(u\to -)$ as a $\cV$-functor $\widehat{\cV}\to \widehat{\cV}$ by setting $\widehat{\cV}(u\to -)_{v\to w}$ to be the mate of $-\circ_{\widehat{\cV}}-$ under the adjunction
\begin{align*}
\cV(\widehat{\cV}(u\to v)\widehat{\cV}(v\to w), \widehat{\cV}(u\to w))
&\cong
\cV(\widehat{\cV}(v\to w) \to \widehat{\cV}(\widehat{\cV}(u\to v)\to  \widehat{\cV}(u\to w)))
\\&=
\cV(\widehat{\cV}(v\to w) \to \widehat{\cV}(\widehat{\cV}(u\to v)\to \widehat{\cV}(u\to w)))
.
\end{align*}
It is straightforward to compute that $u\otimes -$ is a left $\cV$-adjoint to $\widehat{\cV}(u\to -)$ \cite[Eq.~(1.27)]{MR2177301}.
\end{example}

\begin{defn}
Following \cite{MR649797}, we call a $\cV$-category \emph{tensored} if each $\cV$-representable functor $\cR^a=\cC(a\to -): \cC \to \widehat{\cV}$ as defined in Example \ref{ex:VRepresentable} admits a left $\cV$-adjoint $\cL^a: \widehat{\cV}\to \cC$.
\end{defn}

\begin{remark}
We will define the notion of a tensored $\cV$-functor between tensored $\cV$-categories in \S\ref{sec:TensoredVFunctors} after we establish the bijective correspondence between $\cV$-categories and $\cV$-module categories in the next section.
We will then prove in Theorem \ref{thm:LiftToVAdjunction} that we may lift an adjunction of underlying functors $\cL^\cV \dashv \cR^\cV$ to a $\cV$-adjunction if and only if $\cL$ is tensored.
\end{remark}

\section{Equivalence between \texorpdfstring{$\cV$}{V}-categories and oplax right \texorpdfstring{$\cV$}{V}-modules}
\label{sec:V-cat and V-mod}

\subsection{\texorpdfstring{$\cV$}{V}-categories to oplax \texorpdfstring{$\cV$}{V}-modules}
\label{sec:V-cat to V-mod}

Suppose $\cC$ is a $\cV$-category.

\begin{defn}
\label{def:oplax-tensor}
We call $\cC$ \emph{oplax tensored} if for all $a\in \cC$, the functor $\cR_a : \cC^\cV \to \cV$ given by $b\mapsto \cC(a\to b)$ has a left adjoint $\cL_a: \cV \to \cC^\cV$.
\end{defn}
(Throughout \cite[\S 4.1]{1701.00567} we assumed our categories satisfied this property without naming it.)

When $\cC$ is oplax tensored, we can endow $\cC^\cV$ with the structure of an oplax right $\cV$-module category.
First, we define $a\vartriangleleft v = \cL_a(v)$ for $a\in \cC^\cV$ and $v\in \cV$, so that we have an adjunction
\begin{equation}
\label{eq:V-cat to V-mod adjunction}
\cC^\cV(a\vartriangleleft v \to b)
=
\cC^\cV(\cL_a(v) \to b)
\cong
\cV(v\to \cR_a(b))
=
\cV(v\to \cC(a\to b)).
\end{equation}
Now for a fixed $a\in \cC^\cV$, for all $v\in \cV$, the unit of the adjunction $\cL_a \dashv\cR_a$ is $\eta^\cC_{a,v} \in \cV(u \to \cC(a\to a\vartriangleleft u))$, which is given by the mate of the identity in $\cC^\cV(a\vartriangleleft u\to a\vartriangleleft u)$.

Clearly $\vartriangleleft$ is functorial in $\cV$ as $a\vartriangleleft v = \cL_a(v)$ is defined via a functor.
To show it is functorial in $\cC^\cV$, we use Adjunction \eqref{eq:V-cat to V-mod adjunction}.
First, given $f\in \cC^\cV(a\to b)$, we get a natural tranformation $\theta^f : \cR_b \Rightarrow \cR_a$ by precomposition with $f$.
We then define for $v\in \cV$ the morphism $f\vartriangleleft \id_v \in \cC^\cV(a\vartriangleleft v \to b\vartriangleleft v)$ to be the mate of
$$
v 
\xrightarrow{\eta^\cC_{b,v}} 
\cR_b(b\vartriangleleft v)
\xrightarrow{\theta^f_{b\vartriangleleft v}}
\cR_a(b\vartriangleleft v)
$$
under Adjunction \eqref{eq:V-cat to V-mod adjunction}.
It is now straightforward to show that $-\vartriangleleft-: \cC^\cV \otimes \cV \to \cC^\cV$ is a bifunctor.
Indeed, to verify the exchange relation
$
(f\vartriangleleft \id_u) \circ (\id_b \vartriangleleft g)
=
(\id_b \vartriangleleft g)\circ (f\vartriangleleft \id_u)
$
for $f\in \cC^\cV(a\to b)$ and $g\in \cV(u\to v)$, we take mates under Adjunction \eqref{eq:V-cat to V-mod adjunction} to see
\begin{align*}
\mate[(f\vartriangleleft \id_u) \circ (\id_b \vartriangleleft g)]
&=
\mate[(f\vartriangleleft \id_u) \circ \cL_b(g)]
=
\mate[f\vartriangleleft \id_u] \circ \cR_a\cL_b(g)
\\&=
\eta^\cC_{b,u} \circ \theta^f_{b\vartriangleleft u} \circ \cR_a\cL_b(g)
=
\eta^\cC_{b,u} \circ \cR_b\cL_b(g)\circ \theta^f_{b\vartriangleleft v}
=
g\circ \eta^\cC_{b,v} \circ \theta^f_{b\vartriangleleft v}
\\&=
\mate[\cL_a(g)\circ \mate[\eta^\cC_{b,v} \circ \theta^f_{b\vartriangleleft v}]]
=
\mate[(\id_a\vartriangleleft g)\circ (f\vartriangleleft \id_v)].
\end{align*}

We now show that $\cC^\cV$ is strongly unital.
For $a\in \cC^\cV$, we define the distinguished morphism $\rho_a$ to be the mate of $j_a = \id_{a}$ under the adjunction
$$
\cC^\cV(a\vartriangleleft 1_\cV \to a)
\cong
\cV(1_\cV \to \cC(a\to a))
=
\cC^\cV(a\to a).
$$

\begin{lem}
\label{lem:StrictlyUnital}
The morphism $\rho_a \in \cC^\cV(a\vartriangleleft 1_\cV \to a)$ is an isomorphism with inverse $\eta^\cC_{a,{1_\cV}}\in \cC^\cV(a \to a\vartriangleleft 1_\cV)$.
\end{lem}
\begin{proof}
Setting $v=1_\cV$ in Adjunction \eqref{eq:V-cat to V-mod adjunction}, we get a natural isomorphism $\cC^\cV(a\vartriangleleft 1_\cV \to b) \cong \cC^\cV(a\to b)$.
Thus by the Yoneda lemma as in \eqref{eq:YonedaInverses}, $a\vartriangleleft 1_\cV \cong a$, with explicit isomorphism given by the mate of $\id_a \in \cC^\cV(a\to a) \cong \cC^\cV(a\vartriangleleft 1_\cV \to a)$, i.e., $\rho_a$.
The inverse is given by the mate of $\id_{a\vartriangleleft 1_\cV} \in \cC^\cV(a\vartriangleleft 1_\cV \to a\vartriangleleft 1_\cV) \cong \cV(1_\cV \to \cC(a\to a\vartriangleleft v))$, i.e., $\eta_{a,{1_\cV}}^\cC$.
\end{proof}

We now define the oplaxitor
$$
\alpha_{a,u,v}
\in
\cC^\cV(a\vartriangleleft uv
\to
a\vartriangleleft u\vartriangleleft v)
\cong
\cV(uv \to \cC(a \to a\vartriangleleft u\vartriangleleft v))
$$
as the mate of
\begin{equation}
\label{eq:MateOfAlpha}
uv
\xrightarrow{\eta^\cC_{a,u} \eta^\cC_{a\vartriangleleft u,v}}
\cC(a \to a\vartriangleleft u)
\cC(a\vartriangleleft u \to a\vartriangleleft u\vartriangleleft v)
\xrightarrow{-\circ_\cC-}
\cC(a\to a\vartriangleleft u\vartriangleleft v).
\end{equation}
It is now straightforward to verify by taking mates of appropriate composites that $(\cC^\cV, \alpha,\rho)$ is a strongly unital oplax right $\cV$-module category.
We provide the proof that all $\alpha_{c,1_\cV, u}$ and $\alpha_{c,u,1_\cV}$ are invertible below.

\begin{lem}
\label{lem:AlphaInvertibleWhenOneArgumentIs1}
For all $c\in \cC$ and $u\in \cV$,
$\alpha^{-1}_{c,u,1_\cV} = \rho_{c\vartriangleleft u}$
and 
$\alpha^{-1}_{c, 1_\cV, u} = \rho_{c}\vartriangleleft \id_u$.
\end{lem}
\begin{proof}
First, under Adjunction \eqref{eq:V-cat to V-mod adjunction} setting $a=c\vartriangleleft u$, $v=1_\cV$, and $b = c\vartriangleleft u$, 
using Lemma \ref{lem:StrictlyUnital}, we have that the mate of $\alpha_{c,u,1_\cV} \circ \rho_{c\vartriangleleft u}$ is given by
$$
(\eta^\cC_{c,u} \eta^\cC_{c\vartriangleleft u,1_\cV} \rho_{c\vartriangleleft u})
\circ
(-\circ_\cC-\circ_\cC-)
=
\eta^\cC_{c,u},
$$
which is exactly the mate of $\id_{c\vartriangleleft u}$.
Second, under Adjunction \eqref{eq:V-cat to V-mod adjunction} setting $a=c\vartriangleleft 1_\cV$, $v=u$, and $b=c\vartriangleleft u$,
again using Lemma \ref{lem:StrictlyUnital}, we have that the mate of $\alpha_{c,1_\cV,u} \circ (\rho_{c} \vartriangleleft \id_u)$ is given by
$$
(\eta^\cC_{c,1_\cV} \eta^\cC_{c\vartriangleleft 1_\cV,u} (\rho_{c}\vartriangleleft \id_u))
\circ
(-\circ_\cC-\circ_\cC-)
=
(\eta^\cC_{c,1_\cV} \rho_{c} \eta^\cC_{c,u})
\circ
(-\circ_\cC-\circ_\cC-)
=
\eta^\cC_{c,u},
$$
which is exactly the mate of $\id_{c\vartriangleleft u}$.
Here, we used the identity
$$
(\eta^\cC_{c\vartriangleleft 1_\cV,u} (\rho_{c}\vartriangleleft \id_u))
\circ
(-\circ_\cC-)
=
(\rho_{c} \eta^\cC_{c,u})
\circ
(-\circ_\cC-),
$$
which is easily verified as the mate of each side is $\rho_c \vartriangleleft \id_u$ under the adjunction.
\end{proof}

\subsection{Oplax \texorpdfstring{$\cV$}{V}-modules to  \texorpdfstring{$\cV$}{V}-categories}
\label{sec:V-mod to V-cat}

Suppose $\cM$ is an oplax $\cV$-module category.
Similar to \cite[\S 6]{1701.00567}, we assume that for all $a\in \cM$, the functor $\cL_a: \cV \to \cM$ by $v\mapsto a\vartriangleleft v$ has a right adjoint $\cR_a : \cM \to \cV$.
We define a $\cV$-enriched category $\cC$ by
$\cC (a\to b) = \cR_a(b)$, so that we have an adjunction
\begin{equation}
\label{eq:V-mod to V-cat adjunction}
\cM(a \vartriangleleft v \to b)
=
\cM(\cL_a(v) \to b)
\cong
\cV(v \to \cR_a(b))
=
\cV(v\to \cC (a\to b)).
\end{equation}
We have for all $a,b\in \cM$ the evaluation map (counit morphism) $\varepsilon^\cC_{a\to b} \in \cM(a\vartriangleleft \cC(a\to b) \to b)$ given by the mate of the identity in $\cV(\cC (a\to b)\to \cC (a\to b))$.
We define our identity elements $j_a \in \cV(1_\cV \to \cC (a\to a))$ to be the mate of the distinguished morphism
$\rho_a\in\cM(a\vartriangleleft 1_\cV\to a)$.
We define the composition morphism
$$
-\circ_{\cC} - \in
\cV(\cC(a\to b)\cC(b\to c) \to \cC(a\to c))
\cong
\cM(a\vartriangleleft (\cC(a\to b)\cC(b\to c)) \to c)
$$
as the mate of
\begin{align*}
a\vartriangleleft \cC(a\to b)\cC(b\to c)
&\xrightarrow{\alpha_{a, \cC(a\to b), \cC(b\to c)}}
a\vartriangleleft \cC(a\to b)\vartriangleleft \cC(b\to c)
\\&\xrightarrow{\varepsilon^\cC_{a\to b}\id_{\cC(b\to c)}}
b\vartriangleleft \cC(b\to c)
\\&\xrightarrow{\varepsilon^\cC_{b\to c}}
c.
\end{align*}
It is straightforward to verify that $\cC$ is a $\cV$-category.
Indeed, the proof is entirely similar to \cite[\S 6.1 and 6.2]{1701.00567}.
We provide the proof below that the $j_a$ are identity elements in $\cC$ only assuming $\cM$ is oplax and not strongly unital.

\begin{lem}
The morphisms $j_a$ satisfy the identity axioms.
\end{lem}
\begin{proof}
We verify that $(j_a \id_{\cC(a\to b)})\circ (-\circ_\cC-) = \id_{\cC(a\to b)}$, and the other equation is similar.
The mate of $(j_a \id_{\cC(a\to b)})\circ (-\circ_\cC-)$ under the adjunction
$$
\cV(\cC(a\to b) \to \cC(a\to b)) \cong \cM(a \vartriangleleft \cC(a\to b) \to b)
$$
is given by
\begin{align*}
[\id_a\vartriangleleft \cF(j_a \id_{\cC(a\to b)})]
&\circ
\alpha_{a,\cC(a\to a), \cC(a\to b)}
\circ
[\varepsilon^\cC_{a\to a} \vartriangleleft \id_{\cC(a\to b)}]
\circ
\varepsilon^\cC_{a\to b}
\\&=
\alpha_{a,1_\cV, \cC(a\to b)}
\circ
[\id_a \vartriangleleft j_a \vartriangleleft \id_{\cC(a\to b)}]
\circ
[\varepsilon^\cC_{a\to a} \vartriangleleft \id_{\cC(a\to b)}]
\circ
\varepsilon^\cC_{a\to b}
\\&=
\alpha_{a,1_\cV, \cC(a\to b)}
\circ
[\rho_a \vartriangleleft \id_{\cC(a\to b)}]
\circ
\varepsilon^\cC_{a\to b}
\\&=
\varepsilon^\cC_{a\to b}
\end{align*}
which is exactly the mate of $\id_{\cC(a\to b)}$.
In the first equality above, we used naturality of $\alpha$, in the second we used $(\id_a \vartriangleleft j_a) \circ \varepsilon^\cC_{a\to a}= \rho_a$ (which is proven by taking mates), and in the third we used the unitality axiom for $\cM$.
\end{proof}

\subsection{Equivalence}
\label{sec:Equivalence}

We now restrict to the setting of oplax $\cV$-modules which are strongly unital, and prove Theorem \ref{thm:OplaxVMod}.
The proof is very close to \cite[\S 7]{1701.00567}.


Starting with a strongly unital oplax $\cV$-module $\cM$, we can construct the $\cV$-category $\cC$ as in Section \ref{sec:V-mod to V-cat}, and obtain the strongly unital oplax $\cV$-module $\cC^\cV$ as in Section \ref{sec:V-cat to V-mod}.
Since $\cC$ and $\cC^\cV$ have the same objects as $\cM$, we can define $\cF: \cM \to \cC^\cV$ to be the identity on objects.
On morphims, we define $\cF$ by the adjunction
\begin{equation}
\label{eq:V-mod Equivalence}
\cC^\cV(a \to b)
\cong
\cV(1_\cV \to \cC(a\to b))
\cong
\cM(a\vartriangleleft 1_\cV \to b)
\cong
\cM(a \to b)
\end{equation}
This is a functor similar to \cite[Prop.~6.12]{1701.00567}.
We define $\mu_{m,v}\in \cC^\cV(\cF(m \vartriangleleft v)\to \cF(m) \vartriangleleft v)=\cC^\cV(m\vartriangleleft v \to m\vartriangleleft v)$ to be $j_{m\vartriangleleft v}$.
It is easy to check $(\cF, \mu)$ is a functor of oplax $\cV$-modules.
Since $\cF$ is the identity on objects, $\cF$ is essentially surjective.
By \eqref{eq:V-mod Equivalence}, $\cF$ is clearly fully faithful.
Hence $(\cF, \mu)$ is an equivalence of oplax $\cV$-modules.

Conversely, starting with a $\cV$-category $\cC$, we endow $\cC^\cV$ with the structure of a strongly unital oplax $\cV$-module as in Section \ref{sec:V-cat to V-mod}, and we obtain a $\cV$-category $\cC'$ from $\cC^\cV$ as in Section \ref{sec:V-mod to V-cat}.
We define $\cV$-functors $\cG: \cC \to \cC'$ and $\cH: \cC'\to \cC$ as in \cite[Def.~7.5]{1701.00567}.
On objects, $\cG(c)=\cH(c)=c$, since $\cC$, $\cC^\cV$, and $\cC'$ all have the same objects.
We define $\cG_{a\to b}$ and $\cH_{a\to b}$ via adjunction:
$$
\xymatrix@R=2pt{
\cV(\cC(a\to b) \to \cC(a\to b))
\ar@{<->}[r]^(.53){\cong}_(.53){\eqref{eq:V-cat to V-mod adjunction}}
&
\cC^\cV(a \vartriangleleft \cC(a\to b) \to b)
\ar@{<->}[r]^(.47){\cong}_(.47){\eqref{eq:V-mod to V-cat adjunction}}
&
\cV(\cC(a\to b) \to \cC'(a\to b))
\\
\id_{\cC(a\to b)}
\ar@{<->}[r]
&
\varepsilon^\cC_{a\to b}
\ar@{<->}[r]
&
\cG_{a\to b}
\\
\cV(\cC'(a\to b) \to \cC'(a\to b))
\ar@{<->}[r]^(.53){\cong}_(.53){\eqref{eq:V-mod to V-cat adjunction}}
&
\cC^\cV(a \vartriangleleft \cC'(a\to b) \to b)
\ar@{<->}[r]^(.47){\cong}_(.47){\eqref{eq:V-cat to V-mod adjunction}}
&
\cV(\cC'(a\to b) \to \cC(a\to b))
\\
\id_{\cC'(a\to b)}
\ar@{<->}[r]
&
\varepsilon^{\cC'}_{a\to b}
\ar@{<->}[r]
&
\cH_{a\to b}
}
$$
Here, we write $\varepsilon^\cC$ for the counit of Adjunction \eqref{eq:V-cat to V-mod adjunction}, which should not be confused with the counit of Adjunction \eqref{eq:V-mod to V-cat adjunction} from Section \ref{sec:V-mod to V-cat}.
That $\cG$ and $\cH$ are $\cV$-functors which witness a $\cV$-equivalence of $\cV$-categories is identical to the proof of \cite[Thm.~7.4]{1701.00567}.
Indeed, to see $\cG$ and $\cH$ are mutually inverse, we use the Yoneda Lemma as in \eqref{eq:YonedaInverses}.
The adjunctions \eqref{eq:V-cat to V-mod adjunction} and \eqref{eq:V-mod to V-cat adjunction} give us a natural isomorphism between representable functors $\cV(-\to \cC(a\to b)) \cong \cV(- \to \cC'(a\to b))$:
$$
\cV(v \to \cC(a\to b))
\underset{\eqref{eq:V-cat to V-mod adjunction}}{\cong}
\cC^\cV(a \vartriangleleft v \to b)
\underset{\eqref{eq:V-mod to V-cat adjunction}}{\cong}
\cV(v \to \cC'(a\to b)).
$$

\section{\texorpdfstring{$\cV$}{V}-adjunctions and strong \texorpdfstring{$\cV$}{V}-modules}
\label{sec:Tensored V-cat and Strong V-mod}

In this section we assume $\cV$ is closed so that we may form $\widehat{\cV}$.
We now give a condition on a $\cV$-enriched category $\cC$ which corresponds to $\cC^\cV$ being a strong right $\cV$-module category.
We begin by defining the notion of \emph{tensored} $\cV$-functors and use them to prove a helpful result about promoting underlying adjunctions to $\cV$-adjunctions. 

\subsection{Tensored \texorpdfstring{$\cV$}{V}-functors}
\label{sec:TensoredVFunctors}

Throughout this subsection $\cC$ is a $\cV$-category such that for all $a\in\cC$, the functor $\cR_a: \cC^\cV \to \cV$ given by $b\mapsto \cC(a\to b)$ has a left adjoint $\cL_a$ so that $\cC^\cV$ is a strongly unital right oplax $\cV$-module by Theorem \ref{thm:OplaxVMod}, and similarly for $\cD$.

\begin{lem}
\label{lem:UnderlyingModuleFunctor}
If $\cF: \cC \to \cD$ is a $\cV$-functor, the underlying functor $\cF^\cV: \cC^\cV \to \cD^\cV$ can be canonically endowed with the structure of a strongly unital lax $\cV$-module functor by defining, 
for $c\in\cC^\cV$ and $v\in \cV$, 
\begin{align*}
\mu_{c,v}\in \cD^\cV(\cF(c)\vartriangleleft v \to \cF(c\vartriangleleft v)) 
&= 
\cV(1_\cV \to \cD(\cF(c)\vartriangleleft v \to \cF(c\vartriangleleft v)))
\\&\cong
\cV(v \to \cD(\cF(c) \to \cF(c\vartriangleleft v)))
\end{align*}
as the  mate of $\eta^\cC_{c,v} \circ \cF_{c \to c\vartriangleleft v}$, where $\eta^\cC_{c,v} \in \cV(v \to \cC(c \to c\vartriangleleft v)) \cong \cC^\cV(c\vartriangleleft v\to c\vartriangleleft v)$ is the unit of the adjunction.
\end{lem}
\begin{proof}
We first note that for $c\in\cC$, $\mu_{c, 1_\cV}$ is the mate of $\cF^\cV(\eta^\cC_{c,1_\cV}) =\eta^\cC_{c,1_\cV} \circ \cF_{c \to c\vartriangleleft 1_\cV}\in \cD^\cV(\cF(c) \to \cF(c)\vartriangleleft 1_\cV))$.
Thus the mate of $\mu_{c, 1_\cV}\circ \cF^\cV(\rho^\cC_c)$ under adjunction $\cD^\cV(\cF(c)\vartriangleleft 1_\cV \to \cF(c)) \cong \cV(1_\cV\to \cD(\cF(c)\to \cF(c)))$ is given by
\begin{align*}
[(\eta^\cC_{c,1_\cV} \circ \cF_{c\to c\vartriangleleft 1_\cV})(\rho^\cC_c \circ \cF_{c\vartriangleleft 1_\cV \to c})]\circ (-\circ_\cD -)
&=
(\eta^\cC_{c,1_\cV} \rho^\cC_c)\circ (\cF_{c\to c\vartriangleleft 1_\cV}\cF_{c\vartriangleleft 1_\cV \to c})\circ (-\circ_\cD -)
\\&=
(\eta^\cC_{c,1_\cV} \rho^\cC_c)\circ (-\circ_\cC -)\circ \cF_{c\to c}.
\\&=
j_c \circ \cF_{c\to c} 
\\&= 
j_{\cF(c)}
\end{align*}
where we used Lemma \ref{lem:StrictlyUnital} to conclude $(\eta^\cC_{c,1_\cV} \rho^\cC_c)\circ (-\circ_\cC -) = j_c$ in the third equality.
Now $j_{\cF(c)}$ is exactly the mate of $\rho^\cD_{\cF(c)}$ under the adjunction, and thus the unital axiom holds.

We now prove the associative condition.
The mate of $\alpha_{\cF(c), u, v} \circ (\mu_{c,u}\vartriangleleft \id_v) \circ \mu_{c\vartriangleleft u, v}$ under the adjunction 
$$
\cD^\cV(\cF(c) \vartriangleleft uv \to \cF(c \vartriangleleft u\vartriangleleft v))
\cong
\cV(uv \to \cD(\cF(c) \to \cF(c \vartriangleleft u\vartriangleleft v)))
$$
is given by
\begin{align*}
(\eta^\cD_{\cF(c),u} \eta^\cD_{\cF(c)\vartriangleleft u,v})
&\circ 
(-\circ_\cD - ) 
\circ 
(\id_{\cD(\cF(c)\to \cF(c)\vartriangleleft u \vartriangleleft v)} (\mu_{c,u}\vartriangleleft \id_v))
\circ 
(-\circ_\cD - ) 
\circ 
(\id_{\cD(\cF(c)\to \cF(c\vartriangleleft u) \vartriangleleft v)} \mu_{c,v})
\circ 
(-\circ_\cD - ) 
\\&=
(\eta^\cD_{\cF(c),u}[[\eta^\cD_{\cF(c)\vartriangleleft u,v} (\mu_{c,u}\vartriangleleft \id_v)]\circ (-\circ_\cD -)])
\circ 
(-\circ_\cD - ) 
\circ 
(\id_{\cD(\cF(c)\to \cF(c\vartriangleleft u) \vartriangleleft v)} \mu_{c,v})
\circ 
(-\circ_\cD - ) 
\\&=
(\eta^\cD_{\cF(c),u}[(\mu_{c,u}\eta^\cD_{\cF(c),v})\circ (-\circ_\cD -)])
\circ 
(-\circ_\cD - ) 
\circ 
(\id_{\cD(\cF(c)\to \cF(c\vartriangleleft u) \vartriangleleft v)} \mu_{c,v})
\circ 
(-\circ_\cD - ) 
\\&=
[(\eta^\cD_{\cF(c),u}\mu_{c,u})(\eta^\cD_{\cF(c),v}\mu_{c,v})]
\circ 
[(-\circ_\cD -)(-\circ_\cD -)]
\circ 
(-\circ_\cD - ) 
\\&=
[(\eta^\cC_{c,u}\circ \cF_{c \to c\vartriangleleft u})(\eta^\cD_{\cF(c),v}\cF_{c \to c\vartriangleleft v})]
\circ
(-\circ_\cD - ) 
\\&=
(\eta^\cC_{c,u}\eta^\cC_{c,v})
\circ
(-\circ_\cC - ) 
\circ
\cF_{c \to c\vartriangleleft u\vartriangleleft v}
\\&=
\eta^\cC_{c,uv}
\circ
(\id_{\cC(c \to c\vartriangleleft uv)} \alpha_{c,u,v}^\cC)
\circ
(-\circ_\cC - ) 
\circ
\cF_{c \to c\vartriangleleft u\vartriangleleft v}
\\&=
(\eta^\cC_{c,uv}\circ \cF_{c \to c\vartriangleleft uv})
\circ
(\id_{\cC(c \to c\vartriangleleft uv)} (\alpha^\cC_{c,u,v}\circ \cF_{c\vartriangleleft uv \to c \vartriangleleft u\vartriangleleft v}))
\circ
(-\circ_\cC - ) 
\end{align*}
which is exactly the mate of $\mu_{c,uv}\circ \cF^\cV(\alpha_{c,u,v})$ under the above adjunction.

To verify $\mu_{a,u}$ is natural in $a\in\cC$ and $u\in \cV$, suppose $f\in \cC^\cV(a\to b)$ and $g\in \cV(u\to v)$.
Then the mate of $\mu_{a,u} \circ \cF(f \vartriangleleft g)$ under the adjunction
$$
\cD^\cV(\cF(a)\vartriangleleft u \to \cF(b\vartriangleleft v)) \cong \cV(u \to \cD(\cF(a) \to \cF(b\vartriangleleft v)))
$$
is given by
\begin{align*}
(\eta^\cC_{a,u}  (f\vartriangleleft g))
&\circ 
(\cF_{a\to a\vartriangleleft u} \cF_{a\vartriangleleft u \to b \vartriangleleft v})
\circ
(-\circ_\cD-)
\\&=
(\eta^\cC_{a,u}  (f\vartriangleleft g))
\circ
(-\circ_\cC-)
\circ 
\cF_{a\to b \vartriangleleft v}
\\&=
(f( g\circ \eta^\cC_{b,v}))
\circ
(-\circ_\cC-)
\circ 
\cF_{a\to b \vartriangleleft v}
\\&=
(f (g\circ \eta^\cC_{b,v}))
\circ
(\cF_{a\to b \vartriangleleft v} \cF_{a\to b \vartriangleleft v})
\circ 
(-\circ_\cD-)
\\&=
[(f\circ\cF_{a\to b}) (g\circ (\eta^\cD_{\cF(b),v} \mu_{b,v})\circ (-\circ_\cD-))]
\circ
(-\circ_\cD-)
\\&=
(f (g\circ \eta^\cC_{b,v}))
\circ
(\cF_{a\to b \vartriangleleft v} \cF_{a\to b \vartriangleleft v})
\circ 
(-\circ_\cD-)
\\&=
((f\circ\cF_{a\to b}) (g\circ \eta^\cD_{\cF(b),v})
\circ
(-\circ_\cD-)
\circ
(\id_{\cD(a\to \cF(b)\vartriangleleft v)}\mu_{b,v})
\circ
(-\circ_\cD-)
\end{align*}
which is exactly the mate of $(\cF(f)\vartriangleleft g) \circ \mu_{b,v}$ under the above adjunction.
\end{proof}

\begin{defn}
\label{defn:TensoredVFunctor}
We call a $\cV$-functor $\cF: \cC \to \cD$ \emph{tensored} if the canonical maps $\mu_{c,v}\in \cD(\cF(c)\vartriangleleft v \to \cF(c\vartriangleleft v))$ from Lemma \ref{lem:UnderlyingModuleFunctor} are isomorphisms for all $c\in \cC$ and $v\in\cV$.
\end{defn}

We now prove a helpful result on lifting underlying adjunctions to $\cV$-adjunctions.
Suppose $\cL: \cC \to \cD$ and $\cR: \cD \to \cC$ are $\cV$-functors such that $\cL^\cV \dashv \cR^\cV$.
Let $\eta^\cV: \id_{\cC^\cV} \Rightarrow \cL^\cV\circ\cR^\cV$ and $\varepsilon^\cV: \cR^\cV \circ \cL^\cV \Rightarrow \id_{\cD^\cV}$ be the unit and counit of the underlying adjunction respectively.
Recall that for $a\in \cC$ and $d\in \cD$, we defined 
$\theta_{a,d}=(\eta^\cV_{a}\cR_{\cL(a)\to d} )\circ (-\circ_\cC-)$
and
$\kappa_{a,d}= (\cL_{a\to \cR(d)} \varepsilon^\cV_{d})\circ (-\circ_\cD-)$.
Lemma \ref{lem:PromoteToVAdjunction} says that $\cL\dashv_\cV \cR$ if $\theta_{a,d}^{-1} = \kappa_{a,d}$ for all $a\in\cC$ and $d\in\cD$.

\begin{lem}
\label{lem:MateOfKappa}
For $a\in \cC$ and $d\in \cD$, $\kappa_{a,d}$ is equal to the mate of $\mu^\cL_{a, \cC(a\to \cR(d))} \circ \cL(\varepsilon^\cC_{a\to \cR(d)})\circ \varepsilon^\cV_d$ under the adjunction
$$
\cD^\cV(\cL(a)\vartriangleleft \cC(a\to \cR(d)) \to d)
\cong
\cV( \cC(a\to \cR(d)) \to \cD(\cL(a) \to d)).
$$
\end{lem}
\begin{proof}
First, for all $a,b\in \cC$,
\begin{equation}
\label{eq:MateOfKappaProofStep1}
\mu^{\cL}_{a,\cC(a\to b)} \circ \cL^\cV(\varepsilon^\cC_{a\to b}) 
= 
(\id_{\cL(a)} \vartriangleleft \cL_{a\to b}) \circ \varepsilon^\cD_{\cL(a) \to \cL(d)},
\end{equation}
since both are the mate of $\cL_{a\to b}$ under the adjunction
$$
\cD^\cV(\cL(a) \vartriangleleft \cC(a\to b) \to \cL(b))
\cong
\cV(\cC(a \to b) \to \cD(\cL(a) \to \cL(b))).
$$
Next, by naturality of $\varepsilon^\cD$, we have
\begin{equation}
\label{eq:MateOfKappaProofStep2}
\varepsilon^\cD_{c\to \cL(\cR(d))}\circ \varepsilon^\cV_d 
= 
(\id_c \vartriangleleft [(\id_{\cD(c\to \cL(\cR(d))} \varepsilon^\cV_d)\circ (-\circ_\cD-)]) \circ \varepsilon^\cD_{c \to d}
\end{equation}
Now we have that the mate of $\kappa_{a,d}$ under the adjunction is given by
\begin{align*}
(\id_{\cL(a)} \vartriangleleft [(\cL_{a\to \cR(d)} \varepsilon^\cV_d) \circ (-\circ_\cD-)]) 
\circ
\varepsilon^\cD_{\cL(a) \to d}
&
\underset{\text{\eqref{eq:MateOfKappaProofStep2}}}{=}
(\id_{\cL(a)}\vartriangleleft \cL_{a \to \cR(d)})
\circ
\varepsilon^\cD_{\cL(a) \to \cL\cR(d)}
\circ 
\varepsilon^\cV_d
\\&
\underset{\text{\eqref{eq:MateOfKappaProofStep1}}}{=}
\mu^\cL_{a, \cC(a\to \cR(d))}
\circ
\cL^\cV(\varepsilon^\cC_{a\to \cR(d)})
\circ 
\varepsilon^\cV_d.
\qedhere
\end{align*}
\end{proof}

\begin{lem}
\label{lem:MateOfTheta}
If $\cL$ is tensored, then for $a\in \cC$ and $d\in \cD$, $\theta_{a,d}$ is the mate of $\eta^\cV_{a \vartriangleleft  \cD(\cL(a) \to d)} \circ \cR(\mu^\cL_{a, \cD(\cL(a) \to d)})^{-1}\circ \cR(\varepsilon^\cD_{\cL(a) \to d})$ under the adjunction
$$
\cC^\cV(a\vartriangleleft \cD(\cL(a) \to d) \to \cR(d))
\cong
\cV(\cD(\cL(a) \to d) \to \cC(a\to \cR(d))).
$$
\end{lem}
\begin{proof}
Similar to the beginning of the proof of Lemma \ref{lem:MateOfKappa}, for all $c,d\in \cD$, 
\begin{equation}
\label{eq:MateOfThetaProofStep2}
(\id_{\cR(c)} \vartriangleleft \cR_{c\to d}) \circ \varepsilon^\cC_{\cR(c) \to \cR(d)} = \mu^\cR_{c, \cD(c\to d)}\circ \cR(\varepsilon^\cD_{c\to d}),
\end{equation}
since both are the mate of $\cR_{c\to d}$ under the adjunction
$$
\cC^\cV(\cR(c) \vartriangleleft \cD(c\to d) \to \cR(d))
\cong
\cV(\cD(c \to d) \to \cC(\cR(c) \to \cD(d))).
$$
Next, for all $a\in \cC$ and $v\in \cV$, 
\begin{equation}
\label{eq:MateOfThetaProofStep1}
(\eta^\cV_a\vartriangleleft \id_v) \circ \mu^{\cR}_{\cL(a), v}\circ \cR(\mu^\cL_{a,v}) = \eta^\cV_{a\vartriangleleft v}.
\end{equation}
Indeed, under the adjunction
$$
\cC^\cV(a\vartriangleleft v \to \cR\cL(a\vartriangleleft v)) \cong \cV(v \to \cC(a\to \cR\cL(a\vartriangleleft v))),
$$
the mate of the left hand side of \eqref{eq:MateOfThetaProofStep1} is equal to
\begin{align*}
[\eta^\cV_{a}(\eta_{v,\cL(a)}^{\cC} \circ \cR_{\cL(a)\to \cL(a)\vartriangleleft v})&(\mu^{\cL}_{a,v} \circ \cR_{\cL(a)\vartriangleleft v \to \cL(a\vartriangleleft v)})] 
\circ 
(-\circ_\cC-\circ_\cC-)
\\&=
(\eta^\cV_{a}[(\eta_{v,\cL(a)}^{\cC}\mu^{\cL}_{a,v})\circ (-\circ_\cD - )])
\circ
(\id_{\cC(a\to \cR(\cL(a)))}\cR_{\cL(a) \to \cL(a\vartriangleleft v)})
\circ
(-\circ_\cC-)
\\&=
(\eta^\cV_a [\eta^\cC_{a,v} \circ \cL_{a \to a\vartriangleleft v} \circ \cR_{\cL(a) \to \cL(a\vartriangleleft v)}])
\circ
(-\circ_\cC-)
\\&=
(\eta^\cC_{a,v} \eta^\cV_{a\vartriangleleft v})
\circ
(-\circ_\cC-)
\end{align*}
which is exactly the mate of the right hand side of \eqref{eq:MateOfThetaProofStep1}.
Finally, the mate of $\theta_{a,d} = (\eta^\cV_a \cR_{\cL(a) \to d})\circ (-\circ_\cC-)$ under the adjunction is given by
\begin{align*}
(\id_{a} \vartriangleleft (\eta^\cV_a \cR_{\cL(a) \to d}))
&\circ 
\alpha^\cC_{a,\cC(a \to \cR(\cL(a))),\cC(\cR(\cL(a)) \to \cR(d))}
\circ
(\varepsilon^\cC_{a\to \cR(\cL(a))} \vartriangleleft \id_{\cC(\cR(\cL(a)) \to \cR(d))})
\circ 
\varepsilon^\cC_{\cR(\cL(a)) \to \cR(d)}
\\&\hspace{.15cm}=
(\eta^\cV_a\vartriangleleft \cR_{\cL(a) \to d})\circ \varepsilon^\cC_{\cR(\cL(a)) \to \cR(d)}
\\&\underset{\text{\eqref{eq:MateOfThetaProofStep2}}}{=}
(\eta^\cV_a\vartriangleleft \id_{\cD(\cL(a) \to d)})\circ \mu^\cR_{\cL(a), \cD(\cL(a) \to d)} \circ \cR(\varepsilon^\cD_{\cL(a) \to d})
\\&\underset{\text{\eqref{eq:MateOfThetaProofStep1}}}{=}
\eta^\cV_{a \vartriangleleft  \cD(\cL(a) \to d)} \circ \cR(\mu^\cL_{a, \cD(\cL(a) \to d)})^{-1}\circ \cR(\varepsilon^\cD_{\cL(a) \to d}).
\qedhere
\end{align*}
\end{proof}

\begin{thm}
\label{thm:LiftToVAdjunction}
We have $\cL\dashv_\cV \cR$ if and only if $\cL$ is tensored.
\end{thm}
\begin{proof}
Suppose $\cL\dashv_\cV \cR$.
Then for all $a\in \cC$ and $v\in \cV$, we have the following isomorphisms, which are easily seen to be natural in $d\in \cD^\cV$ by construction:
\begin{align*}
\cD^\cV(\cL(a\vartriangleleft v) \to d)
&\cong
\cC^\cV(a\vartriangleleft v \to \cR(d))
\cong
\cV(v \to \cC(a\to \cR(d)))
\\&\cong
\cV(v \to \cD(\cL(a) \to d))
\cong
\cD^\cV(\cL(a)\vartriangleleft v \to d).
\end{align*}
Thus we get a natural isomorphism of representable functors $\cD^\cV(\cL(a)\vartriangleleft v \to - ) \cong \cD^\cV(\cL(a\vartriangleleft v) \to - )$.
Setting $d=\cL(a\vartriangleleft v)$, by the Yoneda Lemma as in Equation \eqref{eq:YonedaInverses}, the mate of $\id_{\cL(a\vartriangleleft v)}$ under the above series of isomorphisms gives a canonical isomorphism in $\cD^\cV(\cL(a)\vartriangleleft v \to \cL(a\vartriangleleft v))$.
We now see that under the above isomorphisms, $\id_{\cL(a\vartriangleleft v)}$ transforms as follows:
\begin{align*}
\id_{\cL(a\vartriangleleft v)}
&\leftrightarrow
\eta^\cV_{a\vartriangleleft v}
\\&\leftrightarrow
(\eta^\cC_{a,v} \eta^\cV_{a\vartriangleleft v}) \circ (-\circ_\cC-)
\\&\leftrightarrow
(\eta^\cC_{a,v} \eta^\cV_{a\vartriangleleft v}) \circ (-\circ_\cC-)\circ \cL_{a \to \cR(\cL(a\vartriangleleft v))}\circ (\id \varepsilon^\cV_{\cL(a\vartriangleleft v)}) \circ (-\circ_\cD -)
\\
&\hspace{1cm}=
[(\eta^\cC_{a,v}\circ \cL_{a\to a\vartriangleleft v})( (\cL^\cV(\eta^\cV_{a\vartriangleleft v} \varepsilon^\cV_{\cL(a\vartriangleleft v)}))\circ(-\circ_\cD-))] \circ (-\circ_\cD-)
=
\eta^\cC_{a,v}\circ \cL_{a\to a\vartriangleleft v}
\\&\leftrightarrow 
\mu^\cL_{a,v}.
\end{align*}
So $\mu^\cL_{a,v}$ is an isomorphism for all $a\in\cC$ and $v\in \cV$, and thus $\cL$ is tensored.

Conversely, suppose $\cL$ is tensored.
For $v\in \cV$, $a\in \cC$, and $d\in\cD$, using Adjunction \eqref{eq:V-cat to V-mod adjunction}, we get the following isomorphisms, which are natural in $v\in\cV$ by construction:
\begin{equation}
\label{eq: Lift To V-Adjunction Natural Isos}
\begin{split}
\cV(v \to \cD(\cL(a) \to d))
&\cong
\cD^\cV(\cL(a)\vartriangleleft v \to d)
\\&\cong
\cD^\cV(\cL(a\vartriangleleft v) \to d)
\\&\cong
\cC^\cV(a\vartriangleleft v \to \cR(d))
\\&\cong
\cV(v \to \cC(a\to \cR(d))).
\end{split}
\end{equation}
Setting $v = \cC(a\to \cR(d))$, by Lemma \ref{lem:MateOfKappa}, $\id_{\cC(a\to \cR(d))}$ transforms as follows under the above isomorphisms:
\begin{align*}
\kappa_{a,d} 
&\leftrightarrow 
\mu^\cL_{a, \cC(a\to \cR(d))} \circ \cL(\varepsilon^\cC_{a\to \cR(d)})\circ \varepsilon^\cV_d
\leftrightarrow
\cL(\varepsilon^\cC_{a\to \cR(d)})\circ \varepsilon^\cV_d
\leftrightarrow
\varepsilon^\cC_{a\to \cR(d)}
\leftrightarrow
\id_{\cC(a\to \cR(d))}.
\end{align*}
Likewise, setting $v = \cD(\cL(a) \to d)$,  by Lemma \ref{lem:MateOfTheta}, $\id_{\cD(\cL(a) \to d)}$ transforms as follows under the above isomorphisms:
\begin{align*}
\id_{\cD(\cL(a) \to d)}
&\leftrightarrow 
\varepsilon^\cD_{\cL(a) \to d}
\\&\leftrightarrow
(\mu^\cL_{a, \cD(\cL(a) \to d)})^{-1}\circ \varepsilon^\cD_{\cL(a) \to d}
\\&\leftrightarrow 
\eta^\cV_{a \vartriangleleft  \cD(\cL(a) \to d)} \circ \cR(\mu^\cL_{a, \cD(\cL(a) \to d)})^{-1}\circ \cR(\varepsilon^\cD_{\cL(a) \to d})
\\&\leftrightarrow 
\theta_{a,d}
\end{align*}
As Equation \eqref{eq: Lift To V-Adjunction Natural Isos} is a string of isomorphisms, we have $\theta_{a,d} = \kappa_{a,d}^{-1}$ for all $a\in \cC$ and $d\in \cD$ by the Yoneda Lemma as in Equation \eqref{eq:YonedaInverses}.
By Lemma \ref{lem:PromoteToVAdjunction}, we have $\cL\dashv_\cV \cR$.
\end{proof}

Thus Theorem \ref{thm:LiftToVAdjunction} gives a necessary and sufficient condition to lift the underlying adjunction to a $\cV$-adjunction under the assumption that $\cC$ and $\cD$ are oplax tensored.
By strengthening the hypothesis to $\cC$ and $\cD$ being tensored, we obtain the following corollary.

\begin{cor}
\label{cor:LiftToVAdjunction}
Suppose $\cC$ and $\cD$ are tensored and $\cL: \cC \to \cD$ and $\cR: \cD \to \cL$ are $\cV$-functors such that $\cL^\cV \dashv_\cV \cR^\cV$.
Then $\cL\dashv_\cV \cR$ if and only if $\cL$ is tensored.
\end{cor}


\subsection{\texorpdfstring{$\cV$}{V}-adjunctions to strong \texorpdfstring{$\cV$}{V}-modules}

Suppose $\cC$ is a $\cV$-category.
In Section \ref{sec:V-cat to V-mod}, we looked at the functors $\cR_a: \cC^\cV \to \cV$ given by $b\mapsto \cC(a\to b)$.
It is important to note that $\cR_a$ can be promoted to a $\cV$-functor $\cR^a:\cC \to \widehat{\cV}$ as in Example \ref{ex:VRepresentable}.

We assume now that $\cC$ is tensored, so that each $\cR^a$ admits a left $\cV$-adjoint $\cL^a: \widehat{\cV} \to \cC$.
As before, we set $a\vartriangleleft v =\cL^a(v)$ so that we have a $\cV$-adjunction
\begin{equation}
\label{eq:V-Adjunction}
\cC(a\vartriangleleft v \to b)
=
\cC(\cL^a(v)\to b)
\cong
\widehat{\cV}(v \to \cR^a(b))
=
\widehat{\cV}(v \to \cC(a\to b)).
\end{equation}
We endow $\cC^\cV$ with the structure of an oplax right $\cV$-module as in Section \ref{sec:V-cat to V-mod}.
We claim now that $\cC^\cV$ is actually a strong $\cV$-module, i.e., the morphisms $\alpha_{a,u,v}\in \cC^\cV(a\vartriangleleft uv\to a\vartriangleleft u\vartriangleleft v)$ are isomorphisms.
Note that since $\cL^a \dashv_\cV \cR^a$, by Corollary \ref{cor:LiftToVAdjunction}, $\cL^a: \widehat{\cV} \to \cC$ is tensored.

\begin{lem}
Under the identification $\widehat{\cV}^\cV=\cV$ in Example \ref{ex:UnderlyingCategoryOfVhat}, the right $\cV$-module structure of $\cV$ is given by $u\vartriangleleft v: = uv$.
\end{lem}
\begin{proof}
Recall that $u \vartriangleleft v = \cL_u(v)$ where $\cL_u= (u\otimes -)^\cV$ and $(u\otimes -)$ is a left $\cV$-adjoint of $\widehat{\cV}(u \to -)$ from Example \ref{example:InnerHomVAdjunction}.
Now on objects, $(u\otimes -)^\cV(v) = uv$, and for a morphism $f\in \cV(v\to w)$, $(u\otimes f)^\cV \in \cV(uv \to uw)$ is equal to $\id_u f$.
\end{proof}

Thus for all $a\in\cC$ and $u,v\in\cV$, 
$$
\mu^{\cL^a}_{u,v}\in 
\cC^\cV(\cL^a(u)\vartriangleleft v\to \cL^a(u\vartriangleleft v))
=
\cC^\cV(a\vartriangleleft u \vartriangleleft v \to a \vartriangleleft uv)
$$
is an isomorphism.

\begin{prop}
\label{prop:AlphaIsInvertible}
For all $a\in \cC$ and $u,v\in\cV$, $\alpha_{a,u,v}=(\mu^{\cL^a}_{u,v})^{-1}$.
Since $\mu^{\cL^a}_{u,v}$ is an isomorphism, so is $\alpha_{a,u,v}$.
\end{prop}
\begin{proof}
For $a,b\in\cC$ and $u,v \in \cV$, the following isomorphisms are natural in $b\in \cC^\cV$ by construction.
\begin{align*}
\cC^\cV(a\vartriangleleft u\vartriangleleft v \to b)
&\cong
\cV(v \to \cC(a\vartriangleleft u \to b))
\\&\cong
\cV(v \to \widehat{\cV}(u \to \cC(a\to b)))
\\&\cong
\cV(uv \to \cC(a\to b))
\\&\cong
\cC^\cV(a\vartriangleleft uv \to b)
\end{align*}
Setting $b= a\vartriangleleft uv$ and transforming $\id_{a\vartriangleleft uv}$ under the series of isomorphisms yields
\begin{align*}
\mu^{\cL^a}_{u,v}
&\leftrightarrow
\eta^{\widehat{\cV}}_{u,v} \circ \cL^a_{u\to uv}
=
\eta^{\widehat{\cV}}_{u,v}\circ(\id_{\widehat{\cV}(u \to uv)}\eta^\cC_{a,uv})\circ (-\circ_{\widehat{\cV}}-)\circ \kappa_{u, a\vartriangleleft uv}
\\&\leftrightarrow
\eta^{\widehat{\cV}}_{u,v}\circ(\id_{\widehat{\cV}(u \to uv)}\eta^\cC_{a,uv})\circ (-\circ_{\widehat{\cV}}-)
\\&\leftrightarrow
\eta^\cC_{a,uv}
\\&\leftrightarrow
\id_{a\vartriangleleft uv}.
\end{align*}
Note we used the expression for $\kappa$ in terms of $\cL^a$ and the unit $\varepsilon^\cC$ of the underlying adjunction from \eqref{eq:ThetaKappa}.
Now setting $b=a\vartriangleleft u\vartriangleleft v$ and transforming $\id_{a\vartriangleleft u\vartriangleleft v}$ under the series of isomorphisms yields
\begin{align*}
\id_{a\vartriangleleft u\vartriangleleft v}
&\leftrightarrow
\eta^\cC_{a\vartriangleleft u,v}
\\&\leftrightarrow
\eta^\cC_{a\vartriangleleft u,v} \circ \theta_{u, a\vartriangleleft u \vartriangleleft v}
\\&\leftrightarrow
[\id_u (\eta^\cC_{a\vartriangleleft u,v} \circ \theta_{u, a\vartriangleleft u \vartriangleleft v})]
\circ 
\varepsilon^{\widehat{\cV}}_{u \to \cC(a \to a\vartriangleleft u \vartriangleleft v)}
\\&\leftrightarrow
[\id_a\vartriangleleft ([\id_u (\eta^\cC_{a\vartriangleleft u,v} \circ \theta_{u, a\vartriangleleft u \vartriangleleft v})]
\circ 
\varepsilon^{\widehat{\cV}}_{u \to \cC(a \to a\vartriangleleft u \vartriangleleft v)})]
\circ 
\varepsilon^\cC_{a\to a\vartriangleleft u \vartriangleleft v}
\\&\qquad =
[\id_a\vartriangleleft ([(\eta^\cC_{a,u}\eta^\cC_{a\vartriangleleft u,v}) \circ (\id_u \cR_{a\vartriangleleft u \to a\vartriangleleft u \vartriangleleft v})]
\circ 
\varepsilon^{\widehat{\cV}}_{\cC(a\to a\vartriangleleft u) \to \cC(a \to a\vartriangleleft u \vartriangleleft v)})]
\circ 
\varepsilon^\cC_{a\to a\vartriangleleft u \vartriangleleft v}
\\&\qquad=
(\id_a\vartriangleleft [(\eta^\cC_{a,u}\eta^\cC_{a\vartriangleleft u,v}) \circ (-\circ_\cC-)])
\circ 
\varepsilon^\cC_{a\to a\vartriangleleft u \vartriangleleft v}
\\&\qquad=
\alpha_{a,u,v}.
\end{align*}
Note we used the expression for $\theta$ in terms of $\cR^a$ and $\eta^\cC$ from \eqref{eq:ThetaKappa}, together with Example \ref{ex:VRepresentable}.
We are now finished by the Yoneda Lemma as in \eqref{eq:YonedaInverses}.
\end{proof}

\subsection{Strong \texorpdfstring{$\cV$}{V}-modules to \texorpdfstring{$\cV$}{V}-adjunctions}

Now suppose $\cM$ is a strong right $\cV$-module category.
As in Section \ref{sec:V-mod to V-cat}, we assume the functors $\cL_a: \cV \to \cM$ given by $v\mapsto m\vartriangleleft v$ have right adjoints, and we use the adjunctions $\cL_a \dashv \cR_a$ to construct a $\cV$-enriched category $\cC$.

We now show that when $\cM$ is strong, the functors $\cL_a: \cV \to \cM$ can be promoted to $\cV$-functors $\cL^a : \widehat{\cV} \to \cC$.
Indeed, we define $\cL^a_{u\to v}$ in
\begin{equation}
\label{eq:L^aDefinition}
\cV(\widehat{\cV}(u\to v) \to \cC(a\vartriangleleft u \to a\vartriangleleft v))
\cong
\cM(a \vartriangleleft u \vartriangleleft \widehat{\cV}(u\to v) \to a\vartriangleleft v)
\end{equation}
to be the mate of $\alpha^{-1}_{a,u,\widehat{\cV}(u\to v)}\circ (\id_a\vartriangleleft \varepsilon^{\widehat{\cV}}_{u\to v})$, where the existence of $\alpha^{-1}$ requires that $\cM$ is strong.
Note that the usual calculations with mates imply
\begin{equation}
\label{eq:L^aMateIdentity}
\alpha^{-1}_{a,u,\widehat{\cV}(u\to v)}\circ (\id_a\vartriangleleft \varepsilon^{\widehat{\cV}}_{u\to v})
=
(\id_{a\vartriangleleft u} \vartriangleleft \cL^a_{u\to v})\circ\varepsilon^\cC_{a\vartriangleleft u\to a\vartriangleleft v}.
\end{equation}

\begin{lem}
We have $\cL^a$ is a $\cV$-functor.
Under the identification $\widehat{\cV}^\cV = \cV$ from Example \ref{ex:UnderlyingCategoryOfVhat}, the underlying functor of $\cL^a$ is $(\cL^a)^\cV=\cL_a$.
\end{lem}
\begin{proof}
We must show that for all $u,v,w\in\cV$, $(-\circ_{\widehat{\cV}}-)\circ \cL^a_{u\to w} =(\cL^a_{v\to w}\cL^a_{v\to w})\circ (-\circ_\cC-)$.
The mate of $(-\circ_{\widehat{\cV}}-)\circ \cL^a_{u\to w}$ under Adjunction \eqref{eq:L^aDefinition} is given by
\begin{align*}
(\id_{a \vartriangleleft u} \vartriangleleft (-\circ_{\widehat{\cV}}-))
\circ
\alpha^{-1}_{a,u,\widehat{\cV}(u\to w)}
\circ
(\id_a \vartriangleleft \varepsilon^{\widehat{\cV}}_{u\to w})
&=
\alpha^{-1}_{a,u,\widehat{\cV}(u\to v)\widehat{\cV}(v\to w)}
\circ
(\id_{a} \vartriangleleft [\id_u(-\circ_{\widehat{\cV}}-)])
\circ
(\id_a \vartriangleleft \varepsilon^{\widehat{\cV}}_{u\to w})
\\&=
\alpha^{-1}_{a,u,\widehat{\cV}(u\to v)\widehat{\cV}(v\to w)}
\circ
(\id_a \vartriangleleft (\varepsilon^{\widehat{\cV}}_{u\to v} \id_w))
\circ
(\id_a \vartriangleleft \varepsilon^{\widehat{\cV}}_{v\to w})
\end{align*}
On the other hand, the mate of $(\cL^a_{v\to w}\cL^a_{v\to w})\circ (-\circ_\cC-)$ under Adjunction \eqref{eq:L^aDefinition} is given by
\begin{align*}
(\id_{a\vartriangleleft u} &\vartriangleleft (\cL^a_{u\to v} \cL^a_{v\to w}))
\circ
\alpha_{a\vartriangleleft u,\cC(a\vartriangleleft u\to a\vartriangleleft v), \cC(a \vartriangleleft v\to a \vartriangleleft w)}
\circ
(\bar{\varepsilon}^\cC_{a\vartriangleleft u\to a\vartriangleleft v} \vartriangleleft \id_{\cC(a \vartriangleleft v\to a \vartriangleleft w)})
\circ
\bar{\varepsilon}^\cC_{a \vartriangleleft v\to a \vartriangleleft w}
\\&=
\alpha_{a\vartriangleleft u,\widehat{\cV}(u\to v),\widehat{\cV}(v\to w)}
\circ
(\id_{a\vartriangleleft u} \vartriangleleft \cL^a_{u\to v} \vartriangleleft\cL^a_{v\to w})
\circ
(\bar{\varepsilon}^\cC_{a\vartriangleleft u\to a\vartriangleleft v} \vartriangleleft \id_{\cC(a \vartriangleleft v\to a \vartriangleleft w)})
\circ
\bar{\varepsilon}^\cC_{a \vartriangleleft v\to a \vartriangleleft w}
\\&=
\alpha_{a\vartriangleleft u,\widehat{\cV}(u\to v),\widehat{\cV}(v\to w)}
\circ
(\alpha^{-1}_{a,u,\widehat{\cV}(u\to v)}\vartriangleleft \id_w)
\circ
(\id_a \vartriangleleft \varepsilon^{\widehat{\cV}}_{u\to v} \vartriangleleft \id_w)
\circ
\alpha^{-1}_{a,v,\widehat{\cV}(v\to w)}
\circ
(\id_a \vartriangleleft \varepsilon^{\widehat{\cV}}_{v\to w})
\\&=
\alpha^{-1}_{a,u,\widehat{\cV}(u\to v)\widehat{\cV}(v\to w)}
\circ
(\id_a \vartriangleleft (\varepsilon^{\widehat{\cV}}_{u\to v} \id_w))
\circ
(\id_a \vartriangleleft \varepsilon^{\widehat{\cV}}_{v\to w})
\end{align*}
The second equality used two instances of \eqref{eq:L^aMateIdentity}, and the last equality used the naturality and associativity of $\alpha$.
\end{proof}

Now since we defined $\cC$ by taking the same objects as $\cM$ and setting $\cC(a\to b)=\cR_a(b)$, we see that each $\cR_a$ can be promoted to a $\cV$-representable $\cV$-functor $\cR^a:\cC\to \widehat{\cV}$ whose underlying functor is $\cR_a$ by Lemma \ref{lem:UnderlyingFunctorOfVRepresentable}.

\begin{prop}
\label{prop:LaTensored}
For all $a\in \cC$ and $u,v\in\cV$, $\mu^{\cL^a}_{u,v} = \alpha_{a,u,v}^{-1}$.
Since $\alpha_{a,u,v}$ is an isomorphism, so is $\mu^{\cL^a}_{u,v}$.
\end{prop}
\begin{proof}
As $\mu^{\cL^a}_{u,v}$ is defined by taking mates, we have
\begin{align*}
\mu^{\cL^a}_{u,v}
&=
(\id_{a\vartriangleleft u} \vartriangleleft \eta^{\widehat{\cV}}_{u,v} \circ \cL^a_{u \to uv}) 
\circ 
\varepsilon^\cC_{a\vartriangleleft u \to a\vartriangleleft uv}
\\&=
(\id_{a\vartriangleleft u} \vartriangleleft \eta^{\widehat{\cV}}_{u,v}) 
\circ 
\alpha^{-1}_{a,u,uv} 
\circ 
(\id_{a\vartriangleleft u} \vartriangleleft \varepsilon^{\widehat{\cV}}_{u\to uv})
&&\text{(by \eqref{eq:L^aMateIdentity})}
\\&=
\alpha^{-1}_{a,u,v} 
\circ 
(\id_{a \vartriangleleft u} \vartriangleleft \eta^{\widehat{\cV}}_{u,v}) 
\circ 
(\id_a \vartriangleleft \varepsilon^{\widehat{\cV}}_{u\to uv})
&&\text{(by naturality)}
\\&=
\alpha^{-1}_{a,u,v}.
&&\qedhere
\end{align*}
\end{proof}

\begin{cor}
\label{cor:AlphaInvertibleImpliesCTensored}
The $\cV$-functor $\cL^a$ is a left $\cV$-adjoint for $\cR^a$.
Thus $\cC$ is tensored.
\end{cor}
\begin{proof}
Since $\cL_a \dashv \cR_a$ for all $a\in\cC$, by Theorem \ref{thm:LiftToVAdjunction}, it suffices to prove that each $\cL^a$ is tensored.
This is exactly the content of Proposition \ref{prop:LaTensored}.
\end{proof}

\section{Completion for \texorpdfstring{$\cV$}{V}-categories}
\label{sec:CompletionForVCats}

For this section, we assume $\cV$ is closed so that we may form $\widehat{\cV}$, and $\cC$ is oplax tensored.

\subsection{The completion operation}

Suppose $\cC$ is a $\cV$-category.

\begin{defn}
\label{defn:CompletionOfVCategory}
We define the completion $\overline{\cC}$ to be the $\cV$-category whose objects are of the form $a\blacktriangleleft u$ where $a\in \cC$ and $u\in \cV$,
and whose hom objects are given by $\overline{\cC}(a\blacktriangleleft u \to b \blacktriangleleft v) = \widehat{\cV}(u\to \cC(a\to b)v)$.
The identity element $j_{a\blacktriangleleft u}$ is the mate of $j_a \id_u$ under the adjunction
$$
\cV(1_\cV \to \overline{\cC}(a\blacktriangleleft u \to a\blacktriangleleft u))
=
\cV(1_\cV \to \widehat{\cV}(u\to \cC(a\to a)u))
\cong
\cV(u \to \cC(a\to a)u).
$$
The composition morphism $-\circ_{\overline{\cC}}-$ is the mate of
$$
(\varepsilon^{\widehat{\cV}}_{u \to \cC(a\to b)v} \id_{\widehat{\cV}(v\to \cC(b\to c)w)})
\circ
(\id_{\cC(a\to b)} \varepsilon^{\widehat{\cV}}_{v \to \cC(b\to c)w})
\circ
((-\circ_\cC -)\id_w)
$$
under the adjunction
\begin{align*}
\cV(\overline{\cC}(a\blacktriangleleft u \to b\blacktriangleleft v)
&\overline{\cC}(b\blacktriangleleft v\to c\blacktriangleleft w) \to
\overline{\cC}(a\blacktriangleleft u\to c\blacktriangleleft w))
\\&=
\cV(\widehat{\cV}(u\to \cC(a\to b)v)\widehat{\cV}(v\to \cC(b\to c)w) \to \widehat{\cV}(u\to \cC(a\to c)w))
\\&\cong
\cV(u\widehat{\cV}(u\to \cC(a\to b)v)\widehat{\cV}(v\to \cC(b\to c)w) \to \cC(a\to c)w)
\end{align*}
It is straightforward to verify that $\overline{\cC}$ is a $\cV$-category.
\end{defn}

\begin{lem}
For every $a\blacktriangleleft u \in \overline{\cC}$, there is a canonical $\cV$-functor $\cL^{a \blacktriangleleft u}=a \blacktriangleleft u- : \widehat{\cV} \to \overline{\cC}$.
\end{lem}
\begin{proof}
On objects, we define $\cL^{a \blacktriangleleft u}(v) = a\blacktriangleleft uv$.
For $v,w\in\cV$, we define $\cL^{a \blacktriangleleft u}_{v \to w}$ to be the mate of $j_a\id_u \varepsilon^{\widehat{\cV}}_{v \to w}$ under the adjunction
\begin{align*}
\cV(\widehat{\cV}(v\to w) \to \overline{\cC}(a \blacktriangleleft uv \to a\blacktriangleleft uw))
&=
\cV(\widehat{\cV}(v\to w) \to \widehat{\cV}(uv\to \cC(a\to a)uw))
\\&\cong
\cV(uv\widehat{\cV}(v\to w) \to \cC(a\to a)uw).
\end{align*}
To verify that $\cL^{a \blacktriangleleft u}$ is a $\cV$-functor, we see that the mate of $(\cL^{a \blacktriangleleft u}_{v\to w}\cL^{a \blacktriangleleft u}_{w\to x})\circ(-\circ_{\overline{\cC}}-)$ is given by
\begin{align*}
(\id_{uv} \cL^{a \blacktriangleleft u}_{v\to w}\cL^{a \blacktriangleleft u}_{w\to x})
&\circ
(\varepsilon^{\widehat{\cV}}_{uv \to \cC(a\to a)uw} \id_{[uw, \cC(a\to a)ux]})
\circ
(\id_{\cC(a\to a)} \varepsilon^{\widehat{\cV}}_{uw \to \cC(a\to a)ux})
\circ
((-\circ_\cC -)\id_{ux})
\\&=
[j_a\id_u][(\varepsilon^{\widehat{\cV}}_{v\to w} \id_{x}) \circ \varepsilon^{\widehat{\cV}}_{w\to x}]
\\&=
[j_a\id_u][(-\circ_{\widehat{\cV}}-) \circ \varepsilon^{\widehat{\cV}}_{v\to x}]
\end{align*}
which is exactly the mate of $(-\circ_{\widehat{\cV}}-)\circ \cL^{a \blacktriangleleft u}_{v\to x}$.
\end{proof}

We define the underlying functor $\cL_{a \blacktriangleleft u}$ on objects by $\cL_{a \blacktriangleleft u}(v) = a \blacktriangleleft uv$,
and on the morphism $f\in \cV(v\to w)$, we have that $\cL_{a \blacktriangleleft u}(f)$ is the mate of $j_a \id_u f \in \cV(uv \to \cC(a\to a)uw)$ under the adjunction.

\begin{prop}
The underlying functor $\cL_{a \blacktriangleleft u}$ is left adjoint to the underlying functor $\cR_{a \blacktriangleleft u}$.
\end{prop}
\begin{proof}
Notice we have a series of isomorphisms which are clearly natural in $v\in\cV$:
\begin{equation}
\label{eq:UnderlyingAdjunctionCompleteL}
\begin{split}
\overline{\cC}^\cV(\cL_{a\blacktriangleleft u}(v) \to b\blacktriangleleft w)
&=
\cV(\id_{\cV}\to \widehat{\cV}(uv\to \cC(a\to b) w))
\\&\cong
\cV(uv\to \cC(a\to b) w) 
\\&\cong
\cV(v\to \widehat{\cV}(u\to\cC(a\to b) w)) 
\\&=
\cV(v \to \cR_{a\blacktriangleleft u}(b\blacktriangleleft w)).
\end{split}
\end{equation}
It remains to show the above isomorphisms are natural in $b\blacktriangleleft w \in \overline{\cC}^\cV$.
We must show that for every
$f\in \overline{\cC}^\cV(b\blacktriangleleft w \to c\blacktriangleleft x)
=
\cV(1_\cV \to [w, \cC(b\to c)x])$,
the following diagram commutes:
$$
\xymatrix{
\overline{\cC}^\cV(a\blacktriangleleft uv \to b\blacktriangleleft w)
\ar[rr]^{\psi}
\ar[d]^{-\circ f}
&&
\cV(v\to \widehat{\cV}(u \to \cC(a\to b) w)
\ar[d]^{-\circ\cR_{a\blacktriangleleft u}(f)}
\\
\overline{\cC}^\cV(a\blacktriangleleft uv \to c\blacktriangleleft x)
\ar[rr]^{\psi}
&&
\cV(v\to \widehat{\cV}(u\to \cC(a\to c) x)
}
$$
where the horizontal arrows $\psi$ are instances of the series of isomorphisms \eqref{eq:UnderlyingAdjunctionCompleteL}.
It is easiest to do so by taking mates under the adjunction
\begin{equation}
\label{eq:UndoOneStep}
\cV(v \to \widehat{\cV}(u\to \cC(a\to c)x)) \cong \cV(uv \to \cC(a\to c)x)
\end{equation}
which effectively undoes one step of \eqref{eq:UnderlyingAdjunctionCompleteL}.
Let
$g \in \overline{\cC}^\cV(a \blacktriangleleft uv \to b \blacktriangleleft w)
=
\cV(1_\cV \to \widehat{\cV}(uv\to \cC(a\to b)w))$, and note that the mate of $\psi(g)$ under \eqref{eq:UndoOneStep} with $c=b$ and $x=w$ is given by
$(\id_{uv} g)\circ \varepsilon_{uv \to \cC(a\to b)w}$.
Now since
$$
\cR_{a\blacktriangleleft u}(f)
\in
\cV(
\overline{\cC}(a\blacktriangleleft u \to b\blacktriangleleft w)
\to
\overline{\cC}(a\blacktriangleleft u \to c\blacktriangleleft x)
)
=
\cV(\widehat{\cV}(u\to \cC(a\to b)w) \to \widehat{\cV}(u\to \cC(a\to c)x))
$$
is $(\id_{\overline{\cC}(a\blacktriangleleft u \to b\blacktriangleleft w)} f)\circ (-\circ_{\overline{\cC}}-)$,
the mate of $\psi(g) \circ \cR_{a\blacktriangleleft u}(f)$ under \eqref{eq:UndoOneStep} is given by
\begin{align*}
(\mate(\psi(g)) f)
&\circ
(-\circ_{\overline{\cC}}-)
\\&=
(\id_{u}\mate(\psi(g))f)
\circ
(\varepsilon^{\widehat{\cV}}_{u \to \cC(a\to b)w}\id_{\widehat{\cV}(w\to \cC(b\to c)x)})
\circ
(\id_{\cC(a\to b)} \varepsilon^{\widehat{\cV}}_{w \to \cC(b\to c)x})
\circ
((-\circ_\cC - )\id_x)
\\&=
(\id_{uv} gf)
\circ
(\varepsilon^{\widehat{\cV}}_{uv \to \cC(a\to b)w} \id_{\widehat{\cV}(w\to \cC(b\to c)x)})
\circ
(\id_{\cC(a\to b)} \varepsilon^{\widehat{\cV}}_{w \to \cC(b\to c)x})
\circ
((-\circ_\cC - )\id_x)
\\&=
(\id_{uv} gf)
\circ
(\id_{uv} (-\circ_{\overline{\cC}}-))
\circ
\varepsilon^{\widehat{\cV}}_{uv \to \cC(a\to c)x},
\end{align*}
which is exactly the mate of $\psi(g \circ f)$ under Adjunction \eqref{eq:UndoOneStep}.
\end{proof}

\begin{cor}
\label{cor:CBarTensored}
The completion $\overline{\cC}$ is tensored.
\end{cor}
\begin{proof}
By Theorem \ref{thm:StrongVMod}, it suffices to show that the stictly unital oplax $\cV$-module structure on $\overline{\cC}^\cV$ induced by $a\blacktriangleleft u \vartriangleleft v := \cL_{a\blacktriangleleft u}(v)= a\blacktriangleleft uv$ is strong, which is immediate.
\end{proof}

\subsection{Universal property of completion}

We now show that completion satisfies a universal property.

\begin{defn}
\label{defn:InclusionVFunctor}
Given a $\cV$-category $\cC$ (not necessarily tensored), the inclusion $\cV$-functor $\cI:\cC \to \overline{\cC}$ is given on objects by
$a\mapsto a\blacktriangleleft 1_\cV$ and $\cI_{a\to b}$ is the mate of $\id_{\cC(a\to b)}$ under the adjunction
\begin{align*}
\cV(\cC(a\to b) \to \overline{\cC}(a\blacktriangleleft 1_\cV \to b \blacktriangleleft 1_\cV))
&=
\cV(\cC(a\to b)\to \widehat{\cV}(1_\cV \to \cC(a\to b)1_\cV))
\\&\cong
\cV(\cC(a\to b)\to \cC(a\to b)).
\end{align*}
It is straightforward to verify by taking mates under the above adjunction that $\cI$ is a $\cV$-functor.
(The key relation is $(\id_{1_\cV}\cI_{a\to b})\circ \varepsilon^{\widehat{\cV}}_{1_\cV \to \cC(a\to b) 1_\cV} = \id_{\cC(a\to b)}$.)
\end{defn}

\begin{remark}
Recall that for $a\in \cC$ and $u\in \cV$, 
$\mu^\cI_{a,u}$ is defined as the mate of $\eta_{u} \circ \cI_{a \to a\vartriangleleft u}$ under the adjunction
$$
\overline{\cC}^\cV(a\blacktriangleleft u \to a\vartriangleleft u \blacktriangleleft 1_\cV)
\cong
\cV(u \to \overline{\cC}(a\to a\vartriangleleft u \blacktriangleleft 1_\cV)).
$$
We get the following two identities for $\mu^\cI_{a,u}$ and $\eta^\cC_{a,u}$ depending on whether we pass through the second equality below under the identification $\widehat{\cV}^\cV=\cV$:
$$
\overline{\cC}^\cV(a\blacktriangleleft u \to a\vartriangleleft u \blacktriangleleft 1_\cV)
=
\cV(1_\cV \to \widehat{\cV}(u \to \cC(a\to a\vartriangleleft u)))
=
\cV(u \to \cC(a\to a\vartriangleleft u)).
$$
Just passing through the first equality, we get the first identity below, and passing to the second, we get the second identity.
\begin{align}
\label{eq:MateOfMu=Eta}
\eta^\cC_{a,u}
&=
(\id_u \mu^\cI_{a,u}) \circ \varepsilon^{\widehat{\cV}}_{u \to \cC(a\to a\vartriangleleft u)}
\\
\label{eq:Mu=Eta}
\eta^\cC_{a,u}
&=
(\id_{1_\cV} (\eta^\cC_{a,u}\circ \cI_{a\to a\vartriangleleft u})) \circ \varepsilon^{\widehat{\cV}}_{1_\cV \to \cC(a\to a\vartriangleleft u)1_\cV}
=
\mate(\eta^\cC_{a,u}\circ \cI_{a\to a\vartriangleleft u})
=
\mu^\cI_{a,u}
\end{align}
\end{remark}

\begin{prop}
\label{prop:LiftToCBar}
Suppose $\cC, \cD$ are $\cV$-categories with $\cD$ tensored and $\cF: \cC \to \cD$ is a $\cV$-functor.
There exists a tensored $\cV$-functor $\overline{\cF} : \overline{\cC}\to \cD$ 
such that $\cI\circ \overline{\cF}\cong \cF$ as $\cV$-functors.
\end{prop}
\begin{proof}
We define a $\cV$-functor $\overline{\cF}: \overline{\cC} \to \cD$ by $\overline{\cF}(a\blacktriangleleft u) = \cF(a)\vartriangleleft u$, and we define $\overline{\cF}_{a\blacktriangleleft u \to b\blacktriangleleft v}$ to be the mate of 
$$
\alpha^{-1}_{\cF(a),u, \widehat{\cV}(u \to\cC(a\to b) v)} 
\circ
(\id_{\cF(a)} \vartriangleleft \varepsilon^{\widehat{\cV}}_{u \to \cC(a\to b) v})
\circ
(\id_{\cF(a)} \vartriangleleft \cF_{a\to b} \id_v)
\circ
\alpha_{\cF(a), \cD(\cF(a)\to \cF(b)), v}
\circ
(\varepsilon^\cD_{\cF(a) \to \cF(b)}\vartriangleleft \id_v)
$$
under the adjunction
\begin{align*}
\cV(\overline{\cC}(a\blacktriangleleft u \to b\blacktriangleleft v) 
&\to
\cD(\cF(a)\vartriangleleft u \to \cF(b)\vartriangleleft v))
\\&=
\cV(\widehat{\cV}(u \to\cC(a\to b) v) \to \cD(\cF(a)\vartriangleleft u \to \cF(b)\vartriangleleft v))
\\&\cong
\cD^\cV(\cF(a)\vartriangleleft u \vartriangleleft \widehat{\cV}(u \to\cC(a\to b) v) \to \cF(b)\vartriangleleft v).
\end{align*}
To verify that $\overline{\cF}$ is a $\cV$-functor, we show that the mates of
$(\overline{\cF}_{a\blacktriangleleft u \to b\blacktriangleleft v}\overline{\cF}_{b\blacktriangleleft v \to c\blacktriangleleft w})\circ (-\circ_\cD-)$
and
$(-\circ_{\overline{\cC}}-)\circ \overline{\cF}_{a\blacktriangleleft u \to c\blacktriangleleft w}$
agree under the above adjunction.
We leave this tedious and straightforward calculation to the reader.

Now for $a\blacktriangleleft u \in \overline{\cC}$ and $v\in \cV$, we have $\mu^{\overline{\cF}}_{a\blacktriangleleft u, v}$ is the mate of $\eta^{\overline{\cC}}_{a\blacktriangleleft u , v} \circ \overline{\cF}_{a\blacktriangleleft u \to a\blacktriangleleft uv}$ under the adjunction
\begin{align*}
\cD^\cV(\cF(a)\vartriangleleft u \vartriangleleft v \to \cF(a)\vartriangleleft uv)
&=
\cD^\cV(\overline{\cF}(a\blacktriangleleft u) \vartriangleleft v \to \overline{\cF}(a\blacktriangleleft u v))
\\&\cong
\cV(v \to \cD(\overline{\cF}(a\blacktriangleleft u) \to \overline{\cF}(a\blacktriangleleft u v))).
\end{align*}
Thus we have
\begin{align*}
\mu^{\overline{\cF}}_{a\blacktriangleleft u, v} 
&=
(\id_{\overline{\cF}(a\blacktriangleleft u)} \vartriangleleft (\eta^{\overline{\cC}}_{a\blacktriangleleft u , v} \circ \overline{\cF}_{a\blacktriangleleft u \to a\blacktriangleleft uv}))
\circ
\varepsilon^\cD_{\overline{\cF}(a\blacktriangleleft u) \to \overline{\cF}(a\blacktriangleleft uv)}
\\&=
(\id_{\cF(a)\vartriangleleft u} \vartriangleleft \eta^{\overline{\cC}}_{a\blacktriangleleft u, v})
\circ
\mate(\overline{\cF}_{a\blacktriangleleft u \to a\blacktriangleleft uv})
\\&=
\alpha^{-1}_{\cF(a), u, v}
\circ
(\id_{\cF(a)} \vartriangleleft 
[
(\id_u \eta^{\overline{\cC}}_{a\blacktriangleleft u, v})
\circ
\varepsilon^{\widehat{\cV}}_{u\to \cC(a\to a)uv}
\circ
(\cF_{a\to a}\id_{uv})
]
)
\circ
\alpha_{\cF(a), \cD(\cF(a)\to \cF(a)), uv}
\circ 
(\varepsilon^\cD_{\cF(a) \to \cF(a)} \vartriangleleft \id_{uv})
\\&=
\alpha^{-1}_{\cF(a), u, v}
\circ
(\id_{\cF(a)} \vartriangleleft 
[
(j_{a}\circ \cF_{a\to a})\id_{uv}
]
)
\circ
\alpha_{\cF(a), \cD(\cF(a)\to \cF(a)), uv}
\circ 
(\varepsilon^\cD_{\cF(a) \to \cF(a)} \vartriangleleft \id_{uv})
\\&=
\alpha^{-1}_{\cF(a), u, v}
\circ
([
(\id_{\cF(a)}\vartriangleleft j_{\cF(a)})
\circ 
\varepsilon^\cD_{\cF(a) \to \cF(a)} 
]\vartriangleleft \id_{uv})
\\&=
\alpha_{\cF(a),u,v}^{-1},
\end{align*}
and so $\mu^{\overline{\cF}}_{a\blacktriangleleft u, v}$ is invertible.
Hence $\overline{\cF}$ is tensored.

We define $\sigma : \cF\Rightarrow\cI \circ \overline{\cF}$ by $\sigma_a = \eta^\cD_{\cF(a),1_\cV} \in \cV(1_\cV \to \cD(\cF(a) \to \cF(a) \vartriangleleft 1_\cV))$, and we note that $\sigma_a^{-1} = \rho_a^{\cD}$ by Lemma \ref{lem:StrictlyUnital}.
To verify $\sigma$ is a $1_\cV$-graded $\cV$-natural isomorphism, it remains to verify \eqref{eq:VNaturalTransformation}.
Notice that \eqref{eq:VNaturalTransformation} is equivalent to 
$
(\sigma_a^{-1}\cF_{a\to b})\circ (-\circ_\cD-)
=
((\cI\circ \overline{\cF})_{a\to b}\sigma_b^{-1}) \circ (-\circ_\cD -)
$.
Under the adjunction
$$
\cV(\cC(a\to b) \to \cD(\cF(a)\vartriangleleft 1_\cV \to \cF(b)))
\cong
\cD^\cV(\cF(a)\vartriangleleft 1_\cV\vartriangleleft \cC(a\to b) \to \cF(b)),
$$
the mate of $((\cI\circ \overline{\cF})_{a\to b}\sigma_b^{-1}) \circ (-\circ_\cD -)$ is given by
\begin{align*}
(\id_{\cF(a)\vartriangleleft 1_\cV}\vartriangleleft & (\cI\circ \overline{\cF})_{a\to b}\sigma_b^{-1})
\circ
\alpha_{\cF(a)\vartriangleleft 1_\cV, \cD(\cF(a)\vartriangleleft 1_\cV \to \cF(b)\vartriangleleft 1_\cV), \cD(\cF(b)\vartriangleleft 1_\cV \to \cF(b))}
\\&\hspace{1cm}\circ
(\varepsilon^\cD_{\cF(a)\vartriangleleft 1_\cV \to \cF(b)\vartriangleleft 1_\cV} \vartriangleleft \id_{ \cD(\cF(b)\vartriangleleft 1_\cV \to \cF(b))})
\circ
\varepsilon_{\cF(b) \vartriangleleft 1_\cV \to \cF(b)}
\\&=
(\id_{\cF(a)\vartriangleleft 1_\cV} \vartriangleleft (\cI\circ \overline{\cF})_{a\to b})
\circ
\varepsilon^\cD_{\cF(a)\vartriangleleft 1_\cV \to \cF(b)\vartriangleleft 1_\cV}
\circ
\rho^\cD_{\cF(b)}
\\&=
(\id_{\cF(a)\vartriangleleft 1_\cV} \vartriangleleft \cI_{a\to b})
\circ
\alpha^{-1}_{\cF(a),1_\cV, \widehat{\cV}(1_\cV \to\cC(a\to b))} 
\circ
(\id_{\cF(a)} \vartriangleleft \varepsilon^{\widehat{\cV}}_{1_\cV \to \cC(a\to b)})
\\&\hspace{1cm}
\circ
(\id_{\cF(a)} \vartriangleleft \cF_{a\to b} \id_{1_\cV})
\circ
\alpha_{\cF(a), \cD(\cF(a)\to \cF(b)), 1_\cV}
\circ
(\varepsilon^\cD_{\cF(a) \to \cF(b)}\vartriangleleft \id_{1_\cV})
\circ
\rho^\cD_{\cF(b)}
\\&=
(\rho^\cD_{\cF(a)} \vartriangleleft \id_{\cC(a\to b)})
\circ
(\id_{\cF(a)} \vartriangleleft \cF_{a\to b})
\circ
(\rho^\cD_{\cF(a)\vartriangleleft \cD(\cF(a)\to \cF(b))})^{-1}
\circ
(\varepsilon^\cD_{\cF(a) \to \cF(b)}\vartriangleleft \id_{1_\cV})
\circ
\rho^\cD_{\cF(b)}
\\&=
(\rho_{\cF(a)}^\cD\vartriangleleft \cF_{a\to b})\circ \varepsilon^\cD_{\cF(a) \to \cF(b)},
\end{align*}
which is exactly the mate of $(\sigma_a^{-1}\cF_{a\to b})\circ (-\circ_\cD -)$.
Notice here we have used the identities
$\alpha^{-1}_{d,v,1_\cV} = \rho^\cD_{d\vartriangleleft v}$
and 
$\alpha^{-1}_{d, 1_\cV, v} = \rho^\cD_{d}\vartriangleleft \id_v$
for all $d\in \cD$ and $v\in \cV$ from Lemma \ref{lem:AlphaInvertibleWhenOneArgumentIs1}.
\end{proof}

\begin{remark}
\label{remark:OpenQuestionAboutCompletion}
The above proof raises the following two interesting questions.
\begin{enumerate}
\item
Given a tensored $\cV$-functor $\cG: \overline{\cC} \to \cD$ such that $\cI \circ \cG \cong \cF$, when do we have that $\cG$ is $\cV$-equivalent to $\overline{\cF}$?
We are unable to provide any candidate natural isomorphism at this time.
\item
When $\cC$ is tensored, 
setting $\cD= \cC$ and $\cF = \id^\cC$, we get a canonical tensored $\cV$-functor $\overline{\id^\cC}: \overline{\cC}\to \cC$ such that $\id_\cC \cong \cI \circ \overline{\id^\cC}$.
We show in Lemma \ref{lem:MuIsANaturalTransformation} below that $\tau_{a\blacktriangleleft u} : = \mu^\cI_{a,u}$ defines a $1_\cV$-graded natural transformation $\tau: \id^{\overline{\cC}} \Rightarrow  \overline{\id^\cC}\circ \cI$.
However we do not know how to show that $\overline{\id^\cC}\circ \cI$ is naturally isomorphic to $\id^{\overline{\cC}}$, since it may be the case that $\overline{\cC}$ is too big; we cannot yet identify $a\vartriangleleft u \blacktriangleleft 1_\cV$ and $a\blacktriangleleft u$.
We will solve this problem in the next section by adding additional hypotheses on $\cC$.
\end{enumerate}
\end{remark}

\begin{lem}
\label{lem:MuIsANaturalTransformation}
Setting $\tau_{a\blacktriangleleft u} : = \mu^\cI_{a,u}$ in
$$
\overline{\cC}^\cV(a\blacktriangleleft u \to a\vartriangleleft u\blacktriangleleft 1_\cV)
=
\overline{\cC}^\cV( \id^{\overline{\cC}}(a\blacktriangleleft u) \to \cI(\overline{\id}^\cC(a\blacktriangleleft u)))
$$
defines a $1_\cV$-graded natural transformation $\tau: \id^{\overline{\cC}} \Rightarrow  \overline{\id^\cC}\circ \cI$.
\end{lem}
\begin{proof}
To show the naturality condition \eqref{eq:VNaturalTransformation},
we compute mates under the adjunction
\begin{align*}
\cV(
\overline{\cC}(a\blacktriangleleft u \to b\blacktriangleleft v)
\to
\overline{\cC}(a\blacktriangleleft u \to b\vartriangleleft v\blacktriangleleft 1_\cV)
)
&=
\cV(
\widehat{\cV}(u \to \cC(a\to b)v) 
\to 
\widehat{\cV}(u \to \cC(a\to b\vartriangleleft v))
)
\\&\cong
\cV(
u\widehat{\cV}(u \to \cC(a\to b)v) 
\to 
\cC(a\to b\vartriangleleft v)
).
\end{align*}
Indeed, the mate of 
of $(\tau_{a\blacktriangleleft u} (\overline{\id^\cC}\circ \cI)_{a\blacktriangleleft u \to b\blacktriangleleft v})\circ (-\circ_{\overline{\cC}}-)$
is given by
\begin{align*}
(
[
(\id_u \tau_{a\blacktriangleleft u}) \circ \varepsilon^{\widehat{\cV}}_{u\to  \cC(a\to a\vartriangleleft u)}
]&
(\overline{\id^\cC}\circ \cI)_{a\blacktriangleleft u \to b\blacktriangleleft v})
)
\circ
(\id_{\cC(a\to a\vartriangleleft u) } \varepsilon^{\widehat{\cV}}_{1_\cV\to  \cC(a\vartriangleleft u \to b\vartriangleleft v)})
\circ
(-\circ_\cC-)
\\&
\underset{\text{\eqref{eq:MateOfMu=Eta}}}{=}
(\eta_{u,a}^\cC \overline{\id^\cC}_{a\blacktriangleleft u \to b\blacktriangleleft v})
\circ
(-\circ_\cC-)
\\&=
\varepsilon^{\widehat{\cV}}_{u \to\cC(a\to b)v}
\circ
(\id_{\cC(a\to b)} \eta^\cC_{b,v})
\circ
(-\circ_\cC-)
\\&
\underset{\text{\eqref{eq:MateOfMu=Eta}}}{=}
\varepsilon^{\widehat{\cV}}_{u \to \cC(a\to b)v}
\circ
(\id_{\cC(a\to b)} [(\id_v\tau_{b\blacktriangleleft v})\circ \varepsilon^{\widehat{\cV}}_{v \to \cC(b \to \vartriangleleft v)}])
\circ
(-\circ_\cC-)
\end{align*}
which is exactly the mate of  $(\id^{\overline{\cC}}_{a\blacktriangleleft u \to b\blacktriangleleft v}\tau_{b\blacktriangleleft v})\circ (-\circ_{\overline{\cC}}-)$.
\end{proof}

\subsection{When representable \texorpdfstring{$\cV$}{V}-functors are tensored}

Recall $\cV$ is closed so we may form $\widehat{\cV}$.
We begin with a lemma that we could have proved in Section \ref{sec:TensoredVFunctors}.

\begin{lem}
\label{lem:NaturalityForRepresentableFunctor}
Fix $a\in \cC$, and consider the $\cV$-functor $\cC(a\to -) : \cC \to \widehat{\cV}$. 
\begin{enumerate}
\item
Writing $\mu_{b,v} = \mu^{\cC(a\to -)}_{b,v}$ for $b\in\cC$ and $v\in \cV$, we have
 $\mu_{b,v} = (\id_{\cC(a\to b)} \eta^\cC_{b,v})\circ (-\circ_\cC -)$.
\item
For all $g\in \cV(v \to \cC(b\to c)w)$, $\mu$ satisfies the naturality condition
\begin{equation}
\label{eq:NaturalForMu}
[\mu_{b,v}(\id_b\vartriangleleft g) \alpha^\cC_{b,\cC(b\to c), w} (\varepsilon^\cC_{b \to c}\vartriangleleft \id_w)]
\circ
(-\circ_\cC-\circ_\cC-\circ_\cC -)
=
(\id_{\cC(a\to b)} g)
\circ 
((-\circ_\cC - ) \id_w)
\circ 
\mu_{c,w}.
\end{equation}
(Recall the $\cC^\cV$-morphisms $\id_b \vartriangleleft g$ and $\varepsilon^\cC_{b\to c} \vartriangleleft \id_w$ are $\cV$-morphisms originating at $1_\cV$.)
\end{enumerate}
\end{lem}
\begin{proof}
To prove (1), recall that
we have identified $\widehat{\cV}^\cV = \cV$, and thus $\mu_{b,v}$ is the mate of 
$\eta^\cC_{b,v} \circ \cC(a\to -)_{b \to b\vartriangleleft v}$ under the adjunction
$$
\cV(\cC(a\to b)v \to \cC(a\to b\vartriangleleft v))
\cong
\cV(v \to \widehat{\cV}(\cC(a\to b) \to \cC(a\to b\vartriangleleft v))).
$$
But the mate of $\eta^\cC_{b,v} \circ \cC(a\to -)_{b \to b\vartriangleleft v}$ is exactly $(\id_{\cC(a\to b)} \eta^\cC_{b,v})\circ (-\circ_\cC -)$ under the above adjunction by the definition of $\cC(a\to -)_{b \to b\vartriangleleft v}$ from Example \ref{ex:VRepresentable}.

We prove (2) using the identity from (1) twice.
Note that $\id_b \vartriangleleft g := \cL_b(g)$ and $\varepsilon^\cC_{b\to c} \vartriangleleft \id_w$ is defined via taking mates as in the beginning of Section \ref{sec:V-cat to V-mod}.
Hence the left hand side of \eqref{eq:NaturalForMu} is equal to
\begin{align*}
[\id_{\cC(a\to b)} \eta^\cC_{b,v}&(\id_b\vartriangleleft g) \alpha^\cC_{b,\cC(b\to c), w} (\varepsilon^\cC_{b \to c}\vartriangleleft \id_w)]
\circ
(-\circ_\cC-\circ_\cC-\circ_\cC-\circ_\cC -)
\\&=
(\id_{\cC(a\to b)}
[(\eta^\cC_{b,v}\cL_b(g)) \circ (-\circ_\cC-)]
\alpha^\cC_{b,\cC(b\to c), w} (\varepsilon^\cC_{b \to c}\vartriangleleft \id_w)]
\circ
(-\circ_\cC-\circ_\cC-\circ_\cC-)
\\&=
[\id_{\cC(a\to b)}
(g\circ \eta^\cC_{b,\cC(b\to c)w})
\alpha^\cC_{b,\cC(b\to c), w} (\varepsilon^\cC_{b \to c}\vartriangleleft \id_w)]
\circ
(-\circ_\cC-\circ_\cC-\circ_\cC-)
\\&=
[\id_{\cC(a\to b)}
[((g\circ \eta^\cC_{b,\cC(b\to c)w}) \alpha^\cC_{b,\cC(b\to c), w})\circ (-\circ_\cC -)])(\varepsilon^\cC_{b \to c}\vartriangleleft \id_w)]
\circ
(-\circ_\cC-\circ_\cC-)
\\&=
[\id_{\cC(a\to b)}
[g\circ (\eta^\cC_{b,\cC(b\to c)}\eta^\cC_{b\vartriangleleft \cC(b\to c), w})\circ (-\circ_\cC-)](\varepsilon^\cC_{b \to c}\vartriangleleft \id_w)]
\circ
(-\circ_\cC-\circ_\cC-)
\\&=
(
\id_{\cC(a\to b)}
[g\circ ([(\eta^\cC_{b,\cC(b\to c)}\varepsilon^\cC_{b \to c})\circ (-\circ_\cC-)] \eta^\cC_{c, w})
\circ (-\circ_\cC-)]
)
\circ
(-\circ_\cC-)
\\&=
[\id_{\cC(a\to b)} (g \circ (\id_{\cC(b\to c)} \eta^\cC_{c, w}) \circ (-\circ_\cC -))] 
\circ
(-\circ_\cC-)
\\&=
(\id_{\cC(a\to b)}g) \circ ((-\circ_\cC -) \id_w)\circ (\id_{\cC(a\to c)}\eta^\cC_{c,w})\circ (-\circ_\cC-),
\end{align*}
which is exactly the right hand side of \eqref{eq:NaturalForMu}.
\end{proof}

\begin{prop}
\label{prop:IsomorphismOfRepresentableFunctors}
Suppose $\cC$ is tensored and $a\in \cC$ and $u\in \cV$ such that the $\cV$-representable functors $\cC(a\to -)$ and $\cC(a\vartriangleleft u \to -)$ are tensored.
Then we have a natural isomorphism of representable functors
$
\overline{\cC}^\cV(a\vartriangleleft u\blacktriangleleft 1_\cV \to -)
\cong
\overline{\cC}^\cV(a\blacktriangleleft u \to -)
$.
\end{prop}
\begin{proof}
Since $\cC$ is tensored and the representable functors $\cC(a\to -)$ and $\cC(a\vartriangleleft u\to -)$ are both tensored, we get the following natural isomorphism of representable functors:
\begin{equation}
\label{eq:YonedaForRepresentableFunctors}
\begin{split}
\overline{\cC}^\cV(a\vartriangleleft u\blacktriangleleft 1_\cV \to b\blacktriangleleft v)
&=
\cV(1_\cV \to \overline{\cC}( a\vartriangleleft u\blacktriangleleft 1_\cV\to b\blacktriangleleft v))
\\&=
\cV(1_\cV \to \widehat{\cV}(1_\cV\to \cC(a\vartriangleleft u\to b)v))
\\&=
\cV(1_\cV \to \cC(a\vartriangleleft u\to b)v)
\\&\cong
\cV(1_\cV \to \cC(a\vartriangleleft u\to b\vartriangleleft v))
\\&=
\cC^\cV(a\vartriangleleft u\to b\vartriangleleft v)
\\&\cong
\cV(u \to \cC(a\to b\vartriangleleft v))
\\&\cong
\cV(u \to \cC(a\to b) v)
\\&\cong
\cV(1_\cV \to \widehat{\cV}(u\to \cC(a\to b)v))
\\&=
\cV(1_\cV \to \overline{\cC}(a\blacktriangleleft u \to b\blacktriangleleft v))
\\&=
\overline{\cC}^\cV(a\blacktriangleleft u \to b\blacktriangleleft v).
\end{split}
\end{equation}
To show the above isomorphism is natural in $b\blacktriangleleft v$, one uses the naturality condition \eqref{eq:NaturalForMu} twice.
We use \eqref{eq:NaturalForMu} the first time for $\mu^{\cC(a\vartriangleleft u\to -)}_{b,v}$ in the first isomophism in \eqref{eq:YonedaForRepresentableFunctors}.
For the second use, notice that when $\cC(a \to -)$ tensored, \eqref{eq:NaturalForMu} is equivalent to
$$
[\id_{\cC(a\to b\vartriangleleft v)}(\id_b\vartriangleleft g) \alpha^\cC_{b,\cC(b\to c), w} (\varepsilon^\cC_{b \to c}\vartriangleleft \id_w)]
\circ
(-\circ_\cC-\circ_\cC-\circ_\cC -)
\circ 
\mu_{c,w}^{-1}
=
\mu_{b,v}^{-1}
\circ
(\id_{\cC(a\to b)} g)
\circ 
((-\circ_\cC - ) \id_w).
$$
We use this equivalent version of \eqref{eq:NaturalForMu} in the third isomorphism in \eqref{eq:YonedaForRepresentableFunctors}.
We leave the rest of the details to the reader.
%
%
%
\end{proof}

\begin{thm}
\label{thm:WhenTensoredIsEquivalentToCompletion}
Suppose $\cC$ is a tensored $\cV$-category.
The following are equivalent:
\begin{enumerate}
\item
Every $\cV$-representable functor $\cR^a=\cC(a\to -):\cC\to \widehat{\cV}$ is tensored.
\item
The $\cV$-functor $\cI: \cC \to \overline{\cC}$ given by $a\mapsto a\blacktriangleleft 1_\cV$ from Definition \ref{defn:InclusionVFunctor} is tensored.
\item
The $1_\cV$-graded natural transformation $\tau:\id^{\overline{\cC}} \Rightarrow \overline{\id^\cC} \circ \cI$ defined by 
$\tau_{a\blacktriangleleft u} :=\mu^\cI_{a,u}$ from Lemma \ref{lem:MuIsANaturalTransformation} is a natural isomorphism.
\item
The $\cV$-functors $\cI: \cC \to \overline{\cC}$ given by $a\mapsto a\blacktriangleleft 1_\cV$ from Definition \ref{defn:InclusionVFunctor} and 
$\overline{\id^\cC}: \overline{\cC} \to \cC$ given by $a\blacktriangleleft u \mapsto a \vartriangleleft u$ 
from Proposition \ref{prop:LiftToCBar} witness a $\cV$-equivalence.
\end{enumerate}
\end{thm}
\begin{proof}
\mbox{}
\item[\underline{$(1)\Rightarrow (2)$:}]
First suppose (1) holds.
Notice that $\cI(a)\vartriangleleft u = a\blacktriangleleft 1_\cV \vartriangleleft u = a\blacktriangleleft u$ and
$\cI(a\vartriangleleft u) = a\vartriangleleft u \blacktriangleleft 1_\cV$.
Since $\cC$ is tensored and every $\cV$-representable functor is tensored, Proposition \ref{prop:IsomorphismOfRepresentableFunctors} gives us a canonical natural isomorphism of representable functors $\overline{\cC}^\cV(a\vartriangleleft u \blacktriangleleft 1_\cV \to -) \cong \overline{\cC}^\cV(a\blacktriangleleft u \to -)$.
By the Yoneda Lemma as in \eqref{eq:YonedaInverses}, we get a canonical isomorphism in $\overline{\cC}^\cV(a\blacktriangleleft u \to a\vartriangleleft u \blacktriangleleft 1_\cV)$.
We claim this isomorphism is exactly $\mu^\cI_{a,u}$, which is thus invertible by \eqref{eq:YonedaInverses}.
Indeed, setting $b= a\vartriangleleft u$ and $v=1_\cV$, 
Adjunction \eqref{eq:YonedaForRepresentableFunctors} becomes the following isomorphism:
\begin{align*}
\overline{\cC}^\cV(a\vartriangleleft u \blacktriangleleft 1_\cV \to a\vartriangleleft u \blacktriangleleft 1_\cV)
&=
\cV(1_\cV \to \cC(a\vartriangleleft u \to a\vartriangleleft u))
\\&=
\cC^\cV(a\vartriangleleft u \to a\vartriangleleft u)
\\&\cong
\cV(u\to \cC(a\to a\vartriangleleft u))
\\&=
\cV(1_\cV\to \widehat{\cV}(u \to \cC(a\to a\vartriangleleft u)))
\\&=
\overline{\cC}^\cV(a\blacktriangleleft u \to a\vartriangleleft u \blacktriangleleft 1_\cV).
\end{align*}
Moreover, $\id_{a\vartriangleleft u \blacktriangleleft 1_\cV}$ transforms as follows under the above isomorphism:
$$
\id_{a\vartriangleleft u \blacktriangleleft 1_\cV}
=
j_{a\vartriangleleft u}
\leftrightarrow
\eta^\cC_{a,u}
\underset{\text{\eqref{eq:Mu=Eta}}}{=}
\mu^\cI_{a,u}.
$$

\item[\underline{$(2)\Leftrightarrow (3)$:}]
The condition that $\cI$ is tensored is exactly the condition that $\mu^\cI_{a,u}$ is invertible for all $a\in \cC$ and $u\in \cV$, which is equivalent to $\tau_{a\blacktriangleleft u}$ being invertible for all $a\blacktriangleleft u \in \overline{\cC}$.
Hence $\cI$ is tensored if and only if $\tau:\id^{\overline{\cC}} \Rightarrow \overline{\id^\cC} \circ \cI$ is an isomorphism.

\item[\underline{$(3)\Rightarrow (4)$:}]
Suppose $\cI$ is tensored, so $\tau:\id^{\overline{\cC}} \Rightarrow \overline{\id^\cC} \circ \cI$ is an isomorphism.
Since we always have an equivalence of $\cV$-functors $\cI \circ \overline{\id^\cC} \cong \id^\cC$ by Proposition \ref{prop:LiftToCBar}, we have a $\cV$-equivalence $\cC \cong \overline{\cC}$.

\item[\underline{$(4)\Rightarrow (1)$:}]
Assume (4) holds.
Then for all $v\in\cV$ and $a,b\in\cC$, the following chain of isomorphisms is natural in $u\in \cV$:
\begin{align*}
\cV(u \to \cC(a\to b)v)
&\cong
\cV(u \to \widehat{\cV}(1_\cV\to \cC(a\to b)v))
\\&=
\cV(u \to \overline{\cC}(a\blacktriangleleft 1_\cV\to b\blacktriangleleft v)
\\&\cong
\cV(u \to \cC(a\vartriangleleft 1_\cV\to b\vartriangleleft v))
\\&\cong
\cV(u \to \cC(a\to b\vartriangleleft v))
\end{align*}
Hence we have an isomorphism of representable functors $\cV(- \to \cC(a\to b)v)\cong \cV(- \to \cC(a\to b\vartriangleleft v)) $, which gives us an isomorphism $\cC(a\to b)v\cong \cC(a\to b\vartriangleleft v)$.
We claim this isomorphism is exactly $\mu^{\cC(a\to -)}_{b,v}$.
Indeed, setting $u= \cC(a\to b)v$ and passing through the above chain of isomorphisms, $\id_{\cC(a\to b)v}$ transforms as follows:
\begin{align*}
\id_{\cC(a\to b)v}
&\leftrightarrow 
\eta^{\widehat{\cV}}_{1_\cV, \cC(a\to b)v}
\\&\leftrightarrow 
\eta^{\widehat{\cV}}_{1_\cV, \cC(a\to b)v}\circ \overline{\id^\cC}_{a\blacktriangleleft 1_\cV \to b\blacktriangleleft v}
\\&\leftrightarrow 
\eta^{\widehat{\cV}}_{1_\cV, \cC(a\to b)v}\circ \overline{\id^\cC}_{a\blacktriangleleft 1_\cV \to b\blacktriangleleft v} \circ (\rho_a^{-1} \id_{\cC(a\vartriangleleft 1_\cV \to b\vartriangleleft v)})\circ (-\circ_\cC-)
\\&
\underset{\text{Lem.~\ref{lem:StrictlyUnital}}}{=}
\eta^{\widehat{\cV}}_{1_\cV, \cC(a\to b)v}\circ \overline{\id^\cC}_{a\blacktriangleleft 1_\cV \to b\blacktriangleleft v} \circ (\eta^\cC_{a,1} \id_{\cC(a\vartriangleleft 1_\cV \to b\vartriangleleft v)})\circ (-\circ_\cC-)
\\&=
\eta^{\widehat{\cV}}_{1_\cV, \cC(a\to b)v}\circ 
\varepsilon^{\widehat{\cV}}_{1_\cV \to \cC(a\to b)v}
\circ
(\id_{\cC(a\to b)} \eta^\cC_{b,v})\circ (-\circ_\cC-)
\\&=
(\id_{\cC(a\to b)} \eta^\cC_{b,v})\circ (-\circ_\cC-)
\\&
\underset{\text{Lem.~\ref{lem:NaturalityForRepresentableFunctor}}}{=}
\mu^{\cC(a\to -)}_{b,v}.
\end{align*}
Hence $\mu^{\cC(a\to -)}_{b,v}$ is an isomorphism by the Yoneda Lemma as in \eqref{eq:YonedaInverses}.
\end{proof}

\subsection{Completion and rigidity of \texorpdfstring{$\cV$}{V}}
\label{sec:VRigid}

We now describe the connection between tensored $\cV$-functors and the notion of rigidity.
As before, we assume $\cV$ is closed.

\begin{example}
\label{ex:VRigid}
When $\cV$ is rigid, 
our convention for duals is given as in \cite[\S 2.7]{1701.00567} by $\ev_v \in \cV(vv^* \to 1_\cV)$ and $\coev_v \in \cV(1_\cV \to v^*v)$.
Hence $\cV$ is closed with $\widehat{\cV}(u\to v) = u^*v$.
Notice that $\widehat{\cV}(u\to v)w = u^*vw = \widehat{\cV}(u\to vw)$ for all $u,v,w\in\cV$, so every $\cV$-representable functor $\widehat{\cV}(u \to -)$ is clearly tensored.
In fact, we show this property \emph{characterizes} rigidity in Lemma \ref{lem:RigidCharacterization} below.

We can now describe $\widehat{\cV}$ in more detail without taking any mates.
The identity element $j_v \in \cV(1_\cV \to [v,v])=\cV(1_\cV \to v^*v)$ for $v\in \widehat{\cV}$ is equal to $\coev_v$.
The composition morphism $-\circ_{\widehat{\cV}}- \in \cV(\widehat{\cV}(u\to v)\widehat{\cV}(v\to w)\to \widehat{\cV}(u\to w))=\cV(u^*vv^*w \to u^*w)$ is equal to $\id_{u^*}\ev_v \id_{w}$.

Suppose $\cC$ is a $\cV$-category with $\cV$ rigid.
The $\cV$-representable functor $\cC(a\to -):\cC \to \widehat{\cV}$ is given by
$$
\cC(a\to -)_{b\to c}
=(\coev_{\cC(a\to b)}\id_{\cC(b\to c)})\circ (-\circ_\cC-)
\in
\cV(\cC(b\to c) \to \cC(a\to b)^* \cC(a \to c)).
$$

We now describe the tensored $\cV$-category $\overline{\cC}$ in greater detail.
We now have $\overline{\cC}(a\blacktriangleleft u \to b \blacktriangleleft v) = \widehat{\cV}(u\to \cC(a\to b)v) = u^*\cC(a\to b)v$, and the identity elements and composition are given by $j_{a\blacktriangleleft u} = \coev_{u} \circ(\id_{u^*} j_a \id_u) \in \cV(1_\cV \to u^*\cC(a\to a)u)$ and
$$
-\circ_{\overline{\cC}}-
=
(\id_{u^*}\id_{\cC(a\to b)}\ev_{v}\id_{\cC(b\to c)}\id_w) \circ (\id_{u^*}(-\circ_\cC-)\id_w)
\in
\cV(
u^*\cC(a\to b)vv^*\cC(b\to c)w \to u^*\cC(a\to c)w).
$$
Moreover, the left $\cV$-adjoint $\cL^{a\blacktriangleleft u}$ of $\cR^{a\blacktriangleleft u}$ is given by
$$
\cL_{v \to w}^{a\blacktriangleleft u}
=
(\id_{v^*} \coev_u \id_w)\circ (\id_{v^*}\id_{u^*} j_a \id_u\id_w)
\in
\cV(v^*w\to v^*u^*\cC(a\to a)u w^*)
$$
Finally, the $\cV$-functor
$\cI: \cC \to \overline{\cC}$ is given by $\cI_{a\to b}=\id_{\cC(a\to b)}\in \cV(\cC(a\to b)\to 1_\cV^*\cC(a\to b)1_\cV)$.
\end{example}


\begin{lem}
\label{lem:RigidCharacterization}
Consider the self-enriched $\cV$-category $\widehat{\cV}$.
Every $\cV$-representable functor $\widehat{\cV}(v\to -)$ is tensored if and only if every object of $\cV$ has a right dual.
\end{lem}
\begin{remark}
This result follows from \cite[Proposition 6.2]{MR3775482}, going back to \cite{MR749468}. We give an explicit proof here for the reader's convenience.
\end{remark}
\begin{proof}
If $\cV$ is rigid, then $\widehat{\cV}(u\to v):=u^*v$, so $\widehat{\cV}(u\to -) = u^*\otimes-$ which is obviously tensored.

Conversely, suppose that for $v\in \cV$, $\widehat{\cV}(v\to -)$ is tensored.
We define $v^*:=\widehat{\cV}(v\to 1_\cV)$ and
\begin{align*}
\ev_v &:= \varepsilon^{\widehat{\cV}}_{v\to 1_\cV}\in \cV(v\widehat{\cV}(v\to 1_\cV)\to 1_\cV)
\\
\coev_v &:= j_v \circ (\mu^{\cV(v\to -)}_{1_\cV, v})^{-1} \in \cV(1_\cV \to \widehat{\cV}(v\to 1_\cV)v)
\end{align*}
One zig-zag relation is readily verified using standard mate calculations:
$$
(\id_v \coev_v) \circ (\ev_v \id_v)
=
(\id_v j_v) \circ (\id_v (\mu^{\widehat{\cV}(v \to -)}_{v , 1_\cV})^{-1}) \circ (\varepsilon^{\widehat{\cV}}_{v \to 1_\cV} \id_v) 
=
(\id_v j_v) \circ \varepsilon^{\widehat{\cV}}_{v\to v}
=
\id_v
$$
The second is a bit more difficult.
First, one shows that $(\eta^{\widehat{\cV}}_{v,1_\cV} \id_{v^*}) \circ (-\circ_{\widehat{\cV}}-) = \varepsilon^{\widehat{\cV}}_{v\to 1_\cV} \circ j_{1_\cV}$ by taking mates under the adjunction
$$
\cV(v\widehat{\cV}(v\to 1_\cV) \to \widehat{\cV}(1_\cV \to 1_\cV))
\cong
\cV(c\widehat{\cV}(v\to 1_\cV) \to 1_\cV),
$$
It follows that $((\mu^{\widehat{\cV}(v\to -)}_{v,1_\cV})^{-1} \id_{v^*})\circ  (\id_{v^*} \varepsilon^{\widehat{\cV}}_{v\to 1}) = -\circ_{\widehat{\cV}}-\in \cV(\widehat{\cV}(v\to v) \widehat{\cV}(v\to 1_\cV) \to \widehat{\cV}(v\to 1_\cV))$.
Hence
$$
(\coev_v \id_{v^*})\circ (\id_{v^*} \ev_v)
=
(j_v\id_{v^*}) \circ ((\mu^{\widehat{\cV}(v \to -)}_{v , 1_\cV})^{-1}\id_{v^*}) \circ (\id_{v^*}\varepsilon^{\widehat{\cV}}_{v \to 1_\cV}) 
=
(j_v\id_{v^*})\circ (-\circ_{\widehat{\cV}}-)
=
\id_{v^*}.
$$
This completes the proof.
\end{proof}

\begin{cor}
Suppose $\cV$ is rigid.
Then every $\cV$-representable functor $\cR^{a\blacktriangleleft u}: \overline{\cC} \to \widehat{\cV}$ is tensored, and $\overline{\overline{\cC}}$ is $\cV$-equivalent to $\overline{\cC}$ by Theorem \ref{thm:WhenTensoredIsEquivalentToCompletion}.
\end{cor}
\begin{proof}
When $\cV$ is rigid,
$$
\overline{\cC}(a\blacktriangleleft u \to b\blacktriangleleft v)
=
\widehat{\cV}(u\to \cC(a\to b)v)
=
u^*\cC(a\to b)v
=
\overline{\cC}(a\blacktriangleleft u \to b\blacktriangleleft 1_\cV)v.
$$
We leave the rest of the details to the reader.
\end{proof}

\begin{cor}
Suppose $\cV$ is rigid and $\cC$ is tensored.
Then $\cI : \cC \to \overline{\cC}$ is tensored, 
and $\cC$ is $\cV$-equivalent to $\overline{\cC}$ by Theorem \ref{thm:WhenTensoredIsEquivalentToCompletion}.
\end{cor}
\begin{proof}
Notice that when $\cC$ is tensored and $\cV$ is rigid, we have an adjunction
\begin{equation}
\label{eq:RigidMorphismSpaceIso}
\begin{split}
\overline{\cC}^\cV(a\blacktriangleleft u \to b\blacktriangleleft v)
&=
\cV(u \to \cC(a \to b)v)
\\&\cong
\cV(uv^* \to \cC(a\to b))
\\&\cong
\cC^\cV(a\vartriangleleft uv^*\to b)
\\&\cong
\cC^\cV(a\vartriangleleft u\to b\vartriangleleft v).
\end{split}
\end{equation}
Define $\Phi : \overline{\cC}^\cV \to \cC^\cV$ by $\Phi(a\blacktriangleleft u) = a\vartriangleleft u$, and on morphisms, $\Phi$ is defined by the adjunction \eqref{eq:RigidMorphismSpaceIso}.

We claim that $\Phi$ is an equivalence of categories.
It is clear that if $\Phi$ is a functor, then $\Phi$ is an equivalence of categories, since it is automatically essentially surjective on objects by definition and fully faithful by \eqref{eq:RigidMorphismSpaceIso}.
Setting $b=a$ and $v = u$, it is readily checked that $\Phi(\id_{a\blacktriangleleft u}) = \id_{a\vartriangleleft u}$.
Hence for an arbitrary $f\in \overline{\cC}^\cV(a\blacktriangleleft u \to b\blacktriangleleft v)=\cV(u \to \cC(a\to b)v)$,
we have $\Phi(f) = (\id_a \vartriangleleft f) \circ \alpha_{a,\cC(a\to b), v}\circ (\varepsilon^\cC_{a\to b} \vartriangleleft \id_v)$.
Now if 
$g\in \overline{\cC}^\cV(b\blacktriangleleft v \to c\blacktriangleleft w)=\cV(v \to \cC(b\to c)w)$,
we see that 
\begin{align*}
\Phi&(f\circ g)
\\&=
(\id_a \vartriangleleft [f\circ (\id_{\cC(a\to b)}g) \circ ((-\circ_\cC-) \id_w)])
\circ
\alpha_{a, \cC(a\to c) , w}
\circ 
(\varepsilon^\cC_{a\to c} \vartriangleleft \id_w)
\\&=
(\id_a \vartriangleleft [f\circ (\id_{\cC(a\to b)}g)])
\circ
\alpha_{a, \cC(a\to b)\cC(b\to c) , w}
\circ
(\alpha_{a, \cC(a\to b),\cC(b\to c)}\vartriangleleft \id_w)
\circ
(\varepsilon^\cC_{a\to b} \vartriangleleft \id_{\cC(b\to c)} \vartriangleleft \id_w)
\circ 
(\varepsilon^\cC_{b\to c} \vartriangleleft \id_w)
\\&=
(\id_a \vartriangleleft [f\circ (\id_{\cC(a\to b)}g)])
\circ
\alpha_{a, \cC(a\to b),\cC(b\to c)w}
\circ
\alpha_{a\vartriangleleft\cC(a\to b),\cC(b\to c),w}
\circ
(\varepsilon^\cC_{a\to b} \vartriangleleft \id_{\cC(b\to c)} \vartriangleleft \id_w)
\circ 
(\varepsilon^\cC_{b\to c} \vartriangleleft \id_w)
\\&=
(\id_a \vartriangleleft f)
\circ
\alpha_{a, \cC(a\to b),v}
\circ
(\varepsilon^\cC_{a\to b} \vartriangleleft \id_v)
\circ
(\id_b \vartriangleleft g)
\circ
\alpha_{b,\cC(b\to c),w}
\circ 
(\varepsilon^\cC_{b\to c} \vartriangleleft \id_w)
\\&=
\Phi(f)\circ \Phi(g).
\end{align*}
Thus $\Phi$ is a functor, and it gives an equivalence of categories $\overline{\cC}^\cV \cong \cC^\cV$.

We now claim that setting $b = a\vartriangleleft u$ and $v=1_\cV$, $\Phi(\mu^\cI_{a,u})=\alpha_{a,u,1_\cV}$.
Indeed,
using the identity $\mu^\cI_{a,u} = \eta^\cC_{a,u}$ from \eqref{eq:Mu=Eta}, we have
$$
\Phi(\mu^\cI_{a,u})
=
(\id_a \vartriangleleft \eta^\cC_{a,u}) \circ \alpha_{a,\cC(a\to a\vartriangleleft u), 1_\cV}\circ (\varepsilon^\cC_{a\to a\vartriangleleft u} \vartriangleleft \id_{1_\cV})
=
\alpha_{a,u, 1_\cV}
\circ
(\id_a \vartriangleleft \eta^\cC_v \vartriangleleft \id_{1_\cV})
\circ
(\varepsilon^\cC_{a\to a\vartriangleleft u} \vartriangleleft 1_\cV)
=
\alpha_{a,u,1_\cV}.
$$
Since $\Phi$ is an equivalence of categories, $\Phi(\mu^\cI_{a,u}) =\alpha_{a,u,1_\cV}$ being invertible implies $\mu^\cI_{a,u}$ is invertible. 
Hence $\cI$ is tensored, and we are finished.
\end{proof}

\section{Closed \texorpdfstring{$\cV$}{V}-monoidal categories}
\label{sec:ClosedVmonoidal}

For the remainder of this article, $\cV$ will denote a braided monoidal category.

\subsection{\texorpdfstring{$\cV$}{V}-monoidal categories}

We now recall the basics of (strict) $\cV$-monoidal categories from \cite{1701.00567}.

\begin{defn}
A (strict) $\cV$-monoidal category consists of a $\cV$-category $\cC$, together with the following additional data:
\begin{itemize}
\item (identity object) a distinguished object $1_\cC\in \cC$
\item (tensor product of objects) for all $a,b\in \cC$, an object $ab\in \cC$
\item (tensor product for hom-objects) for all $a,b,c,d\in \cC$, a distinguished morphism $-\otimes_\cC - \in \cV( \cC(a\to b) \cC(c\to d) \to \cC(ac\to bd))$.
\end{itemize}
subject to the following axioms:
\begin{itemize}
\item (strict unitality for objects) for all $a\in \cC$, $1_\cC a = a = a1_\cC$.
\item (strict associator for objects) for all $a,b,c,\in \cC$, $a(bc)=(ab)c$.
\item (unitality) for all $a,b\in\cC$, $(j_{1_\cC} \id_{\cC(a\to b)})\circ (-\otimes_\cC-) = \id_{\cC(a\to b} = (\id_{\cC(a\to b)}j_{1_\cC})\circ (-\otimes_\cC-)$ and $(j_a j_b)\circ (-\otimes_\cC -) = j_{ab}$,
\item (associatitivity)
as morphisms in $\cV(\cC(a\to b)\cC(c\to d) \cC(e\to f) \to \cC(ace\to bdf))$, we have
$(\id (-\otimes_\cC -)) \circ (-\otimes_\cC-)=((-\otimes_\cC -)\id) \circ (-\otimes_\cC-)$,
and
\item (braided interchange) for all $a,b,c,d,e,f$, the following diagram commutes:
\begin{equation}
\label{eq:BraidedInterchance}
\begin{tikzcd}
\cC(a\to b) \cC(d\to e)\cC(b\to c) \cC(e\to f)
\ar[rr, "(-\otimes_\cC-)(-\otimes_\cC-)"]
\ar[dd, "\id\beta_{\cC(d\to e),\cC(b\to c)}\id"]
&&
\cC(ad\to be)\cC(be\to cf)
\ar[d, "-\circ_\cC-"]
\\
&&
\cC(ad\to ef)
\\
\cC(a\to b) \cC(b\to c) \cC(d\to e)\cC(e\to f)
\ar[rr, "(-\circ_\cC-)(-\circ_\cC-)"]
&&
\cC(a\to c) \cC(d\to f)
\ar{u}[swap]{-\otimes_\cC-}
\end{tikzcd}
\end{equation}
where $\beta$ is the braiding in $\cV$.
\end{itemize}
\end{defn}
It is straightforward to verify that the underlying category $\cC^\cV$ of a $\cV$-monoidal category is a monoidal category.

\begin{example}
\label{ex:SelfEnrichedVMonoidal}
When $\cV$ is braided closed monoidal, can make $\widehat{\cV}$ a $\cV$-monoidal category by defining $-\otimes_{\widehat{\cV}}-$ to be the mate of $(\id_u \beta_{w,[u,v]} \id_x)\circ (\varepsilon^{\widehat{\cV}}_{u\to v}\varepsilon^{\widehat{\cV}}_{w\to x})$ under the adjunction
$$
\cV(\widehat{\cV}(u\to v)\widehat{\cV}(w\to x) \to \widehat{\cV}(uw\to vx))
\cong
\cV(uw \widehat{\cV}(u\to v)\widehat{\cV}(w\to x) \to vx).
$$
We leave the details to the reader as a valuable exercise.
\end{example}

\begin{defn}[\cite{1701.00567}]
\label{def:VMonoidalFunctor}
Suppose $\cC, \cD$ are $\cV$-monoidal categories.
A (strong) strongly unital $\cV$-monoidal functor $\cC \to \cD$ consists of a pair $(\cF, \nu)$ where $\cF: \cC \to \cD$ is a $\cV$-functor such that $\cF(1_\cC) = 1_\cD$ and $\cF_{1_\cC \to 1_\cC} = j_{1_\cD}$, 
and 
$\nu=(\nu_{a,b})_{a,b\in\cC}$ is family of isomorphisms $\nu_{a,b} \in \cV(1_\cV \to \cD(\cF(ab) \to \cF(a)\cF(b)))$ which satisfy the following conditions:
\begin{itemize}
\item (unitality)
For all $a\in \cC$, $\nu_{a,1_\cC} = \id_{\cF(a)} = \nu_{1_\cC,a}$,\footnote{The axiom that $\nu_{a,1_\cC} = \id_{\cF(a)} = \nu_{1_\cC,a}$ for all $a\in \cC$ is not stated in \cite[Def.~2.6]{1701.00567}; rather only the condition that $\nu_{1_\cC, 1_\cC} = j_{1_\cD}$ appears.
This was an error, as this axiom is used throughout, and its verification was omitted (this gap is covered here).
See also Lemma \ref{lem:AlphaInvertibleWhenOneArgumentIs1} and Remark \ref{remark:ClassifyingFunctorExtraStrictUnitalityAxiom}.
}
\item (associativity)
for all $a,b,c\in \cC$, as composites in $\cD^\cV(\cF(abc) \to \cF(a)\cF(b)\cF(c))$, $\nu_{a,bc}\circ (\id_{\cF(a)}\nu_{b,c}) = \nu_{ab,c}\circ (\nu_{a,b}\id_{\cF(c)})$,

\item (naturality)
\begin{equation}
\label{eq:NaturalityForVMonoidal}
\begin{tikzcd}
\cC(a \to c)\cC(b\to d)
	\ar[r, "-\otimes_\cC-"]
	\ar[d,"\cF_{a\to c}\cF_{b\to d}"]
& \cC(ab \to cd) \ar[dr, "\cF_{ab \to cd}"]	
&
\\
\cD(\cF(a)\to \cF(c))\cD(\cF(b)\to \cF(d)) \ar[dr, "-\otimes_\cD -"]
& &
\cD(\cF(ab) \to \cF(cd)) \ar[d, "-\circ \nu_{c,d}"]	  
\\
& 
\cD(\cF(a)\cF(b) \to \cF(c)\cF(d)) \ar[r, "\nu_{a,b}\circ-"] 
&
\cD(\cF(ab) \to \cF(c)\cF(d))
\end{tikzcd}
\end{equation}
%
%
\end{itemize}
We leave the definition of a lax/oplax $\cV$-monoidal functor to the reader.
Note that the underlying functor of a strong/lax/oplax $\cV$-monoidal functor is strong/lax/oplax monoidal with the same tensorator/laxitor/oplaxitor.
\end{defn}

\begin{defn}
Suppose $\cC,\cD$ are $\cV$-monoidal categories and $(\cF,\nu^\cF), (\cG,\nu^\cG): \cC \to \cD$ are $\cV$-monoidal functors.
A strongly unital $\cV$-monoidal natural transformation $\sigma:(\cF,\nu^\cF) \Rightarrow (\cG,\nu^\cG)$ is a $1_\cV$-graded natural transformation $\sigma: \cF \Rightarrow \cG$ satisfying the naturality condition \eqref{eq:VNaturalTransformation} and the additional axioms
\begin{itemize}
\item
(unitality) $\sigma_{1_\cC} \in \cV(1_\cV \to \cD(\cF(1_\cC) \to \cG(1_\cC))) = \cV(1_\cV \to \cD(1_\cD \to 1_\cD))$ is equal to $j_{1_\cD}$,
\item
(monoidality)
for all $a,b\in\cC$, as composites in $\cD^\cV(\cF(ab) \to \cG(a)\cG(b))$,
$\sigma_{ab}\circ \nu^\cG_{a,b} = \nu^\cF_{a,b}\circ (\sigma_a \sigma_b)$.
\end{itemize}
\end{defn}

As before, we denote by $\cR_a : \cC^\cV \to \cV$ the representable functor $b\mapsto \cC(a\to b)$ from Example \ref{ex:Representable}.
When $\cV$ is closed, we denote by $\cR^a: \cC \to \widehat{\cV}$ the $\cV$-representable functor from Example \ref{ex:VRepresentable}.

\subsection{Closed \texorpdfstring{$\cV$}{V}-monoidal categories}

\begin{example}
\label{ex:aOtimes- is V-monoidal}
Suppose $\cC$ is $\cV$-monoidal and $a\in \cC$.
We can define a $\cV$-functor $a\otimes- :\cC \to \cC$ on objects by $a\otimes b = ab$ and $(a\otimes -)_{b \to c} \in \cV(\cC(b\to c) \to \cC(ab \to ac))$ is $(j_a \id_{\cC(b\to c)})\circ (-\otimes_\cC-)$.
It is straightforward to verify $a\otimes -$ is a $\cV$-functor using the braided interchange relation.

Notice that the underlying functor is given by $a\otimes - : \cC^\cV \to \cC^\cV$, where $a\otimes b = ab$ and for $f\in \cC^\cV(b\to c)$, $a\otimes f = \id_a \otimes f \in \cC^\cV(ab \to ac)$.
\end{example}

\begin{defn}
We call a $\cV$-monoidal category $\cC$ \emph{closed} if every $\cV$-functor $a\otimes -$ has a right $\cV$-adjoint, denoted $[a,-]: \cC \to \cC$.
\end{defn}

\begin{example}
Recall that a $\cV$-monoidal category is called \emph{rigid} if the underlying monoidal category $\cC^\cV$ is rigid.
As in \cite{1701.00567}, to ease the notation, we assume $(ba)^*=a^*b^*$ for all $a,b\in \cC$.

Notice that a rigid $\cV$-monoidal category is closed with $[a,-] = a^*\otimes-$, where the $\cV$-adjunction is witnessed via the Frobenius reciprocity isomorphisms
$$
\theta_{b,c}
=
[(\coev_a j_b) \circ (-\otimes_\cC-)][(j_{a^*} \id_{\cC(ab \to c)}) \circ (-\otimes_\cC-)]\circ (-\circ_\cC-)
\in
\cV(\cC(ab \to c) \to \cC(b \to a^*c)).
$$
(It is an important exercise using the braided interchange relation to verify the above morphism is a natural isomorphism with the obvious inverse.)
\end{example}

\begin{lem}
\label{lem:ClosedAdjoints}
If $\cC$ is closed, then $\cC$ is oplax tensored if and only if $R_{1_\cC}$ admits a left adjoint.
\end{lem}
\begin{proof}
If $\cR_{1_\cC}: \cC^\cV \to \cV$ admits a left adjoint $\cF: \cV \to \cC^\cV$, we get a left adjoint $\cL_a :\cV\to \cC^\cV$ by $v\mapsto a\cF(v)$:
\begin{equation}
\begin{split}
\label{eq:ClosedR_aAdjunction}
\cV(v \to \cR_a(b))
&=
\cV(v \to \cC(a\to b))
\cong
\cV(v \to \cC(1_\cC\to [a,b]))
\\&\cong
\cC^\cV(\cF(v) \to [a,b])
\cong
\cC^\cV(a\cF(v)\to b)
=:
\cC^\cV(\cL_a(v)\to b).
\end{split}
\end{equation}
The other direction is trivial.
\end{proof}

\subsection{Classification of rigid \texorpdfstring{$\cV$}{V}-monoidal categories}

In \cite{1701.00567}, we proved a classification theorem for rigid $\cV$-monoidal categories, which is an analog of Theorem \ref{thm:OplaxVMod}.

\begin{thm}[{\cite[Thm.~1.1]{1701.00567}}]
\label{thm:OplaxVMonoidal}
Let $\cV$ be a braided monoidal category.
There is a bijective correspondence
\[
\left\{\,
\parbox{7cm}{\rm Rigid $\cV$-monoidal categories $\cC$ such that $\cR_{1_\cC} : \cC^\cV \to \cV$ admits a left adjoint}\,\right\}
\,\,\cong\,\,
\left\{\,\parbox{8cm}{\rm Pairs $(\cT,\cF^{\scriptscriptstyle Z})$ with $\cT$ a rigid monoidal category and $\cF^{\scriptscriptstyle Z}: \cV\to Z(\cT)$ braided oplax monoidal, such that $\cF:=\cF^{\scriptscriptstyle Z}\circ R$ admits a right adjoint}\,\right\}.
\]
\end{thm}
Here, $R: Z(\cT) \to \cT$ denotes the forgetful functor.

\begin{remark}
As in \cite{1701.00567}, we abuse nomenclature by assuming our oplax monoidal functors $(\cF, \nu): \cV \to \cT$ are strongly unital; that is 
$\cF(1_\cV) = 1_\cT$ and $\nu_{v, 1_\cV} = \id_{\cF(v)} = \nu_{1_\cV, v}$ for all $v\in \cV$.
See also Remark \ref{remark:ClassifyingFunctorExtraStrictUnitalityAxiom}.
\end{remark}

\subsection{Classification of closed \texorpdfstring{$\cV$}{V}-monoidal categories}
\label{sec:ClassificationOfClosedVMonoidalCategories}

It is straightforward to generalize Theorem \ref{thm:OplaxVMonoidal}, relaxing rigidity on both sides to closed, obtaining a bijective correspondence between closed $\cV$-monoidal categories and pairs $(\cT, \cF^{\scriptscriptstyle Z})$ with $\cT$ closed.
This is a better analog of Theorem \ref{thm:OplaxVMod} than Theorem \ref{thm:OplaxVMonoidal}.

\begin{thm*}[Thm.~\ref{thm:ClosedOplaxVMonoidal}]
Let $\cV$ be a braided monoidal category.
There is a bijective correspondence
\[
\left\{\,
\parbox{7cm}{\rm Closed $\cV$-monoidal categories $\cC$ such that $\cR_{1_\cC}$ admits a left adjoint}\,\right\}
\,\,\cong\,\,
\left\{\,\parbox{8.3cm}{\rm Pairs $(\cT,\cF^{\scriptscriptstyle Z})$ with $\cT$ a closed monoidal category and $\cF^{\scriptscriptstyle Z}: \cV\to Z(\cT)$ braided oplax monoidal, such that $\cF:=\cF^{\scriptscriptstyle Z}\circ R$ admits a right adjoint}\,\right\}.
\]
\end{thm*}

We now give a brief description of both directions, together with a proof of the essential lemmas needed when $\cC$ (respectively $\cT$) is closed rather than rigid.

Starting with a $\cV$-monoidal $\cC$ such that $\cR_{1_\cV}: \cC^\cV \to \cV$ admits a left adjoint $\cF: \cV \to \cC^\cV$, we let $\eta$ be the unit of the adjunction.
We see that $\cF$ can be endowed with the structure of a strongly unital oplax monoidal functor by defining $\nu_{u,v}\in \cC^\cV(\cF(uv) \to \cF(u)\cF(v))$ to be the mate of 
$(\eta^\cC_{1_\cC,u}\eta^\cC_{1_\cC,v})\circ (-\otimes_\cC-)$ under the adjunction
$$
\cC^\cV(\cF(uv) \to \cF(u)\cF(v))
\cong
\cV(uv \to \cC(1_\cC \to \cF(u)\cF(v))).
$$

\begin{remark}
\label{remark:ClassifyingFunctorExtraStrictUnitalityAxiom}
In \cite[Lem.~4.4]{1701.00567}, it was shown that $\cF(1_\cV)$ is canonically isomorphic to $1_{\cC^\cV}$ via the Yoneda Lemma, and thus we may identify $\cF(1_\cV) = 1_{\cC^\cV}$.
That paper (in error!) omitted to show the easily verified that under this identification, $\nu_{1_\cV, v} = \id_{\cF(v)} = \nu_{v, 1_\cV}$ for all $v\in \cV$.
Indeed, the verification is the same as the proof of Lemma \ref{lem:AlphaInvertibleWhenOneArgumentIs1}.
\end{remark}

Since $\cC$ is closed, we have a $\cV$-adjunction $a\otimes - \dashv_\cV [a,-]$ between $\cV$-functors $\cC\to \cC$.
(Notice that $\cC^\cV$ is obviously closed.)
This means there is a family of isomorphims $\theta_{b,c} \in \cV(\cC(ab \to c)\to \cC(b\to [a,c]))$ satisfying \eqref{eq:V-Adjoint1} and \eqref{eq:V-Adjoint2}.
Now we can understand the adjunction \eqref{eq:ClosedR_aAdjunction}:
starting with $g\in \cC^\cV(a\cF(v)\to b)$, its mate is given by
$(\eta_v (g\circ \theta_{\cF(v),b}^{a}))\circ (-\circ_\cC-)\circ \kappa_{1,b}^{a}$.

We now state the essential lemma, whose proof is now easier and more conceptual than the proof of \cite[Lem.~4.6 and Appendix B]{1701.00567}!

\begin{lem}
\label{lem:MateOfIdOfTensorClosed}
The mate of $\id_{a\cF(v)}$ under Adjunction \eqref{eq:ClosedR_aAdjunction} is $(j_a \eta_v)\circ (-\otimes_\cC-)$.
\end{lem}
\begin{proof}
Setting $b=a\cF(v)$, we see that the mate of $j_{a\cF(v)}=(j_aj_{\cF(v)})\circ(-\otimes_\cC-)$ is given by
\begin{align*}
(
\eta_v &
[(j_aj_{\cF(v)})\circ(-\otimes_\cC-)\circ \theta_{\cF(v),a\cF(v)}]
)
\circ
(-\circ_\cC-)
\circ
\kappa_{1_\cC , a\cF(v)}
\\&=
(
[(j_a\eta_v)\circ(-\otimes_\cC-)]
[(j_aj_{\cF(v)})\circ(-\otimes_\cC-)\circ \theta_{\cF(v),a\cF(v)}\circ \kappa_{\cF(v), a\cF(v)}]
)
\circ
(-\circ_\cC-)
&&
\text{by \eqref{eq:V-Adjoint1}}
\\&=
(
[(j_a\eta_v)\circ(-\otimes_\cC-)]
[(j_aj_{\cF(v)})\circ(-\otimes_\cC-)]
)
\circ
(-\circ_\cC-)
\\&=
(
j_a
[(\eta_v j_{\cF(v)})\circ (-\circ_\cC-)]
)
\circ
(-\otimes_\cC-)
&&
\text{by \eqref{eq:BraidedInterchance}}
\\&=
(j_a \eta_v)\circ (-\otimes_\cC-).
&&
\qedhere
\end{align*}
\end{proof}

\begin{remark}
\label{rem:OtherMateOfMuForVMonoidal}
This lemma has many important implications.
First, one can use this lemma and the braided exchange relation to prove that 
$$
(\eta_{a,u}^\cC \eta_{a\cF(u), v}^\cC) \circ (-\circ_\cC-)
=
(\eta_{1_\cC,u}^\cC \eta_{\cF(u), v}^\cC) \circ (-\circ_\cC-) \circ (j_a \otimes_\cC -)
=
(\eta_{1_\cC,u}^\cC \eta_{1_\cC, v}^\cC) \circ (-\otimes_\cC-) \circ (j_a \otimes_\cC -).
$$
These maps above are thus all equal to the mate of $\alpha_{a, u,v}$ from \eqref{eq:MateOfAlpha} where the right $\cV$-module structure of $\cC$ is given by $a\vartriangleleft u := a\cF(u)$.
This also implies that $\alpha_{a,u,v} = \id_a \nu_{u,v}$ for all $a\in \cC$ and $u,v\in \cV$.
Setting $a=1_\cC$, we have $\nu_{u,v} = \alpha_{1_\cC, u,v}$.
We will use these facts heavily in the proof of Lemma \ref{lem:VMonoidalAdjoint} below.
\end{remark}

With Lemma \ref{lem:MateOfIdOfTensorClosed} in hand, we can lift $\cF: \cV\to \cC^\cV$ to a braided oplax monoidal functor $\cF^{\scriptstyle Z}: \cV \to Z(\cC^\cV)$ by defining a half-braiding $e_{a, \cF(v)} \in \cC^\cV(a\cF(v)\to \cF(v)a)$ to be the mate of $(\eta_v j_a)\circ (-\otimes_\cC-)$ under Adjunction \eqref{eq:ClosedR_aAdjunction}.
Indeed, all proofs in \cite[\S4 and 5]{1701.00567} now apply verbatim to the closed (rather than rigid) case, using Lemma \ref{lem:MateOfIdOfTensorClosed} rather than \cite[Lem.~4.6]{1701.00567}.

In the other direction, starting with a closed monoidal category $\cT$ and a strongly unital braided oplax monoidal functor $(\cF^{\scriptscriptstyle Z},\nu):\cV \to Z(\cT)$ such that $\cF := \cF^{\scriptscriptstyle Z}\circ R$ admits a right adjoint $\cR$ (where $R: Z(\cT) \to \cT$ is the forgetful functor), we notice that the functors $\cL_a: \cV \to \cT$ by $\cL_a(v)=a\cF(v)$ admit right adjoints:
\begin{equation}
\label{eq:V-T Adjunction Closed}
\begin{split}
\cT(\cL_a(v)\to b)
=
\cT(a\cF(v)\to b)
\cong
\cT(\cF(v)\to [a,b])
\cong
\cV(v\to \cR([a,b]))
\end{split}
\end{equation}
We thus define a $\cV$-category $\cT\dslash\cF$ as follows:
\begin{itemize}
\item
The objects of $\cT\dslash \cF$ are the same as those of $\cT$,
\item
The hom objects are given by $\cT\dslash\cF(a\to b)=\cR([a,b])$.
\item
the identity morphism
$j_a\in\cV(1_\cV\to \cR([a,a]))$ is the mate of $\id_a\in \cT(a\to a)$,
\item
Using the notation\,\footnote{
In \cite[Def.~6.2]{1701.00567}, we used the notation $[a\to b] := \cF(\cT\dslash\cF(a\to b))$.
However, now that we use $[a,-]$ for the right $\cV$-adjoint of $a\otimes -$ for $a\in \cT\dslash\cF$, we use $\langle a\to b\rangle$ to not overload the bracket notation.
} 
\begin{align*}
\langle a\to b\rangle
&:= 
\cF(\cT\dslash \cF(a\to b))
\text{ and} 
\\
\langle a\to b; b\to c;\cdots\rangle&:= \cF(\cT\dslash \cF(a\to b)\cT\dslash \cF(b\to c) \cdots ),
\end{align*}
we let $\varepsilon^{\cT\dslash\cF}_{a\to b} \in \cT(a\langle a\to b\rangle \to b)$ be the counit of Adjunction \eqref{eq:V-T Adjunction Closed}.
The composition $-\circ_{\cT \dslash\cF}- \in \cV(\cT\dslash\cF(a\to b)\cT\dslash\cF(b\to c) \to \cT\dslash\cF(a\to c))$ is the mate of
$(\id_a\nu_{\cT\dslash\cF( a\to b), \cT\dslash\cF (b\to c)})\circ (\varepsilon^{\cT\dslash\cF}_{a\to b} \id_{\langle b\to c\rangle})\circ \varepsilon^\cT_{b\to c}$.
\item
The tensor product $-\otimes_{\cT \dslash\cF}- \in \cV(\cT\dslash\cF(a\to b)\cT\dslash\cF(c\to d) \to \cT\dslash\cF(ac\to bd))$ is the mate of
$(\id_{ac}\nu_{\cT\dslash\cF( a\to b), \cT\dslash\cF (c\to d)})\circ(\id_a e_{c,\langle a\to b\rangle} \id_{\langle c\to d\rangle}) \circ (\varepsilon^{\cT\dslash\cF}_{a\to b} \varepsilon^{\cT\dslash\cF}_{c\to d})$.
\end{itemize}
The verification that $\cT\dslash \cF$ is a $\cV$-monoidal category is identical to \cite[\S 6.1-6.4]{1701.00567}, as is the monoidal equivalence $\cT \cong \cT\dslash \cF^\cV$.
Clearly $\cR_{1_{\cT\dslash \cF}}$ admits a left adjoint, as
$$
\cV(v \to \cR_{1_{\cT\dslash \cF}}(a))
=
\cV(v \to \cR([1_\cT, a]))
\cong
\cV(\cF(v) \to [1_\cT,a])
\cong
\cT(\cF(v)\to a)
\cong
\cT\dslash\cF(\cF(v)\to a).
$$

It remains to show that $\cT\dslash \cF$ is closed.
First, we can describe the right oplax $\cV$-module structure of $\cT\dslash \cF^\cV$ in more detail.

\begin{remark}
\label{rem:OplaxVModuleStructureOfT}
Since the hom objects of $ \cT\dslash\cF^\cV$ are defined using Adjunction \eqref{eq:V-T Adjunction Closed}, under the monoidal equivalence $\cT \cong \cT\dslash\cF^\cV$, we may identify the oplax right $\cV$-module structure by
$a\vartriangleleft v = a\cF(v)$ and $\alpha_{a,u,v} = \id_a \nu_{u,v}$ as in Remark \ref{rem:OtherMateOfMuForVMonoidal}.
\end{remark}

Recall that by Example \ref{ex:aOtimes- is V-monoidal}, for each $a\in \cT\dslash \cF$, we have a $\cV$-monoidal functor $a\otimes-:\cT\dslash \cF \to \cT\dslash \cF$ whose underlying functor is $a\otimes-: \cT \to \cT$.

\begin{defn}
We denote by $\varepsilon^\cT$ the counit of the adjunction $\cT(ab\to c)\cong \cT(b \to [a,c])$, which should not be confused with $\varepsilon^{\cT\dslash \cF}$!
We define $[a,-]_{b\to c}$ to be the mate of $(\varepsilon^\cT_{a\to b} \id_{\langle b\to c\rangle})\circ \varepsilon^{\cT\dslash\cF}_{b\to c}$ under the adjunction
\begin{equation}
\label{eqn: T closed adjunction}
\begin{split}
\cV(\cT\dslash \cF(b\to c) \to \cT\dslash\cF([a,b]\to [a,c]))
&=
\cV(\cT\dslash \cF(b\to c) \to \cR([[a,b], [a,c]]))
\\&\cong
\cT([a,b]\langle b\to c\rangle \to [a,c])
\\&\cong
\cT(a[a,b]\langle b\to c\rangle \to c).
\end{split}
\end{equation}

To verify $[a,-]$ is a $\cV$-functor, we calculate that the mates of $([a,-]_{b\to c}[a,-]_{c\to d})\circ (-\circ_{\cT\dslash \cF}-)$ and
$(-\circ_{\cT\dslash \cF}-)\circ [a,-]_{b\to d}$
agree under Adjunction \eqref{eqn: T closed adjunction}.
The mate of the former is given by
\begin{align*}
(\id_a \id_{[a,b]}(\cF([a,-]_{b\to c}[a,-]_{c\to d})
&\circ
\nu_{\cT\dslash\cF([a,b]\to[a,c]),\cT\dslash\cF([a,c]\to [a,d])}))
\\&
\circ
(\id_a \varepsilon^{\cT\dslash \cF}_{[a,b]\to [a,c]} \id_{\langle[a,c]\to [a,d]\rangle})
\circ
(\id_a \varepsilon^{\cT\dslash \cF}_{[a,c]\to [a,d]})
\circ
\varepsilon^\cT_{a\to d}
\\&=
(\id_a \id_{[a,b]} \nu_{\cT\dslash\cF(b\to c), \cT\dslash\cF(c\to d)})
\circ
(\id_a \id_{[a,b]} \cF([a,-]_{b\to c})\cF([a,-]_{c\to d}))
\\&\hspace{3cm}
\circ
(\id_a \varepsilon^{\cT\dslash \cF}_{[a,b]\to [a,c]} \id_{\langle [a,c]\to [a,d]\rangle})
\circ
(\id_a \varepsilon^{\cT\dslash \cF}_{[a,c]\to [a,d]})
\circ
\varepsilon^\cT_{a\to d}
\\&=
(\id_a \id_{[a,b]} \nu_{\cT\dslash\cF(b\to c), \cT\dslash\cF(c\to d)})
\circ
(\id_a \id_{[a,b]} \cF([a,-]_{b\to c})\id_{\langle c\to d\rangle})
\\&\hspace{3cm}
\circ
(\id_a \varepsilon^{\cT\dslash \cF}_{[a,b]\to [a,c]} \id_{\langle c\to d\rangle})
\circ
(\varepsilon^{\cT}_{a\to c}\id_{\langle c\to d\rangle})
\circ
\varepsilon^{\cT\dslash \cF}_{c\to d}
\\&=
(\varepsilon^{\cT}_{a\to b} \nu_{\cT\dslash\cF(b\to c), \cT\dslash\cF(c\to d)})
\circ
(\varepsilon^{\cT\dslash \cF}_{b\to c}\id_{\langle c\to d\rangle})
\circ
\varepsilon^{\cT\dslash \cF}_{c\to d}
\\&=
(\varepsilon^{\cT}_{a\to b} \cF(-\circ_{\cT\dslash \cF}-))
\circ
\varepsilon^{\cT\dslash \cF}_{b\to d},
\end{align*}
which is exactly the mate of the latter.
\end{defn}

\begin{lem}
\label{lem:UnderlyingAdjunctionOfTensorInT}
The underlying functor $(a\otimes-)^\cV$ is left adjoint to $[a,-]^\cV$.
\end{lem}
\begin{proof}
Under the equivalence $\cT \cong \cT\dslash \cF^\cV$, the underlying functor $(a\otimes-)^\cV$ is just $a\otimes - : \cT\to \cT$, and similarly for $[a,-]^\cV$.
Since $\cT$ is closed, we are finished.
\end{proof}

\begin{prop}
The $\cV$-functor $a\otimes-$ is left $\cV$-adjoint to $[a,-]$.
\end{prop}
\begin{proof}
By Lemma \ref{lem:UnderlyingAdjunctionOfTensorInT} and Theorem \ref{thm:LiftToVAdjunction}, it suffices to prove that $a\otimes -: \cT\dslash\cF \to \cT\dslash\cF$ is tensored.
For $b\in\cT\dslash\cF$ and $v\in \cV$, $\mu^{a\otimes -}_{b,v} \in \cT\dslash\cF^\cV(ab\vartriangleleft v \to (a\vartriangleleft v)b) = \cT(ab \cF(v) \to a\cF(v)b)$ is given by the mate of $\eta^{\cT\dslash\cF}_{b,v}\circ (a\otimes -)_{b \to b\vartriangleleft v}$ under the adjunction
$$
 \cT\dslash\cF^\cV(ab\vartriangleleft v \to a(b\vartriangleleft v))
 \cong
 \cV(v \to \cT\dslash\cF(ab \to a(b\vartriangleleft v))).
$$
By Remark \ref{rem:OplaxVModuleStructureOfT}, identifying $\cT\dslash\cF^\cV = \cT$ identifies $a\vartriangleleft u = a\cF(u)$ and 
$\cT\dslash\cF^\cV(ab\vartriangleleft v \to a(b\vartriangleleft v)) = \cT(ab\cF(v) \to ab\cF(v))$.
We claim that under this identification, $\mu^{a\otimes -}_{b,v} = \id_{ab\cF(v)}$, which is obviously invertible.
Indeed, by Example \ref{ex:aOtimes- is V-monoidal}, 
$$
\eta^{\cT\dslash\cF}_{b,v}\circ (a\otimes -)_{b \to b\vartriangleleft v} 
= 
\eta_{b,v}^{\cT\dslash\cF} \circ (j_a \id_{\cT\dslash\cF(b\to b\cF(v))}) \circ (-\otimes_{\cT\dslash \cF}-), 
$$
whose mate under the above adjunction is given by
\begin{align*}
(\id_{ab}\cF(\eta_{b,v}^{\cT\dslash\cF}& \circ (j_a \id_{\cT\dslash\cF(b\to b\cF(v))})))
\circ
(\id_{ab}\nu_{\cT\dslash\cF(a\to a),\cT\dslash\cF(b\to b\cF(v))})
\circ
(\id_a e_{b,\langle a\to a\rangle} \id_{\langle b\to b\cF(v)\rangle})
\circ
(\varepsilon^{\cT\dslash\cF}_{a\to a}\varepsilon^{\cT\dslash\cF}_{b\to b\cF(v)})
\\&=
(\id_{ab}[\cF(\eta_{b,v}^{\cT\dslash\cF}) \circ (j_{\cF(a)} \id_{\cT\dslash\cF(b\to b\cF(v))})])
\circ
(\id_a e_{b,\langle a\to a\rangle} \id_{\langle b\to b\cF(v)\rangle})
\circ
(\varepsilon^{\cT\dslash\cF}_{a\to a}\varepsilon^{\cT\dslash\cF}_{b\to b\cF(v)})
\\&=
[(\id_{a}j_{\cF(a)})\circ \varepsilon^{\cT\dslash\cF}_{a\to a}]
[(\id_{b}\cF(\eta_{b,v}^{\cT\dslash\cF}))\circ \varepsilon^{\cT\dslash\cF}_{b\to b\cF(v)}]
\\&=
\id_{a}\id_{b\cF(v)}.
\end{align*}
This completes the proof.
\end{proof}

Finally, \cite[\S7]{1701.00567} remains unchanged.
This concludes the proof of Theorem \ref{thm:ClosedOplaxVMonoidal}.
\qed

\section{Tensored \texorpdfstring{$\cV$}{V}-monoidal categories}

In this section, we assume that $\cV$ is a braided closed monoidal category so that we may form the $\cV$-monoidal category $\widehat{\cV}$ as in Example \ref{ex:SelfEnrichedVMonoidal}.

\subsection{Tensored \texorpdfstring{$\cV$}{V}-monoidal categories}

\begin{defn}
Similar to an ordinary $\cV$-category, we call a $\cV$-monoidal category \emph{tensored} if the $\cV$-representable functors $\cR^a=\cC(a\to -): \cC \to \widehat{\cV}$ admit left $\cV$-adjoints.
\end{defn}

Suppose $\cC$ is closed and $\cR^{1_\cC}=\cC(1_\cC\to -): \cC \to \widehat{\cV}$ admits a left $\cV$-adjoint $\cF$.
By Remark \ref{rem:UnderlyingAdjunction}, the underlying functor $\cF^\cV : \cV \to \cC^\cV$ is left adjoint to the underlying functor $\cR_{1_\cC}$.
As in the proof of Theorem \ref{thm:ClosedOplaxVMonoidal}, every representable functor $\cR_a=\cC(a\to -)^\cV : \cC^\cV \to \cV$ admits a left adjoint $\cL_a$.
This allows us to canonically equip $\cF^\cV$ with the structure of a strongly unital braided oplax monoidal functor $((\cF^\cV)^{\scriptscriptstyle Z},\nu): \cV \to Z(\cC^\cV)$.
In fact, we will see that $((\cF^\cV)^{\scriptscriptstyle Z},\nu)$ is strong monoidal.

\begin{lem}
\label{lem:VMonoidalAdjoint}
Suppose $\cC$ is closed and $\cR^{1_\cC}=\cC(1_\cC\to -): \cC \to \widehat{\cV}$ admits a left $\cV$-adjoint $\cF:\widehat{\cV} \to \cC$.
\begin{enumerate}
\item
$\cF$ is tensored with $(\mu^\cF_{u,v})^{-1} = \nu_{u,v}$ for all $u,v\in \widehat{\cV}$.
\item
$((\cF^\cV)^{\scriptscriptstyle Z},\nu): \cV \to Z(\cC^\cV)$ is strong monoidal.
\item
The morphisms $\nu_{u,v}\in \cC^\cV(\cF(uv) \to \cF(u)\cF(v))$ endow $\cF$ with the structure of a $\cV$-monoidal functor.
\item
$\cC$ is tensored.
\end{enumerate}
\end{lem}
\begin{proof}
\mbox{}
\item[(1)]
First, $\cF$ is tensored by Theorem \ref{thm:LiftToVAdjunction}.
As explained in Remark \ref{rem:OtherMateOfMuForVMonoidal}, notice that $\nu_{u,v} = \alpha_{1_\cC, u,v}$ where the right $\cV$-module structure of $\cC$ is given by $a\vartriangleleft u := a\cF(u)$.
Now when $\cF$ is a left $\cV$-adjoint of $\cC(1_\cC \to -)$, we have $\alpha_{1_\cC, u,v}$ is invertible with inverse equal to $\mu^{\cL^{1_\cC}}_{u,v} = \mu^\cF_{u,v}$ by setting $a=1_\cC$ in the proof of Proposition \ref{prop:AlphaIsInvertible}.

\item[(2)]
By (1), $\nu_{u,v}\in \cV(1_\cV \to \cC(\cF(uv)\to \cF(u)\cF(v)))$ is invertible for all $u,v\in \cV$, and the result follows immediately.

\item[(3)]
Since $(\cF^\cV,\nu)$ is strong monoidal by (1), we see that the $\nu_{u,v}$ automatically satisfy the unitality and associativity axioms for $(\cF, \nu)$ to be a $\cV$-monoidal functor, since these are merely properties of the underlying functor $\cF^\cV$.
It remains to prove the naturality condition \eqref{eq:NaturalityForVMonoidal}.
Under the adjunction
$$
\cV(\widehat{\cV}(u\to w)\widehat{\cV}(v\to x)\to \widehat{\cV}(\cF(uv) \to \cF(w)\cF(x)))
\cong
\cC^\cV( \cF(uv)\widehat{\cV}(u\to w)\widehat{\cV}(v\to x) \to \cF(w)\cF(x)),
$$
the mate of
$(\cF_{u\to w}\cF_{v\to x})\circ (-\otimes_{\cC}-)\circ (\nu_{u,v}\circ_{\cC}-)$
is given by
\begin{align*}
(\id_{\cF(uv)}&[\cF(
(\cF_{u\to w}\cF_{v\to x})\circ (-\otimes_{\cC}-)\circ (\nu_{u,v}\circ_{\cC}-)
)]) 
\\&\hspace{2.8cm}\circ
(\varepsilon^\cC_{\cF(uv) \to \cF(u)\cF(v)} \id_{\cC(\cF(u)\cF(v) \to \cF(w)\cF(x))})
\circ
\varepsilon^\cC_{\cF(u)\cF(v) \to \cF(w)\cF(x)}
\\&=
(\nu_{u,v} \cF(\cF_{u\to w}\cF_{v\to x}))
\circ
(\id_{\cF(u)\cF(v)} \nu_{\widehat{\cV}(u\to w),\widehat{\cV}(v\to x)})
\circ
(\id_{\cF(u)} e_{\cF(v),\langle u\to w\rangle}\id_{\langle v\to x\rangle})
\circ
(\varepsilon^\cC_{\cF(u) \to \cF(w)}\varepsilon^\cC_{\cF(v)\to \cF(x)})
\\&=
(\nu_{u,v} \nu_{\widehat{\cV}(u\to w),\widehat{\cV}(v\to x)})
\circ
(\id_{\cF(u)} e_{\cF(v),\langle u\to w\rangle}\id_{\langle v\to x\rangle})
\\&\hspace{3.2cm}
\circ
(
[(\id_{\cF(u)}\cF(\cF_{u\to w})  )\circ \varepsilon^\cC_{\cF(u) \to \cF(w)}]
[(\id_{\cF(v)}\cF(\cF_{v\to x}))\circ \varepsilon^\cC_{\cF(v)\to \cF(x)}]
)
\\&=
(\nu_{u,v} \nu_{\widehat{\cV}(u\to w),\widehat{\cV}(v\to x)})
\circ
(\id_{\cF(u)} e_{\cF(v),\langle u\to w\rangle}\id_{\langle v\to x\rangle})
\\&\hspace{3.2cm}
\circ
(
[\nu^{-1}_{\cF(u), \widehat{\cV}(u\to w)}\circ \cF(\varepsilon^{\widehat{\cV}}_{u\to w})]
[\nu^{-1}_{\cF(v), \widehat{\cV}(v\to x)}\circ \cF(\varepsilon^{\widehat{\cV}}_{v\to x})]
)
\\&=
\nu^{-1}_{uv, \widehat{\cV}(u\to w)\widehat{\cV}(v\to x)}
\circ
\cF(
(\id_{u} \beta_{v, \widehat{\cV}(u\to w)} \id_{\widehat{\cV}(v\to x)})
\circ
(\varepsilon^{\widehat{\cV}}_{u\to w} \varepsilon^{\widehat{\cV}}_{v\to x})
)
\circ
\nu_{v,x},
\end{align*}
which is also equal to the mate of $(-\otimes_{\widehat{\cV}}-) \circ \cF_{uv\to wx} \circ (-\circ_\cC \nu_{w,x})$ through a similar calculation.
Here, we have used that
\begin{equation}
\label{eq:MateOf F u to w}
(\id_{\cF(u)}\cF(\cF_{u\to w})  )\circ \varepsilon^\cC_{\cF(u) \to \cF(w)} 
= 
\nu^{-1}_{u, \widehat{\cV}(u\to w)} \circ \cF(\varepsilon^{\widehat{\cV}}_{u \to w})
\end{equation}
and the analogous statement replacing $u$ and $w$ with $v$ and $x$ respectively.
Indeed, since $\nu_{u,\widehat{\cV}(u\to w)}^{-1} = \mu^\cF_{u,\widehat{\cV}(u\to w)}$ by Proposition \ref{prop:AlphaIsInvertible} as in (1), both maps in \eqref{eq:MateOf F u to w} are the mate of $\cF_{u\to w}$ under the adjunction
$$
\cC^\cV(\cF(u)\cF(\widehat{\cV}(u\to w)) \to \cF(w))
\cong
\cV(\widehat{\cV}(u\to w) \to \cC(\cF(u) \to \cF(w))).
$$

\item[(4)]
Recall from Remark \ref{rem:OtherMateOfMuForVMonoidal} that the right oplax $\cV$-module structure of $\cC$ is given by $a\vartriangleleft u := a\cF(u)$, and $\alpha_{a,u,v} = \id_a \nu_{u,v}$ for all $a\in \cC$ and $u,v\in \cV$.
Since $\nu_{u,v}$ is invertible by (1), we see $\alpha_{a,u,v}$ is invertible for every $a\in \cC$ and $u,v\in \cV$.
Hence $\cC$ is tensored by Proposition \ref{prop:LaTensored} and Corollary \ref{cor:AlphaInvertibleImpliesCTensored}.
\end{proof}

The following corollary is now immediate.

\begin{cor}
A closed $\cV$-monoidal category is tensored if and only if $\cR^{1_\cC}$ admits a left $\cV$-adjoint.
\end{cor}



%





\subsection{Classification of tensored closed \texorpdfstring{$\cV$}{V}-monoidal categories}
\label{sec:ClassificationOfTensoredClosedVMonoidalCategories}

We now prove the analog of Theorem \ref{thm:StrongVMod} for $\cV$-monoidal categories.
Recall that a monoidal functor $(\cF, \nu)$ is called \emph{strong} if $\nu$ is a natural isomorphism.

\begin{thm*}[Theorem \ref{thm:TensoredVMonoidalEquivalence}]
Let $\cV$ be a closed monoidal category.
Under Theorem \ref{thm:ClosedOplaxVMonoidal}, there is a bijective correspondence
\[
\left\{\,
\parbox{6.9cm}{\rm Tensored closed $\cV$-monoidal categories}
\,\right\}
\,\,\cong\,\,
\left\{\,\parbox{8.3cm}{\rm Pairs $(\cT,\cF^{\scriptscriptstyle Z})$ with $\cT$ a closed monoidal category and $\cF^{\scriptscriptstyle Z}: \cV\to Z(\cT)$ braided strong monoidal, such that $\cF:=\cF^{\scriptscriptstyle Z}\circ R$ admits a right adjoint}\,\right\}.
\]
We also get a bijective correspondence replacing closed with rigid on both sides above.
\end{thm*}

\begin{proof}
Under the bijective correspondence from Theorem \ref{thm:ClosedOplaxVMonoidal}, it suffices to prove that $\cC$ tensored implies $(\cF^{\scriptscriptstyle Z},\nu): \cV \to Z(\cC^\cV)$ is strong monoidal, and that $(\cF^{},\nu): \cV \to Z(\cT)$ being strong monoidal implies $\cT\dslash\cF$ is tensored.

First, when $\cC$ is closed and tensored, $\cR^{1_\cC}=\cC(1_\cC \to -): \cC \to \widehat{\cV}$ admits a left $\cV$-adjoint $\cF: \cV \to \widehat{\cC}$.
By Remark \ref{rem:UnderlyingAdjunction}, the underlying functor $\cF^\cV$ is a left adjoint to the underlying functor $\cR_{1_\cC}$, and can be endowed with the structure of a braided oplax monoidal functor $((\cF^\cV)^{\scriptscriptstyle Z},\nu) : \cV \to Z(\cC^\cV)$.
By Lemma \ref{lem:VMonoidalAdjoint}, we see that $((\cF^\cV)^{\scriptscriptstyle Z},\nu)$ is strong monoidal.

Conversely, suppose $\cF^{\scriptscriptstyle Z}: \cV \to Z(\cT)$ is strong monoidal.
By (4) of Lemma \ref{lem:VMonoidalAdjoint}, it suffices to show we can promote $\cF = \cF^{\scriptscriptstyle Z} \circ R : \cV \to \cT$ (where $R: Z(\cT)\to \cT$ is the forgetful functor) to a $\cV$-functor $\cF: \widehat{\cV} \to \cT\dslash\cF$ which is left $\cV$-adjoint to $\cT\dslash\cF(1_\cT\to -)$.
Since $\cF$ is strong monoidal, we can define $\cF_{u \to v}$ as mate of $\nu^{-1}_{u,\widehat{\cV}(u\to v)}\circ \cF(\varepsilon^{\widehat{\cV}}_{u\to v})$
as on the right hand side of \eqref{eq:MateOf F u to w}
under the adjunction
$$
\cV(\widehat{\cV}(u\to v)\to \cT\dslash\cF(\cF(u)\to \cF(v)))
\cong
\cT(\cF(u)\cF(\widehat{\cV}(u\to v))\to \cF(v)).
$$
Under the adjunction
$$
\cV(\widehat{\cV}(u\to v)\widehat{\cV}(v\to w)\to \cT\dslash\cF(\cF(u)\to \cF(w)))
\cong
\cT\dslash\cF^\cV( \cF(u)\cF(\widehat{\cV}(u\to v)\widehat{\cV}(v\to w) \to \cF(u)),
$$
we see the mate of
$(\cF_{u\to v}\cF_{v\to w}) \circ (-\circ_{\cT\dslash\cF}-)$ 
is given by
\begin{align*}
&(\id_{\cF(u)} \cF(\cF_{u\to v}\cF_{v\to w}))
\circ
(\id_{\cF(u)} \circ \nu_{\cT\dslash\cF(u\to v), \cT\dslash\cF(v\to w)})
\circ
(\varepsilon^{\cT\dslash\cF}_{\cF(u) \to \cF(v)} \id_{\langle v\to w\rangle})
\circ
\varepsilon^{\cT\dslash\cF}_{\cF(v) \to \cF(w)}
\\&=
(\id_{\cF(u)} \circ \nu_{\widehat{\cV}(u\to v), \widehat{\cV}(v\to w)})
\circ
([(\id_{\cF(u)} \cF(\cF_{u\to v}))\circ\varepsilon^{\cT\dslash\cF}_{\cF(u) \to \cF(v)}] \id_{\langle v\to w\rangle})
\circ
[(\id_{\cF(v)} \cF(\cF_{v\to w}))\circ\varepsilon^{\cT\dslash\cF}_{\cF(v) \to \cF(w)}]
\\&=
(\id_{\cF(u)} \circ \nu_{\widehat{\cV}(u\to v), \widehat{\cV}(v\to w)})
\circ
([\nu^{-1}_{u,\widehat{\cV}(u\to v)}\cF(\varepsilon^{\widehat{\cV}}_{u\to v})] \id_{\langle v\to w\rangle})
\circ
[\nu^{-1}_{v,\widehat{\cV}(v\to w)}\cF(\varepsilon^{\widehat{\cV}}_{v\to w})]
\\&=
\nu^{-1}_{u, \widehat{\cV}(u\to v)\widehat{\cV}(v\to w)}
\circ
\cF(
(\varepsilon^{\widehat{\cV}}_{u\to v} \id_{\widehat{\cV}(v\to w)})
\circ
\varepsilon^{\widehat{\cV}}_{v\to w}
)
\\&=
\nu^{-1}_{u, \widehat{\cV}(u\to v)\widehat{\cV}(v\to w)}
\circ
\cF(
(\id_{u} (-\circ_{\widehat{\cV}}-))
\circ
\varepsilon^{\widehat{\cV}}_{u\to w}
)
\\&=
(\id_\cF(u) \cF(-\circ_{\widehat{\cV}}-))
\circ
\nu^{-1}_{u, \widehat{\cV}(u\to w)}
\circ
\cF(
\varepsilon^{\widehat{\cV}}_{u\to w}
)
\end{align*}
which is exactly the mate of $(-\circ_{\widehat{\cV}}-)\circ \cF_{u\to w}$.
Hence $\cF$ is a $\cV$-functor.

We already know the underlying functor $\cF^\cV: \cV \to \cT\dslash\cF^\cV$ is left adjoint to $\cR_{1_{\cT\dslash\cF}}: \cT \dslash\cF^\cV \to \cV$ by the proof of Theorem \ref{thm:ClosedOplaxVMonoidal} in \S\ref{sec:ClassificationOfClosedVMonoidalCategories}.
Thus by Theorem \ref{thm:LiftToVAdjunction}, to show $\cF$ is a left $\cV$-adjoint of $\cT\slash\cF(1_{\cT\slash\cF} \to -)$, it suffices to prove that $\cF$ is tensored.
Note that $\mu^\cF_{u,v}$ is by definition the mate of $\eta^{\widehat{\cV}}_{u,v}\circ \cF_{u\to uv}$ under the adjunction
$$
\cT\dslash\cF^\cV(\cF(u)\cF(v) \to \cF(uv))
\cong
\cV(v\to \cT\dslash\cF(\cF(u) \to \cF(uv))).
$$
But notice by definition of $\cF_{u\to uv}$, this mate is also given by
$$
(\id_{\cF(u)}\cF(\eta^{\widehat{\cV}}_{u,v}))
\circ
\nu^{-1}_{u, \widehat{\cV}(u\to uv)}
\circ
\cF(\varepsilon^{\widehat{\cV}}_{u\to uv})
=
\nu^{-1}_{u,v}
\circ
\cF(
(\id_u\eta^{\widehat{\cV}}_{u,v})
\circ
\varepsilon^{\widehat{\cV}}_{u\to uv}
)
=
\nu^{-1}_{u,v}.
$$
Hence $\mu^\cF_{u,v}=\nu^{-1}_{u,v}$ is invertible, and $\cF$ is tensored.
\end{proof}

\section{Completion for \texorpdfstring{$\cV$}{V}-monoidal categories}
\label{sec:CompletionForVMonoidalCategories}

We now discuss the completion operation for $\cV$-monoidal categories.
In this section, $\cV$ is braided and closed so we may form the self-enriched $\cV$-monoidal category $\widehat{\cV}$.

\begin{defn}
\label{defn:TensorProductMorphismForCompletion}
Given a $\cV$-monoidal category $\cC$, we define its completion $\overline{\cC}$ as an extension of Definition \ref{defn:CompletionOfVCategory}.
As before, $\overline{C}$ has objects of the form $a\blacktriangleleft u$ for $a\in \cC$ and $u\in \cV$.
We define $\overline{\cC}(a\blacktriangleleft u \to b\blacktriangleleft v)$, $j_{a\blacktriangleleft u}$, and $-\circ_{\overline{\cC}}-$ as before.
We additionally define:
\begin{itemize}
\item
$1_{\overline{\cC}} := 1_\cC \blacktriangleleft 1_\cV$.
\item
$(a\blacktriangleleft u)(b\blacktriangleleft v):=ab \blacktriangleleft uv$.
\item
$-\otimes_{\overline{\cC}}-$ is the mate of
$$
(\id_u \beta_{w, \widehat{\cV}(u\to\cC(a\to b)v)} \id_{\widehat{\cV}(w\to \cC(c\to d)x)})
\circ
(\varepsilon^{\widehat{\cV}}_{u \to \cC(a\to b)v} \varepsilon^{\widehat{\cV}}_{w \to \cC(c\to d)x})
\circ
(\id_{\cC(a\to b)} \beta_{v,\cC(c\to d)}^{-1} \id_x)
\circ
((-\otimes_\cC-)\id_{vx})
$$
under the adjunction
\begin{align*}
\cV(\widehat{\cV}(u\to \cC(a\to b)v)& \widehat{\cV}(w\to \cC(c\to d)x) \to \widehat{\cV}(uw\to \cC(ac\to bd)vx)
\\&\cong
\cV(uw \widehat{\cV}(u\to \cC(a\to b)v)\widehat{\cV}(w\to \cC(c\to d)x) \to \cC(ac\to bd)vx).
\end{align*}
\end{itemize}
It is a worthwhile exercise to verify that $\overline{\cC}$ satisfies the axioms of a $\cV$-monoidal category.
\end{defn}

Suppose $\cC$ is a $\cV$-monoidal category, and form $\overline{\cC}$ as above.
Note that $\overline{\cC}$ is tensored by Corollary \ref{cor:CBarTensored}, since this is merely a property of the underlying $\cV$-category of $\overline{\cC}$ (obtained by forgetting the $\cV$-monoidal structure).

\subsection{Universal property of completion for \texorpdfstring{$\cV$}{V}-monoidal categories}

\begin{defn}
We now endow our $\cV$-functor $\cI : \cC \to \overline{\cC}$ by $a\mapsto a\blacktriangleleft 1_\cV$
from Definition \ref{defn:InclusionVFunctor}
with the structure of a $\cV$-monoidal functor.
We define
$$
\nu^\cI_{a,b}
\in
\cV(1_\cV \to \overline{\cC}(ab\blacktriangleleft 1_\cV \to (a\blacktriangleleft 1_\cV)(b\blacktriangleleft 1_\cV)))
=
\cV(1_\cV \to \overline{\cC}(ab\blacktriangleleft 1_\cV \to ab\blacktriangleleft 1_\cV))
$$
to be $j_{ab\blacktriangleleft 1_\cV}$.
It is straightforward to check the necessary diagrams commute, and $(\cI, \nu)$ is $\cV$-monoidal.
\end{defn}

\begin{thm}
\label{thm:UniversalPropertyForVMonoidal}
Suppose $\cC$ and $\cD$ are $\cV$-monoidal categories with $\cD$ tensored 
and closed
and $(\cF,\nu^\cF): \cC \to \cD$ is a $\cV$-monoidal functor such that the underlying $\cV$-functor $\cF: \cC\to \cD$ is tensored.
There exists a tensored $\cV$-monoidal functor $(\overline{\cF}, \nu^{\overline{\cF}}): \overline{\cC} \to \cD$ such that $\cI \circ \overline{\cF}\cong \cF$ as $\cV$-monoidal functors.
\end{thm}
\begin{proof}
By Proposition \ref{prop:LiftToCBar}, we know that the underlying $\cV$-functor of $\cF$ (forgetting the $\cV$-monoidal structure) factors through a tensored $\cV$-functor $\overline{\cF}:\overline{\cC}\to \cD$, i.e., there is a $\cV$-natural isomorphism $\sigma: \cF\Rightarrow \cI \circ \overline{\cF}$.
It remains to show $\overline{\cF}$ can be endowed with the structure of a $\cV$-monoidal functor such that $\sigma$ is $\cV$-monoidal.

Since $\cD$ is tensored, 
under the bijective correspondence in Theorem \ref{thm:TensoredVMonoidalEquivalence}, there is a strong monoidal functor $(\cG^{\scriptscriptstyle Z},\nu^\cG): \cV \to Z(\cD^\cV)$, and we define $\cG := \cG^{\scriptscriptstyle Z} \circ R$ where $R: Z(\cD^\cV) \to \cD^\cV$ is the forgetful functor.
Moreover, note that the right $\cV$-module structure of $\cD^\cV$ is given by $d\vartriangleleft v = d\cG(v)$ with $\alpha_{d,u,v} = \id_d \nu^\cG_{u,v}$ similar to Remark \ref{rem:OtherMateOfMuForVMonoidal}.
We define $\nu^{\overline{\cF}}_{a\blacktriangleleft u,b\blacktriangleleft v}$ for $a\blacktriangleleft u,b\blacktriangleleft v\in\overline{\cC}$ 
to be equal to 
$(\nu^\cF_{a,b}\nu^\cG_{u,v}) \circ (\id_{\cF(a) e_{\cF(b), \cG(u)}}\id_{\cG(v)})$
in
$$
\cD^\cV( 
\overline{\cF}((a\blacktriangleleft u)(b\blacktriangleleft v)) 
\to 
\overline{\cF}(a\blacktriangleleft u)\overline{\cF}(b\blacktriangleleft v)
)
=
\cD^\cV( 
\cF(ab) \cG(uv) 
\to 
\cF(a)\cG(u)\cF(b)\cG(v)
).
$$
Here, $e_{\cF(b), \cG(u)}$ is the half-braiding for $\cG^{\scriptscriptstyle Z}(u)\in Z(\cD^\cV)$.
Associativity now follows immediately from associativity of $\nu^\cF$ and $\nu^\cG$, together with the hexagon axiom for the half-braidings $e$, and the fact that each $\nu^{\cG}$ is actually morphism in $Z(\cD^\cV)$.
(Indeed, this proof is similar to the displayed diagram in the proof of \cite[Prop.~7.4]{1607.06041}, where all tensorators are implicitly suppressed.)
Moreover, under the identification of the right $\cV$-module structure $d\vartriangleleft v := d \cG(v)$ as in Remark \ref{rem:OtherMateOfMuForVMonoidal}, we have that 
$$
\sigma_a 
:= 
\rho^\cD_a
\in 
\cV(1_\cV \to \cD(\cF(a) \to \cF(a)\vartriangleleft 1_\cV))
=
\cV(1_\cV \to \cD(\cF(a) \to \cF(a)))
$$ 
is identified with $j_{\cF(a)}$, and thus $\sigma$ is obviously monoidal.

Finally, to verify \eqref{eq:NaturalityForVMonoidal}, one shows the mates of the mophisms
$$
(\overline{\cF}_{a\blacktriangleleft u \to c\blacktriangleleft w}\overline{\cF}_{b\blacktriangleleft v\to d\blacktriangleleft x})
\circ
(-\otimes_{\cD}-)
\circ
(\nu_{a\blacktriangleleft u, b\blacktriangleleft v}^{\overline{\cF}}\circ_\cD-)
\qquad
\text{and}
\qquad
(-\otimes_{\overline{\cC}}-)
\circ
\overline{\cF}_{ab\blacktriangleleft uv \to cd\blacktriangleleft wx}
\circ
(-\circ_\cD\nu_{c\blacktriangleleft w, d\blacktriangleleft x}^{\overline{\cF}})
$$
are equal under the adjunction
\begin{align*}
\cV(
&\overline{\cC}(a\blacktriangleleft u \to c\blacktriangleleft w)
\overline{\cC}(b\blacktriangleleft v \to d\blacktriangleleft x)
\to
\cD(
\cF(ab)\vartriangleleft uv
\to 
(\cF(c)\vartriangleleft w)(\cF(d)\vartriangleleft x)
)
)
\\&=
\cV(
\widehat{\cV}(u \to \cC(a\to c)w)
\widehat{\cV}(v \to \cC(b\to d)x)
\to
\cD(\cF(ab)\cG(uv) \to \cF(c)\cG(w)\cF(d)\cG(x))
)
\\&\cong
\cD^\cV(
\cF(ab)
\cG(uv)
\cG(
\widehat{\cV}(u \to \cC(a\to c)w)
\widehat{\cV}(v \to \cC(b\to d)x)
)
\to
\cF(c)\cG(w)\cF(d)\cG(x)
),
\end{align*}
where one uses the definitions of $-\circ_\cD-$, $\overline{\cF}$, and $-\otimes_{\overline{\cC}}-$ as mates given in Section \ref{sec:V-mod to V-cat} (see also \cite[Prop.~4.9]{1701.00567}), Proposition \ref{prop:LiftToCBar}, and Definition \ref{defn:TensorProductMorphismForCompletion} respectively.
We leave this enjoyable exercise to the reader, who may wish to use the string diagrammatic calculus to perform this calculation.
We point out that one should keep in mind that $\cG$ applied to any morphism in $\cV$ is a morphism in $Z(\cD^\cV)$.
\end{proof}

\subsection{When tensored \texorpdfstring{$\cV$}{V}-monoidal categories are equivalent to their completions}

Since being tensored is a property of the underlying $\cV$-category of a $\cV$-monoidal category, we now adapt the results of \S\ref{sec:CompletionForVCats} to the $\cV$-monoidal setting by merely checking monoidality when necessary.

Recall from Lemma \ref{lem:MuIsANaturalTransformation} that $\tau_{a\blacktriangleleft u} : = \mu^\cI_{a,u}$ defines a $1_\cV$-graded $\cV$-natural transformation $\tau: \id^{\overline{\cC}} \Rightarrow  \overline{\id^\cC}\circ \cI$.

\begin{lem}
Suppose $\cC$ is closed.
The $1_\cV$-graded $\cV$-natural transformation $\tau: \id^{\overline{\cC}} \Rightarrow  \overline{\id^\cC}\circ \cI$ is monoidal.
\end{lem}
\begin{proof}
We must verify for $a\blacktriangleleft u,b\blacktriangleleft v\in\overline{\cC}$ that the composites 
$
\tau_{ab\blacktriangleleft uv}
\circ 
\nu^{\overline{\id^\cC}\circ \cI}_{a\blacktriangleleft u, b\blacktriangleleft v}
$ 
and  
$
\nu^{\id^{\overline{\cC}}}_{a\blacktriangleleft u,b\blacktriangleleft v}
\circ 
(\tau_{a\blacktriangleleft u}\tau_{b\blacktriangleleft v})
$
are equal in
\begin{align}
\overline{\cC}^\cV(
(a\blacktriangleleft u)(b\blacktriangleleft v) 
\to 
(a\vartriangleleft u\blacktriangleleft 1_\cV)(b\vartriangleleft v\blacktriangleleft 1_\cV)
)
&=
\overline{\cC}^\cV(
ab\blacktriangleleft uv
\to 
a\cF(u)b\cF(v)\blacktriangleleft 1_\cV
)
\notag
\\&=
\cV(uv \to \cC(ab \to a\cF(u)b\cF(v)))
\notag
\\&\cong
\cC^\cV(ab\cF(uv) \to a\cF(u)b\cF(v))
\label{eq:TauMonoidalAdjunction}
\end{align}
where $\cF: \cV \to Z(\cC^\cV)$ is the strong monoidal functor from Theorem \ref{thm:TensoredVMonoidalEquivalence},
and we have identified $\widehat{\cV}^\cV = \cV$ as in Example \ref{ex:UnderlyingCategoryOfVhat}.
Under the above Adjunction \eqref{eq:TauMonoidalAdjunction},
the mate of 
$\mu^\cI_{a,u} \mu^\cI_{b,v} = (\eta^\cC_{a,u} \eta^\cC_{b,v})\circ (-\otimes_\cC -)$
is given by
\begin{align*}
(\id_{ab}\cF(\eta^\cC_{a,u} \eta^\cC_{b,v}))
&\circ
(\id_{ab} \nu_{\cC(a\to a\cF(u)), \cC(b\to b\cF(v))})
\circ
(\id_a e_{b,\langle a\to a\cF(u)\rangle} \id_{\langle b\to b\cF(v)\rangle})
\circ
(\varepsilon^\cC_{a\to a\cF(u)} \varepsilon^\cC_{b\to b\cF(v)})
\\&=
(\id_{ab} \nu_{u,v})
\circ
(\id_{ab}\cF(\eta^\cC_{a,u})\cF(\eta^\cC_{b,v}))
\circ
(\id_a e_{b,\langle a\to a\cF(u)\rangle} \id_{\langle b\to b\cF(v)\rangle})
\circ
(\varepsilon^\cC_{a\to a\cF(u)} \varepsilon^\cC_{b\to b\cF(v)})
\\&=
(\id_{ab} \nu_{u,v})
\circ
(\id_a e_{b,\cF(u)} \id_{\cF(v)})
\circ
([
(\id_{a}\cF(\eta^\cC_{a,u}))
\circ
\varepsilon^\cC_{a\to a\cF(u)}
]
[
(\id_{b}\cF(\eta^\cC_{b,v}))
\circ
\varepsilon^\cC_{b\to b\cF(v)}
])
\\&=(\id_{ab} \nu_{u,v})\circ (\id_a e_{b,\cF(u)}\id_{\cF(v)}),
\end{align*}
which is exactly the mate of 
$
\mu_{ab, uv}^\cI
\circ 
\nu^{\overline{\id^\cC}\circ \cI}_{a\blacktriangleleft u, b\blacktriangleleft v}
$.
\end{proof}

\begin{thm}
\label{thm:EquivalentConditionsForVMonoidalCComplete}
The following are equivalent for a tensored closed $\cV$-monoidal category $\cC$.
\begin{enumerate}
\item
Every $\cV$-representable functor $\cR^a =\cC(a\to -): \cC \to \widehat{\cV}$ is tensored.
\item
The $\cV$-monoidal functor $\cI: \cC \to \overline{\cC}$ given by $a\mapsto a\blacktriangleleft 1_\cV$ is tensored.
\item
The $1_\cV$-graded monoidal $\cV$-natural transformation $\tau: \id^{\overline{\cC}} \Rightarrow \overline{\id^\cC} \circ \cI$ is an isomorphism.
\item
The $\cV$-monoidal functors $\cI: \cC \to \overline{\cC}$ 
and 
$\overline{\id^\cC}: \overline{\cC} \to \cC$
witness a $\cV$-monoidal equivalence.
\end{enumerate}
\end{thm}
\begin{proof}
We know the $\cV$-functors $(\cI,\nu^\cI): \cC \to \overline{\cC}$ and $(\overline{\id^\cC}, \nu^{\overline{\id^\cC}}): \overline{\cC} \to \cC$ are monoidal
and the $1_\cV$-graded $\cV$-natural isomorphism $\sigma: \overline{\id^\cC} \circ \cI\Rightarrow \id^{\overline{\cC}}$ and the $1_\cV$-graded $\cV$-natural transformation $\tau: \id^{\overline{\cC}} \Rightarrow \overline{\id^\cC} \circ \cI$ are monoidal.
Hence the result follows formally from Theorem \ref{thm:WhenTensoredIsEquivalentToCompletion}.
\end{proof}

Combining the above theorem with \S\ref{sec:VRigid}, when $\cV$ is rigid, we get the following corollaries.

\begin{cor}
Suppose $\cV$ is rigid and $\cC$ is a closed $\cV$-monoidal category.
Then $\overline{\overline{\cC}}$ is $\cV$-monoidally equivalent to $\overline{\cC}$.
\end{cor}

\begin{cor}
Suppose $\cV$ is rigid and $\cC$ is a tensored closed $\cV$-monoidal category.
Then $\cC$ is $\cV$-monoidally equivalent to $\overline{\cC}$.
\end{cor}

\subsection{Is the completion closed?}
\label{sec:CompletionClosedWhenVRigid}

In addition to Remark \ref{remark:OpenQuestionAboutCompletion}, we have the following interesting question.
If $\cC$ is a closed $\cV$-monoidal category, when is $\overline{\cC}$ closed?
Of course, $\overline{\cC}$ is closed under any of the hypotheses of Theorem \ref{thm:EquivalentConditionsForVMonoidalCComplete}, as $\overline{\cC}$ is $\cV$-monoidally equivalent to $\cC$.
But perhaps $\overline{\cC}$ is closed under some weaker assumptions, e.g., one could additionally assume $\cC(1_\cC \to -)$ admits either a left adjoint or a left $\cV$-adjoint.

The problem arises when trying to define the object $[a\blacktriangleleft u, c\blacktriangleleft w]$ in a way where we get an underlying adjunction
$$
\overline{\cC}^\cV((a\blacktriangleleft u)(b\blacktriangleleft v) \to c\blacktriangleleft w)
\cong
\overline{\cC}^\cV(b\blacktriangleleft v \to [a\blacktriangleleft u, c\blacktriangleleft w]).
$$
Under the identification $\widehat{\cV}^\cV = \cV$, the left hand side above is equal to
$\cV(uv \to \cC(ab\to c)w)$, and at this point it is not clear to us how to proceed unless $\cV$ is rigid.

\begin{example}
Building on Example \ref{ex:VRigid}, when $\cV$ is rigid, we can describe the $\cV$-monoidal structure of $\widehat{\cV}$ in more detail without taking mates.
The tensor product morphism is given by
$$
-\otimes_{\widehat{\cV}} - 
=
\beta^{-1}_{u^*w,v^*}\id_{x}
\in 
\cV(\widehat{\cV}(u\to w)\widehat{\cV}(v\to x) \to \widehat{\cV}(uv \to wx)) 
= 
\cV(u^*wv^*v \to v^*u^*wx).
$$ 

We can describe the $\cV$-monoidal structure of $\overline{\cC}$ in greater detail without taking mates.
Indeed,
the tensor product morphism is given by
\begin{align*}
-\otimes_{\overline{\cC}} - 
&=
(\beta^{-1}_{u^*\cC(a\to c)w,v^*}\id_{\cC(b\to d)x})
\circ
(\id_{v^*u^*\cC(a\to c)}\beta^{-1}_{w,\cC(b\to d)}\id_x)
\circ
(\id_{v^*u^*}(-\otimes_\cC-)\id_{wx})
\\&\hspace{1cm}\in
\cV(
\overline{\cC}(a\blacktriangleleft u \to c \blacktriangleleft w)
\overline{\cC}(b\blacktriangleleft v \to d \blacktriangleleft x)
\to
\overline{\cC}(ab\blacktriangleleft uv \to cd \blacktriangleleft wx)
)
\\&\hspace{1cm}=
\cV(u^*\cC(a\to c)wv^*\cC(b\to d)x \to v^*u^*\cC(ab\to cd) wx)
\end{align*}
\end{example}

\begin{defn}
Suppose $\cC$ is a closed $\cV$-monoidal category and $\cV$ is rigid.
By Lemma \ref{lem:RigidCharacterization}, every $\cV$-representable functor $\widehat{\cV}(u\to- ): \widehat{\cV} \to \widehat{\cV}$ is tensored.
In other words, $\widehat{\cV}(u\to vw) = u^*vw = \widehat{\cV}(u\to v)w$ for all $u,v,w\in \cV$.

Recall from Example \ref{ex:aOtimes- is V-monoidal} that 
\begin{align*}
(a\blacktriangleleft u\otimes -)_{b\blacktriangleleft v \to c\blacktriangleleft w}
&:=
(j_{a\blacktriangleleft u} \id_{\overline{\cC}(b\blacktriangleleft v\to c\blacktriangleleft w)})\circ (-\otimes_{\overline{\cC}}-)
\\&=
\id_{v^*}
[
(\coev_{u}\id_{\cC(b\to c)})
\circ
\beta^{-1}_{u^*, \cC(b\to c)}
\circ
(j_a\otimes_\cC -)
]
\id_w
\end{align*}
gives a well-defined $\cV$-functor $\overline{\cC} \to \overline{\cC}$.
It is easy to verify that setting $[a\blacktriangleleft u,b\blacktriangleleft v] := [a,b]\blacktriangleleft u^*v$
and
\begin{align*}
[a\blacktriangleleft u,-]_{b\blacktriangleleft v\to c\blacktriangleleft w}
&=
\id_{v^*}
[
(\id_{v^*} [a,-]_{b\to c} \id_w)
\circ
(\id_{v^*} \coev_{u^*} \id_{\cC([a,b] \to [a,c])} \id_{w})
\circ
(\id_{u^{**}} \beta^{-1}_{u^*, \cC([a,b] \to [a,c]})
]
\id_w
\\&\hspace{1cm}\in
\cV(v^*\cC(b\to c) w \to v^*u^{**}\cC([a,b]\to [a,c])u^*w)
\\&\hspace{1cm}=
\cV(\overline{\cC}(b\blacktriangleleft v\to c\blacktriangleleft w) \to \overline{\cC}([a,b]\blacktriangleleft \widehat{\cV}(u\to v) \to [a,c]\blacktriangleleft \widehat{\cV}(u\to w)))
\end{align*}
gives a well defined $\cV$-functor $[a\blacktriangleleft u,-]:\overline{\cC}\to\overline{\cC}$.
\end{defn}

\begin{prop}
\label{prop:ClosedImpliesCompletionClosed}
The $\cV$-functor $((a\blacktriangleleft u)\otimes -)$ is left $\cV$-adjoint to $[a\blacktriangleleft u,-]$ via the isomorphism
\begin{align*}
\overline{\theta}_{b\blacktriangleleft v , c\blacktriangleleft v}
&:=
\id_{v^*} 
[
\beta^{-1}_{u^*, \cC(ab\to c)}
\circ
(\id_{u^*} \theta_{b,c})
]
\id_w
\\&\hspace{1cm}\in 
\cV(v^*u^* \cC(ab\to c) w\to v^*\cC(b\to [a,c]u^*w),
\\&\hspace{1cm}=
\cV(
\overline{\cC}((a\blacktriangleleft u)(b\blacktriangleleft v) \to c\blacktriangleleft w)
\to
\overline{\cC}(b\blacktriangleleft v \to [a\blacktriangleleft u, c\blacktriangleleft w])
),
\end{align*}
where $\theta$ is the isomorphism for the $\cV$-adjunction $(a\otimes-)\dashv_\cV [a,-]$.
Thus $\overline{\cC}$ is closed.
\end{prop}
\begin{proof}
First, \eqref{eq:V-Adjoint1} holds if and only if 
$$
(((a\blacktriangleleft u)\otimes-)_{b\blacktriangleleft v \to d\blacktriangleleft x} \id_{\overline{\cC}(ad\blacktriangleleft ux\to c\blacktriangleleft w)}) 
\circ
(-\circ_{\overline{\cC}} -)
\circ
\overline{\theta}_{b\blacktriangleleft v,c\blacktriangleleft w}
=
(\id_{\overline{\cC}(b\blacktriangleleft v\to d\blacktriangleleft x)} \overline{\theta}_{d\blacktriangleleft x,c\blacktriangleleft w})
\circ
(-\circ_{\overline{\cC}}-).
$$
It is straightforward to check that the above equality follows from the fact that
$$
((j_a\otimes_\cC -) \id_{\cC(ad\to c)}) \circ (-\circ_\cC-) \circ \theta_{b,c}
=
(\id_{\cC(b\to d)} \theta_{d,c}) \circ (-\circ_\cC -),
$$
which is exactly \eqref{eq:V-Adjoint1} for $\theta$ for the  $\cV$-adjunction $(a\otimes -) \dashv_\cV [a,-]$.
Similarly, the fact that \eqref{eq:V-Adjoint2} holds for $\overline{\theta}$ follows immediately from 
$$
(-\circ_\cC -) \circ \theta_{b, d}
=
(\theta_{b,c} [a,-]_{c\to d}) \circ (-\circ_\cC-),
$$
which is exactly \eqref{eq:V-Adjoint2} for $\theta$ for the  $\cV$-adjunction $(a\otimes -) \dashv_\cV [a,-]$.
\end{proof}

\renewcommand*{\bibfont}{\small}
\setlength{\bibitemsep}{0pt}
\raggedright
\printbibliography

\end{document}